\documentclass[12pt,
a4paper]{amsart}
\usepackage[english]{babel}

\usepackage[utf8]{inputenc}
\usepackage{amsmath}
\usepackage{amscd}
\usepackage{amstext}
\usepackage{amsbsy}
\usepackage{amsopn}
\usepackage{amsthm}
\usepackage{amsxtra}
\usepackage{amsfonts}
\usepackage{amssymb}
\usepackage{mathrsfs}
\usepackage{euscript}
\usepackage{array}
\usepackage{verbatim}

\theoremstyle{plain}
      \newtheorem{theorem}{Theorem}
      \newtheorem{lemma}[theorem]{Lemma}
      \newtheorem{corollary}[theorem]{Corollary}
      \newtheorem{proposition}[theorem]{Proposition}
      
      \theoremstyle{definition}
      \newtheorem{definition}[theorem]{Definition}
      
      \theoremstyle{remark}
      \newtheorem{remark}[theorem]{Remark}

\newcounter{Step}
\newenvironment{step}[0]{\bigskip\addtocounter{Step}{1}\noindent\textbf{Step \theStep~:} }{\begin{flushright}\tiny \end{flushright}}
\newenvironment{claim}[0]{\bigskip\noindent\textbf{Claim~:} }{\
  \begin{flushright} \end{flushright}}

\def\o{\otimes}

\def\N{\mbox{I\hspace{-.15em}N} }
\def\R{\mbox{I\hspace{-.15em}R} }

\def\C{\hspace{.17em}\mbox{l\hspace{-.47em}C} }

\def\Fi{\mbox{I\hspace{-.15em}F} }

\def\H{\mathscr{H}}
\def\F{\mathscr{F}}
\def\cC{\mathscr{C}}
\def\I{\mathscr{J}}
\def\M{\mathscr{M}}
\def\B{\mathscr{B}}
\def\V{\mathscr{V}}
\def\CV{\mathscr{CL}}

\def\hM{\hat{M}}
\def\hMe{\hat{M}^{(\epsilon)}}
\def\Me{\hat{M}^{(\epsilon)}_0}
\def\oM{\hat{\overline{M}}_0}
\def\ohH{\hat{\overline{\H}}}
\def\oH{\overline{\H}}
\def\He{\H^{(\epsilon)}}
\def\hH{\hat{\H}^{(\epsilon)}}

\def\U{\mathscr{U}}
\def\E{\mathcal{E}}
\def\A{\mathcal{A}}
\def\H{\mathcal{H}}
\def\K{\mathcal{K}}
\def\TC{\mathcal{TC}}

\def\P{\mathcal{P}}

\newcommand{\Ad}{\operatorname{Ad}}
\newcommand{\op}{\operatorname{op}}
\newcommand{\Lim}{\operatorname{Lim}}

\begin{document}
\title[A non-commutative Path Space approach to stationary free SDE]{A non-commutative Path Space approach to stationary free Stochastic Differential Equations}
\begin{abstract} By defining tracial states on a non-commutative analogue of a path space, we construct Markov dilations for a class of conservative
completely Markov semigroups on finite von Neumann algebras. This class includes all symmetric semigroups. For well
chosen semigroups (for instance with generator any divergence form operator associated to a derivation valued in the
coarse correspondence) those dilations give rise to stationary solutions of certain free SDEs previously considered by D.
Shlyakhtenko. 
Among applications, we prove a non-commutative Talagrand
inequality for non-microstates free entropy (relative to a subalgebra $B$ and a completely positive map $\eta:B\to B$). We also use those new deformations in conjunction
with Popa’s deformation/rigidity techniques. For instance, combining our results with techniques
of Popa-Ozawa and Peterson, we prove that the von Neumann algebra of a countable discrete group with CMAP and positive first $L^2$ Betti number has no Cartan subalgebras.
\end{abstract}

\author[Y. Dabrowski]{Yoann Dabrowski}\address{ Department of Mathematics, University of California at Los Angeles\\ Los Angeles, California 90095\\ and \\  Laboratoire d'Informatique Institut Gaspard Monge Universit\'{e}
  Paris-Est,\\ 5 bd Descartes, Champs-sur-Marne F-77454 Marne-la-Valle cedex 2 FRANCE}
\email{yoann@math.ucla.edu}
\thanks{Research partially supported by NSF grants DMS-0555680 and DMS-0900776.}

\date{}
\maketitle

\begin{center}\section*{\textsc{Introduction}}
\end{center}

The purpose of this paper is to introduce a Path space approach to solve  stochastic differential equations (SDEs) driven by free Brownian notion, a large $N$ limit of Brownian motions on $N\times N$ hermitian matrices. Beyond applications to random matrices (\cite{BGC},\cite{SG}), free SDEs already appeared useful in getting lower bounds on microstates free entropy dimension $\delta_0(X_{1},...X_{n})$ (\cite{S07}), especially in the goal of comparing it with $L^{2}$-Betti numbers of groups(cf. e.g. \cite{L2}), specifically $\beta^{(2)}_1(\Gamma)-\beta^{(2)}_0(\Gamma)+1$. 
(see also \cite{MS05} for comparison with another entropy dimension). This study hugely motivates the search for stationary solutions of the following kind of free SDEs~:
\begin{equation}\label{coarseSDE}
X_t=X_0-\frac{1}{2}\int_0^t (\Delta(X))_s ds +\int_0^t(\delta(X))_s\# dS_s.
\end{equation}

Here, $\delta:L^2(M,\tau)\to (L^2(M,\tau)\o L^2(M,\tau))^N$ is a densely defined real closed derivation valued in a direct sum of coarse correspondences of a finite von Neumann algebra $(M,\tau)$, $\Delta=\delta^*\delta$ the associated divergence form operator, $S_t=(S_t^1,...,S_t^N)$ a free Brownian motion $\left(\sum_ia_i\o b_i\right)\#S=\sum_ia_iS^ib_i$ the usual isometry from the coarse correspondence enabling the definition of a stochastic integral.

Let us note that (1) is chosen to make possible stationary solutions, i.e. solutions such that the von Neumann algebra generated by  translates at time $t$, $W^*(X_t)$, does not depend on $t$ so that $\alpha_t(X_0)=X_t$ extends to a von Neumann algebra endomorphism from $M=W^*(X_0)$ to the $II_1$ $W^*$-algebra $\widetilde{M}=W^*(X_t,t\geq 0)$ where the equation is solved. Solving this equation thus gives a deformation (in the sense of Popa) of $M$ or a (trace preserving) dilation of the completely Markov semigroup $\phi_t=\exp(-t\Delta/2)$ since $\alpha_t$ also satisfies $E_M(\alpha_t(X))=\phi_t(X)$. As explained later, we do apply in this paper those dilations in conjunction with Popa's deformation/rigidity techniques, and actually dilate much more general completely Markov semigroups, solving a problem of \cite{AD04} also recently motivated by applications to non-commutative harmonic analysis (\cite{JLMX},\cite{JM}).

Up to now those equations have only been solved under rather restrictive conditions. 
 Let us give an overview in the easiest case.
In \cite{S07}, Shlyakhtenko obtained lower bounds on microstates free entropy dimension in studying the following free stochastic differential equation~:
\begin{equation}\label{conjSDE}X_{t}^{(i)}=X_{0}^{(i)}-\frac{1}{2}\int_{0}^{t}\xi_{s}^{(i)}ds + S_{t}^{(i)}\end{equation} where $(\xi_{s}^{1},...,\xi_s^n)$ are conjugate variables of $(X_{s}^{(1)},...,X_{s}^{(n)})$'s in the sense of \cite{Vo5}, $S_{t}^{(i)}$ a 
free Brownian motion free with respect to $X_{0}^{(i)}$. 
This is the previous equation with $\delta$ the free difference quotient of \cite{Vo5}.

In \cite{S07}, (\ref{conjSDE}) was solved in order to get stationary solutions under analyticity assumptions on $\xi$. 
Later, in \cite{Dab09}, the author solved a (dual stochastic partial differential equation) variant of (\ref{conjSDE}), extending its resolution to a Lipschitz condition on conjugate variables.  Again, in \cite{Dab09}, this construction enabled us to obtain lower bounds on microstates free entropy dimension $\delta_0(X_{1},...X_{n})$ (using the result of \cite{S07}, of course under an overall $R^{\omega}$ embeddability assumption). In
that respect, the crucial step needed to apply the approach of \cite{S07} is to find a stationary solution $X_t$ to (\ref{conjSDE}) or more generally (\ref{coarseSDE}) with $X_t\in W^{*}(M,\{S_t; t\geq 0\})\simeq M*L(\Fi_{\infty})$ the von Neumann algebra generated by initial conditions and free Brownian motion $\{S_t; t\geq 0\}$. By considering another example where the free difference quotient is replaced by derivations coming from $\ell^2$ group cocycles (see \cite{Dab09}), we realized that obtaining a stationary solution lying in this free product 
is equivalent to a conservativity condition of a (classical) Markov Chain, well known to be restrictive.

Thus we are forced with the question of solving those kinds of free SDEs beyond those conditions which is the main motivation of this paper. 
Before discussing more our approach, let us state a result to be compared with previous ones. In what follows, we will solve (\ref{conjSDE}) for $\xi\in L^1(M,\tau)$, and (\ref{coarseSDE}) for any densely defined real closed derivation as soon as $M$ has separable predual.
Since we have to build a solution without knowing a priori that it lives in the above free product, we will carry out a (kind of) Kolmogorov-Daniell construction to get dilations of a completely Markov semigroup on $M$. This entails solving various equations deduced from a formal application of Ito formula (and thus necessary for any dilation to solve (\ref{coarseSDE})). In contrast to Sauvageot's dilation in \cite{Sau86}, let us emphasize that our dilation satisfies rather canonical equations, since we want it to solve a free SDE in certain cases. From a non-commutative probability point of view, our approach, based on the construction of a state of a non-commutative analogue of a path space, is quite natural. It is in the spirit of the general philosophy that in the non-commutative context, probability can mainly be thought of ``in law". Technically, the usefulness of a path space was already apparent in the work of Biane, Capitaine and Guionnet \cite{BGC} in the context of large deviations for matricial Brownian motion. In our context of solving stochastic differential equations, its power of course comes from the ability it gives us to build a process without imposing any a priori structure on the von Neumann algebra that would contain it, thanks to the duality involved in the trace space of the non-commutative Path space (see $\S1$ for a definition of this universal $C^*$-algebra). Consequently, at the end, the von Neumann algebra generated by our process is not a priori explicit, being obtained by a GNS construction for the state we build on the Path space. 

In course of solving our SDEs, we will end up constructing, in Theorem \ref{main}, dilations for a general class of completely Markov semigroups $\phi_t$ on finite von Neumann algebras. As in \cite{CiS}, these are ultraweakly pointwise continuous semigroups of completely positive contractions.  The class we are able to dilate includes all conservative (i.e. $\tau\phi_t=\tau$) completely Markov symmetric (i.e. $\tau$-symmetric in the sense of \cite{CiS} $\tau(\phi_t(a)b)=\tau(a\phi_t(b))$) semigroups. Stated otherwise, getting a trace preserving dilation proves all those maps $\phi_t$ are factorizable in the sense of \cite{AD04}. This is to be contrasted with non-factorizable maps (and non-symmetric semigroups of such) found in \cite{HM} in finite dimension. 
However, our construction also applies to various non-symmetric conservative completely Markov semigroups, generated by generators of non-symmetric Dirichlet forms, for which the antisymmetric part is a derivation, assuming certain domain assumptions (cf. part 2.1). In \cite{JRS}, the authors found independently a dilation in the symmetric case, which seems to be the same as ours
, but they actually started with a natural description of our $\alpha$-approximation (in terms of another stochastic differential equation, or rather they use a semigroup like approximation $(id -\phi_t)/t$ instead of a resolvent like approximation $(\alpha(id-\eta_{\alpha}))$ as approximation of the generator $\Delta$ of the semigroup) and didn't study carefully the limit, since this is not crucial to reach their goal of getting a dilation. Starting with the goal of building a solution of a SDE, necessarily dilating special semigroups (we realized lately this may enable us to dilate really general semigroups), we first started with natural equations for the limit deducing the best approximation, which turned out to be the same as their approximation as we discovered just before publishing this paper.  To sum up, our careful understanding of the limit enables us to get dilations of a few non-symmetric semigroups, since in that case, the approximation is not a dilation in tracial $W^*$-probability spaces, and this requires a careful understanding of the limit to prove traciality only appearing at the limit level. Moreover, our understanding of the limit is crucial in the applications we were motivated by, as described bellow. However, the reader will maybe (like us) prefer their proof of positivity of the approximation, much more natural in \cite{JRS}, since coming from a canonical description of this approximation by a SDE, rather than tedious explicit computations (as in this paper). We hope the conjunction of both approaches will give rise to better applications coming from the two original motivations.

As a consequence of our dilation construction for a huge class of non-symmetric conservative completely Markov semigroups, we prove that our dilations often satisfy a Stochastic Differential Equation as above (see Theorem \ref{main2}). More precisely, our dilation satisfies a SDE driven by a B-free Brownian motion of covariance $\eta$, as soon as the generator of the semigroup is a divergence form operator $\delta^*\delta$ with $\delta$ a derivation (as it has to be because of \cite{CiS} but) valued in the canonical bimodule $H(M,\eta\circ E_B)$ for $\eta$ a completely positive map on $B$, where $B$ is a \textit{non-evolving} subalgebra (i.e. $\phi_t(b)=b\ \forall b\in B$, equivalently $\delta(B)=0$), and modulo certain technical assumptions (but including any derivation in the coarse correspondence case ($B=\C$), when $M$ has separable predual, as explained earlier).

Since applications to lower bounds on microstates free entropy dimension are out of reach via this method (except under conditions similar to those already obtained earlier by the author), we give two other applications. First, we use our dilations to remove a technical assumption in a result of Popa and Ozawa (\cite{ozawapopa},\cite{ozawapopaII}) giving a first application of this new ``stochastic" deformations (dilating deformations first considered in Deformation/rigidity context by Jesse Peterson (\cite{Pe04},\cite{Pe06}, \cite{P10}), and used also in 
\cite{ozawapopaII}). Moreover, we also prove in Theorem \ref{Tal} a free Talagrand transportation cost inequality for non-microstates free entropy $\chi^*(X_1,...,X_n:B,\eta)$ (relative to a subalgebra $B$ and a completely positive map $\eta:B\to B$, as defined in \cite{ShlyFreeAmalg00}, see also \cite{Vo5} in the case $B=\C$), extending the one variable result  (in the case $B=\C$) of \cite{BV01} (see also \cite{Ledoux}, \cite{LedouxPopescu} for new proofs) and (a priori) improving the multivariable result of Hiai and Ueda for microstates free entropy $\chi(X_1,...,X_n)$ (\cite{HiaiUeda}, the improvement is due to the inequality between $\chi$ and $\chi^*$ proven in \cite{BGC}. Note also that we have an extension only for estimates on Wasserstein distance to semicircular elements and not for the convex polynomial potential variant they also consider, but our inequality goes beyond the case $B=\C$). This second result thus belongs to a chapter of the general goal of proving for non-microstates free entropy more results known for the microstates variant. Beyond the usefulness of this study from a free probability viewpoint, this general goal generally provides merely functional analytic proofs of results relying heavily on volume estimates when proven with microstates entropy, enabling to enlarge our understanding of non-commutative probability rather than merely using our knowledge of classical probability to give applications to free probability or von Neumann algebras. For more details about free entropies, we refer the reader to the survey \cite{VoS} for a list of properties as well as applications of free entropies in the theory of von Neumann algebras.

The paper is organized as follows. In  Section 1, we explain in what sense we use a noncommutative Path space. In Section 2, we construct our dilations. This requires to prove, in a first subsection, a few preliminaries on``Carr\'{e}-du-champs" of (non-symmetric non-commutative) Dirichlet forms. Then we build our dilation modulo a positivity assumption on an approximation. A second subsection thus explains the limiting procedure enabling this. We then prove the lacking positivity of the state we want to build on our Path space. We finally prove a symmetry condition and deduce traciality from this.

In Section 3, we explain the application to Talagrand's inequality, proving in so doing that our dilation actually solves the above stochastic differential equation (when the derivation is the free difference quotient, under a finite free Fisher information assumption). Actually, we prove a SDE for subsystems with a derivation looking like free difference quotient, since this will be the key tool to solve in the next section a more general SDE.
We moreover prove infinitesimal estimates along solutions of more general polynomial drift SDEs, even though Talagrand's inequality will be based on the really special case of the Orstein-Uhlenbeck process. In so doing, we especially prove that such solutions of SDEs have bounded conjugate variable, a result of independent interest.

 In Section 4, we prove our dilation actually solves a SDE, in the case explained above. To do so, we have to produce the Brownian motion (not necessarily in the von Neumann algebra of the process) by coupling our equation to an Orstein-Uhlenbeck equation (modulo a drift change disappearing if we take the right conditional expectation, this equation is a special case of the previous section for semicircular variables). This is where we use dilations of semigroups coming from non-symmetric Dirichlet forms, since the coupled equation dilate such a non-symmetric semigroup.

Finally, in section 5, we give our small application in conjunction with Popa's deformation/Rigidity techniques, without claiming any originality in this respect 
beyond the introduction of a new deformation and the study of several of its properties. We should emphasize, however, that this way of building a deformation in trying to solve 
  an abstract equation having sufficiently good properties to obtain desirable properties of the deformation, is really new in Deformation/Rigidity Theory, where deformations are usually produced
from relatively concrete examples. We hope this will be the first step of a systematic use of non-commutative probabilistic ideas with the goal of building deformations with specific behaviors.

\bigskip\textbf{Acknowledgments}
The author would like to thank Professor Marius Junge for communicating him when this work was almost written a preliminary version of \cite{JRS}  and Professor Jesse Peterson for suggesting that a corollary like \ref{GeneDefRig} should be true at a time when the author only knew how to use an earlier version of section 4 (reduced to derivations valued in the coarse coming from cocycles of groups) to get Corollary \ref{SpecDefRig}, this motivated an attempt to understand how to use deformation/rigidity techniques with less knowledge on orthogonal complement of the $L^2$ space of the initial time algebra in the $L^2$ space produced by our deformation (even though our new section 4 hugely extends our understanding in that respect). The author would also
like to thank D. Shlyakhtenko, J. Peterson and P. Biane for plenty of useful
discussions.

\tableofcontents

\section{Non-commutative path space}
In classical probability, the path space $\prod_{t\in \R_{+}}X$ is a useful tool to define processes indexed by $\R_{+}$ with value in a (locally) compact space $X$ through the definition of a measure on it, nothing but a state on $\bigotimes_{t\in \R_{+}}C^{0}(X)$.
Analogously, we want to define a non-commutative path space in the following definition. We consider a $C^{*}$-algebra $C$ or a family $C_{t}$ of $C^{*}$-algebras (in the first case we consider $C_{t}=C$).

\begin{definition}
The algebraic path space indexed by a set $I$ over $C$ or $(C_{t})$, denoted $\mathcal{P}_{I,alg}(C)$ resp. $\mathcal{P}_{I,alg}(C_{t})$ (or $\mathcal{P}_{alg}(C)$ if $I$ is fixed, e.g. $\R_{+}$ (or sometimes $\R$) in this paper), is the algebraic free product $\bigstar_{t\in I}C_{t}$.
For distinct $t_{i}\in I,i=1,...,n$, we denote $C_{max,t_{1},...,t_{n}}=C_{t_{1}}\bigstar...\bigstar C_{t_{n}}$ the maximal free product (the one with universal property in the category of $C^{*}$-algebras) and then consider the $C^{*}$-algebraic path space $\mathcal{P}_{I,max}(C)$ ($\mathcal{P}_{I,max}(C_{t})$ etc.) the natural inductive limit for all finite set families in $I$.
\end{definition}

In this paper we will be mainly concerned with a $W^{*}$-probability space $(M,\tau)$ (i.e. $\tau$ a faithful tracial normal state)
and in building states on $\mathcal{P}_{alg}(M)$ or $\mathcal{P}_{max}(M)$ so that the state restricted to each M is $\tau$, so that we will get a stationary process, by definition.
Since $M$ is a $C^*$ algebra, it is readily seen that the algebraic Path space is the span of unitary elements in that case. A standard result (see e.g. Proposition 7.2 in the book of Nica and Speicher \cite{NS}) enables to carry out a GNS construction on any state over $\mathcal{P}_{alg}(M)$, thus especially extending automatically to a state over $\mathcal{P}_{max}(M)$ by a universal property. We will thus (almost) always work in this paper on algebraic Path space.

\section{Construction of a stationary process}
We refer to \cite{S07} (cf. also \cite{Dab09}) and the applications bellow for motivating the search of a (stationary) process satisfying $$X_{t}^{(i)}=X_{0}^{(i)}-\frac{1}{2}\int_{0}^{t}\xi_{s}^{(i)}ds + S_{t}^{(i)}$$ where $\xi_{s}^{i}$ is a i-th conjugate variable of $X_{s}^{(i)}$'s in the sense of \cite{Vo5}, $S_{t}^{(i)}$ a free Brownian motion free with respect to $X_{0}^{(i)}$.

Using (formally) Ito's formula (proven under some assumptions in \cite{BS98}), one expects for instance the following equation (say for non-commutative polynomials $P_{1},Q_{1},P_{2}$, $X_{t}=(X_{t}^{(1)},...,X_{t}^{(n)})$):

\begin{align*}\tau(P_{1}(X_{t})Q(X_{0})P_{2}(X_{t}))&=\tau(P_{1}(X_{0})Q(X_{0})P_{2}(X_{0}))\\ & -\frac{1}{2}\int_{0}^{t}ds\ 
\tau(P_{1}(X_{s})Q(X_{0})\Delta (P_{2}(X_{s})))+\tau(\Delta (P_{1}(X_{s}))Q(X_{0})P_{2}(X_{s}))  \\ &+ \int_{0}^{t}ds\ \tau\circ m \circ 1\otimes \tau\circ m \otimes 1 (\delta(P_{1})(X_{s}))\otimes Q(X_{0})\delta (P_{2})(X_{s}))\end{align*}

We wrote as usual $\delta=(\delta_{1},...,\delta_{n})$ the free difference quotient on $M=W^{*}(X_{1},...,X_{n})$ (same action on the $s$-time variables
)  $\Delta=\delta^{*}\delta$ the corresponding generator of a Dirichlet form and if we write $\phi_{t}$ the semigroup generated by $-\frac{1}{2}\Delta$, one expects a rewriting after ``variation of constants"~:
\begin{align*}\tau(P_{1}(X_{t})Q(X_{0})&P_{2}(X_{t}))=\tau(\phi_{t}(P_{1}(X_{0}))Q(X_{0})\phi_{t}(P_{2}(X_{0}))\\ &+ \int_{0}^{t}ds\ \tau\circ m \circ 1\otimes \tau\circ m \otimes 1 (\delta\circ\phi_{t-s}(P_{1})(X_{s}))\otimes Q(X_{0})\delta \circ\phi_{t-s}(P_{2})(X_{s}))\end{align*}

We could also have said we want to study a mild solution of the above equation in the spirit of PDE theory (see also \cite{Dab09} in the free Stochastic PDE case).

This is already less singular and make sense for $Q\in M$, $P_{i}\in M\cap D(\Delta^{\epsilon}),\epsilon>0$ (even $\epsilon =0$ as we will see later, but using the defining differential equation after Cauchy-Schwarz instead of an a priori bound for analytic semigroups). From now on, one can see $X_{s}$ as a formal variable meaning a variable at time $s$ in $\mathcal{P}_{alg}(M)$. Note that $\tau$ above make sense as a state on a two times free product (after suitable generalization to more alternating patterns of 0 and t times).

 The actual definition will use an $\alpha$-approximation natural from a Dirichlet form viewpoint. This will require Dirichlet form preliminaries to get the right convergences of these approximations. We will then define simultaneously $\alpha$-approximations and their limits and prove the right kind of limits enabling us to continue the construction by induction. In a first time all those limits will be proven under the assumption the maps we build have nice uniform (in the approximation $\alpha$) boundedness in $M$. We will then prove in a third part those assumption are indeed satisfied in proving a positivity property giving those boundedness automatically by a standard $C^*$-algebra argument. Finally, we will have to prove our formulas produce traces. This will require 
a symmetry property. Alternatively, we could say we have written our defining formula in using this symmetry in the right way to get an almost explicitly positive definition, and we have to use it again to get a more rotation invariant variant.

\subsection{Approximations of ``Carr\'{e} du Champs" of Dirichlet forms}

Let us fix some notations (close to those of \cite{Pe06}).
We consider $M$ a finite von Neumann algebra 
 with normal faithful tracial
state $\tau$, and $\H$ an $M-M$-bimodule. $D(\delta)$ a weakly dense
*-subalgebra of $M$. 
We suppose here that $\delta:D(\delta)\rightarrow\H$ is a real closable
derivation (real means $\langle \delta(x),y\delta(z) \rangle = \langle \delta(z^{*})y^{*},\delta(x^{*})\rangle$).
$\Delta=\delta^{*}\bar{\delta}$ the corresponding generator of a conservative completely
Dirichlet form, as proven in \cite{S3} (see this paper for the non-commutative
definition of a Dirichlet form, here the Dirichlet form is
$\mathcal{\tilde{E}}(x)=\langle\delta(x),\delta(x)\rangle,
D(\mathcal{\tilde{E}})=D(\Delta^{1/2})$, completely means that $\Delta\otimes I_{n}$
is also the generator of a Dirichlet form on $\mathrm{M}_{n}(M)$).
 Let us introduce a deformation of
resolvent maps (a multiple of a so-called strongly continuous contraction
resolvent, cf e.g. \cite{MaR} for the terminology) $\eta_{\alpha}=\alpha(\alpha + \Delta)^{-1}$, which are
 unital, tracial ($\tau\circ\eta_{\alpha}=\tau$), completely positive maps, and moreover
contractions on $L^{2}(M,\tau)$ and normal contractions on $M$, such that
$||x-\eta_{\alpha}(x)||\leq 2||x||$  and
$||x-\eta_{\alpha}(x)||_{2}\rightarrow_{\alpha\rightarrow \infty}
0$ (as recalled e.g.
in Prop 2.5 of \cite{CiS}). We will also consider $\varphi_{t}=e^{-t\Delta/2}$ the
semigroup of generator $-\Delta/2$.

We will also consider $\E$ a non-symmetric completely Dirichlet form as in \cite{NSNC} (we assume it conservative as above i.e. $\E(1,.)=\E(.,1)=0$). Especially $\E$ is a coercive closed form on $L^{2}(M,\tau)$ and we will assume its symmetric part is $\mathcal{\tilde{E}}$ above with $D(\mathcal{\tilde{E}})=D(\E)$ a domain making it closed as usual.
We have thus also given adjoint $G_{\alpha}=\alpha(\alpha+A)^{-1}, \hat{G}_{\alpha}$ families of resolvent maps \footnote{for the generator $A$ of $\E$ and its adjoint $A^{*}$, to fix ideas, we use notations of \cite{NSNC} $\E(x,y)=\langle x,Ay\rangle$ (except for resolvents), $\E^{\beta}(x,y)=\beta \langle x, y-G_{\beta}(y)\rangle=\langle x, AG_{\beta}(y)\rangle$
}  on $L^{2}(M,\tau)$ letting $M$ stable, completely positive, unital, tracial normal as seen on $M$. We also consider corresponding  semigroups $\phi_{t}=e^{-tA/2},\phi_{t}^{*}$ (strongly continuous on $L^{2}$, ultraweakly continuous normal on $M$, contractive on both and $L^{1}(M)$). Let us also consider $\phi_{t,\alpha}=e^{-tAG_{\alpha}/2}$ the usual Yosida approximation (also contractive on the same spaces, etc since $AG_{\alpha}$ also generates a completely Dirichlet form by the proof of the standard equivalence theorem with Markovianity of $G_{\alpha}$). We will also assume the non symmetric part of the generator is a derivation, we will express this later in terms of carr\'{e}-du-champs. For the reader's convenience 
let us quote the following result corresponding to proposition 1.5 in
\cite{NSNC} (or lemma 2.11 and Theorem 2.13 in \cite{MaR}). Recall that we have a constant $K$ of coercivity expressing $|\E_1(x,y)|\leq K\E_1(x,x)^{1/2}\E_1(y,y)^{1/2},x,y\in D(\E)_{sa}$, with $\E_1=\E+\langle .,.\rangle$.

\begin{lemma}\label{basic}
Let $\{\E, D(\E)\}$ be a coercive closed form on a Hilbert space $H$, and $\{G_\alpha \}_{\alpha>0}$ , the
 associated resolvent. Then, setting $\E^{(\beta)} (x, y)~:= \beta(x, y - G_\beta y)$, $x, y \in H$, we get  \begin{enumerate}\item[(i)] $|\E_1^{(\beta)} (x, y)| \leq 4(K + 1)\tilde{\E}_1 (x, x)^{1/2}\tilde{\E}_1^{(\beta)} (y, y)^{1/2}, x \in D(\E), y \in H$
\item[(i)']$\E (G_\beta(x), G_\beta(x))\leq \E^{(\beta)} (x, x)$, for $x$ self-adjoint.
 \item[(ii)] Let $x \in H$. Then $x \in D(\E) \Longleftrightarrow \sup_{\beta>0} \tilde{\E}^{(\beta)} (x, x) < \infty$
\item[(ii)'] Let $x \in H$. Then $x \in D(\E) \Longleftrightarrow \liminf_{\beta>0} \tilde{\E}_1 (G_{\beta}(x), G_{\beta}(x)) < \infty$
\item[(iii)] $\forall x, y \in D(\E)$, $\lim_{\beta\rightarrow\infty} \E^{(\beta)} (x, y) = \E(x, y)$ and  $\lim_{\beta\rightarrow\infty} \E_1 (G_{\beta}(x)-x, G_{\beta}(x)- x) = 0$.
\end{enumerate}\end{lemma}
Using the result of \cite{DL} that $M\cap D(\Delta^{1/2})$ is a *-subalgebra of $M$, dense in $L^{2}(M,\tau)$ and a core for $\Delta^{1/2}$, one may consider $B$ the $C^{*}-algebra$, norm closure of it in $M$, so that $\B=B\cap D(\Delta^{1/2})$ is a form core and dense in $B$ and thus $\delta$ seen as a derivation on $B$ is a $C^{*}$-Dirichlet form in the terminology of \cite{CiS}, $\B$ the corresponding Dirichlet algebra.

We assume given a $L^{1}(M,\tau)$-valued product on $\H$, compatible with the Hilbert bimodule structure, i.e defined so that $\tau(\langle \xi,\xi'\rangle_{L^{1}(M,\tau)}a)=\langle \xi,\xi'a\rangle_{\H}$ for any $a\in M$, well-defined using $\sigma$-weak continuity of the action. Thus we consider for $a,b\in\B$, $\Gamma(a,c,b)=\langle c^{*}\delta(a^{*}),\delta(b))\rangle_{L^{1}(M,\tau)}=\langle \delta(a^{*}),c\delta(b))\rangle_{L^{1}(M,\tau)}.$
By Th 9.3 and Lemma 9.1 in \cite{CiS} (using also Th 8.3) we get for $a,c\in\B, b,d\in B$: $$\tau(d^{*}\Gamma(c^{*},1,a)b)= \lim_{\alpha\rightarrow\infty}\frac{1}{2}\tau\left[d^{*}c^{*}\Delta\eta_{\alpha}(a)b+d^{*}\Delta\eta_{\alpha}(c^{*})ab- d^{*}\Delta\eta_{\alpha}(c^{*}a)b\right].$$

Since for $a,c,b\in\B$, $\Gamma(a,c,b)=\Gamma(ac,1,b)-a\Gamma(c,b)$, it is natural to write~:
$$\tilde{\Gamma}_{\alpha}(a,c,b)=\frac{1}{2}\left[\Delta\eta_{\alpha}(ac)b+a\Delta\eta_{\alpha}(cb)-\Delta\eta_{\alpha}(acb)-a\Delta\eta_{\alpha}(c)b\right].$$

so that $\Gamma_{\alpha}(a,c,b)$ converges weakly in $B^{*}$
 to $\Gamma(a,c,b)$ (We will see later norm convergence in $L^1$ in the case we are most interested in). Note moreover that $\Gamma_{\alpha}(a,c,b)\in M$ for $a,c,b\in M$.

Considering now the non-symmetric context, we get for any $a,b,c,d\in\B$ like in lemma 3.1 of \cite{CiS}, but using the standard proposition 1.5(iii) of \cite{NSNC} (lemma \ref{basic} (iii) above)~:
$$\E(cdb^{*},a)+\E(db^{*}a^{*},c^{*})-\E(db^{*},c^{*}a)= \lim_{\alpha\rightarrow\infty}\tau\left[d^{*}c^{*}AG_{\alpha}(a)b+d^{*}AG_{\alpha}(c^{*})ab- d^{*}AG_{\alpha}(c^{*}a)b\right].$$
The assumption of the non symmetric part being a derivation will be assumed stating for any $a,b,c,d\in\B$~:
\begin{equation}\label{derForm}\E(cdb^{*},a)+\E(db^{*}a^{*},c^{*})-\E(db^{*},c^{*}a)=\mathcal{\tilde{E}}(cdb^{*},a)+\mathcal{\tilde{E}}(db^{*}a^{*},c^{*})-\mathcal{\tilde{E}}(db^{*},c^{*}a),\end{equation} so that one gets (more generally for $b,d\in B$):
$$\tau(d^{*}\Gamma(c^{*},1,a)b)= \lim_{\alpha\rightarrow\infty}\frac{1}{2}\tau\left[d^{*}c^{*}AG_{\alpha}(a)b+d^{*}AG_{\alpha}(c^{*})ab- d^{*}AG_{\alpha}(c^{*}a)b\right], $$

and the dual 
$$\tau(d^{*}\Gamma(c^{*},1,a)b)= \lim_{\alpha\rightarrow\infty}\frac{1}{2}\tau\left[d^{*}c^{*}A^{*}\hat{G}_{\alpha}(a)b+d^{*}A^{*}\hat{G}_{\alpha}(c^{*})ab- d^{*}A^{*}\hat{G}_{\alpha}(c^{*}a)b\right].$$

It is natural to write~:
$$\Gamma_{\alpha}(a,c,b)=\frac{1}{2}\left[AG_{\alpha}(ac)b+a AG_{\alpha}(cb)-AG_{\alpha}(acb)-aAG_{\alpha}(c)b\right],$$

so that $\Gamma_{\alpha}(a,c,b)$ converges weakly in $B^{*}$
 to $\Gamma(a,c,b)$. Likewise we define  $\hat{\Gamma}_{\alpha}(a,c,b)$ for $\hat{G}_{\alpha}$. Let us note at this point that a general argument using Stinespring's theorem as in lemmas 3.1 and 3.5 in \cite{CiS} shows $(a_{k}\Gamma_{\alpha}(b_{i},c_{j}c_{j'}^{*},b_{i'}^{*})a_{k'}^{*})_{((i,j,k),(i',j',k'))}$ is a positive matrix and thus (using a duality argument to bound $L^{1}$-norm) one gets via Cauchy-Schwarz and a usual property of states of $C^{*}$ algebras:
$$||\Gamma_{\alpha}(a,cc^{*},b^{*})||_{1}\leq ||\Gamma_{\alpha}(a,cc^{*},a^{*})||_{1}^{1/2}||\Gamma_{\alpha}(b,cc^{*},b^{*})||_{1}^{1/2},$$
$$||\Gamma_{\alpha}(a,cc^{*},a^{*})||_{1}\leq ||c||^{2}||\Gamma_{\alpha}(a,1,a^{*})||_{1}=||c||^{2}\tilde{\E^{\alpha}}(a).$$

\begin{lemma}\label{l3}
Let $a,b\in M$, we have $$\phi_{t}(ab)=\phi_{t}(a)\phi_{t}(b)+\int_{0}^{t}ds\ \phi_{s}(\Gamma(\phi_{t-s}(a),1,\phi_{t-s}(b))$$ where the integral is understood as Bochner integral of a function in $L^{1}([0,t],L^{1}(M,\tau))$. We will call $f_{u,t}^{(2)}(a,1,b)=f_{0,t-u}^{(2)}(a,1,b)=\int_{u}^{t}ds\phi_{s-u}(\Gamma(\phi_{t-s}(a),1,\phi_{t-s}(b))$. We have the analogue for $\phi^{*}$, then we will write $h_{u,t}^{(2)}(a,1,b)=\int_{u}^{t}ds\phi^{*}_{s}(\Gamma(\phi^{*}_{t-s}(a),1,\phi^{*}_{t-s}(b))$.
\end{lemma}
\begin{proof}
Bochner measurability is easy
, by polarization, $b=a^{*}$ is enough, in which case $\Gamma(\phi_{t-s}(a),1,\phi_{t-s}(a^{*}))$ is positive (see Lemma 9.1 in \cite{CiS}) thus  the $L^{1}$ norm is $\tilde{\E}(\phi_{t-s}(a))$ known to be of integral $||a||_{2}^{2}-||\phi_{t}(a)||_{2}^{2}$.
 If we consider $\phi_{s}(\phi_{t-s}(a)\phi_{t-s}(b))$ we get a derivative~:
$$\frac{1}{2}\left(\phi_{s}(A\phi_{t-s}(a)\phi_{t-s}(b))+\phi_{s}(\phi_{t-s}(a)A\phi_{t-s}(b))-A\phi_{s}(\phi_{t-s}(a)\phi_{t-s}(b))\right).$$
We easily see using the limit result above, taking scalar product with $c\in M$ (using $\phi_{s}^{*}(c)\in \B$), this is $\phi_{s}(\Gamma(\phi_{t-s}(a),1,\phi_{t-s}(b)))$. Now the previous result and Lebesgue Theorem (applied after taking a scalar product
) gives the result.
\end{proof}

From this we deduce a useful expression for $\Gamma_{\alpha}$:

\begin{lemma}\label{l4}

We have $\Gamma_{\alpha}(.,1,.)=\Gamma_{\alpha}^{(0)}+\Gamma_{\alpha}^{(1)}$ with the following bounded operators (for instance from $M\hat{\o}M\rightarrow L^{1}(M)$ $\hat{\o}$ always denote projective tensor product):
$$\Gamma_{\alpha}^{(0)}(a,1,b):=\alpha^{2}\int_{0}^{\infty}dt \ e^{-\alpha t}(1-\phi_{2t})(a)(1-\phi_{2t})(b),\ \ \ \ \Gamma_{\alpha}^{(1)}:= G_{\alpha}\Gamma G_{\alpha}^{\o},$$
where $G_{\alpha}^{\o}$  is the ``resolvent" for $A^{\o}:=A\o 1+ 1\o A$, $G_{\alpha}$ understood as sending $M\hat{\o}M$ to $D(\delta)\hat{\o}D(\delta)$, $\Gamma$ extended from this space to $L^1(M)$.
\end{lemma}

\begin{proof}
Start the proof for $a,b\in M$
This is mainly the use of the Laplace transform of (e.g.) Prop 1.10 in \cite{MaR} with our notations~:
 $G_{\alpha}(a)=\alpha\int_{0}^{\infty}dt\ e^{-t\alpha}\phi_{2t}(a).$
(e,g. with $a\in L^{2}(M)$, or $M$), moreover~:
$\Gamma_{\alpha}(a,1,b)=\alpha(1-G_{\alpha})(a)b+a\alpha(1-G_{\alpha})(b)-\alpha(1-G_{\alpha})(ab),$
and we immediately deduce~:
$$\Gamma_{\alpha}(a,1,b)=\Gamma_{\alpha}^{(0)}(a,1,b)+\alpha^{2}\int_{0}^{\infty}dte^{-t\alpha}(\phi_{2t}(ab)-\phi_{2t}(a)\phi_{2t}(b)).$$

From the previous lemma, one deduces~:
\begin{align*}\int_{0}^{\infty}dt\ e^{-t\alpha}(\phi_{2t}(ab)-\phi_{2t}(a)\phi_{2t}(b))&=\int_{0}^{\infty}dt\ e^{-t\alpha}\int_{0}^{t}ds\phi_{2s}(\Gamma(\phi_{2(t-s)}(a),1,\phi_{2(t-s)}(b))\\ &=\int_{0}^{\infty}dse^{-\alpha s}\phi_{2s}(\Gamma(\int_{0}^{\infty}due^{-\alpha u}\phi_{2u}(a)\o 1\o \phi_{2u}(b)).\end{align*}
The use of Fubini Theorem is justified by the case $a=b^{*}$ under $\tau(c^{*}c.), c\in M$.
\end{proof}
\begin{lemma}\label{l5}For any $a,b\in D(A)\cap M$, $||(\Gamma_{\alpha}-\Gamma)(a\o b)||_{1}\rightarrow 0,$ when $\alpha\to\infty.$
\end{lemma}
\begin{proof}
We can first use Cauchy-Schwarz~:
$$||\Gamma_{\alpha}^{(0)}(a\o b)||_{1}\leq \left(\int_{0}^{\infty}dt\ \alpha^{2} e^{-\alpha t} ||1-\phi_{2t}(a)||_{2}^{2}\right)^{1/2}\left(\int_{0}^{\infty}dt\ \alpha^{2} e^{-\alpha t} ||1-\phi_{2t}(b)||_{2}^{2}\right)^{1/2}.$$
For $a\in D(A)$, we have $||1-\phi_{t}(a)||_{2}\leq t ||A(a)||_{2}$ implying 
\begin{align*}\int_{0}^{\infty}&dt\ \alpha^{2} e^{-\alpha t} ||1-\phi_{2t}(a)||_{2}^{2}\leq \int_{0}^{\infty}dt\ \alpha^{2} e^{-\alpha t}4t^{2}||A(a)||_{2}^{2}\end{align*}
Since the measures $dt1_{|0,\infty]}\alpha^{2} e^{-\alpha t}t$ converges weakly to a Dirac in $0$ 
 the last term goes to $0$.

For $a,b\in D(A)$, self-adjoints (so that $||\Delta^{1/2}a||_{2}^{2}=\E(a)$).
\begin{align*}&||\Gamma_{\alpha}^{(1)}- \Gamma(a\o b)||_{1} \leq||\Gamma G_{\alpha}^{\o}-\Gamma(a\o b)||_{1}+ ||(G_{\alpha}-1)\Gamma(a\o b)||_{1}\\ &\leq ||(G_{\alpha}-1)\Gamma(a\o b)||_{1}+\\&\int_{0}^{\infty}dt\ \alpha e^{-\alpha t}(||\Delta^{1/2}(\phi_{2t}-id)(a)||_{2}||\Delta^{1/2}\phi_{2t}(b)||_{2}+||\Delta^{1/2}(\phi_{2t}-id)(b)||_{2}||\Delta^{1/2}(a)||_{2})\\ &\leq ||(G_{\alpha}-1)\Gamma(a\o b)||_{1}+\int_{0}^{\infty}dt\ \alpha e^{-\alpha t}(||(\phi_{2t}-id)A(a)||_{2}^{1/2}||A(b)||_{2}^{1/2}||b||_{2}^{1/2}||(\phi_{2t}-id)(a)||_{2}^{1/2}+\\ &||(\phi_{2t}-id)A(b)||_{2}^{1/2}||A(a)||_{2}^{1/2}||a||_{2}^{1/2}||(\phi_{2t}-id)(b)||_{2}^{1/2}).\end{align*}
and this converges to 0 since the measures $dt1_{|0,\infty]}\alpha e^{-\alpha t}$ converges weakly to a Dirac in $0$.
\end{proof}
We know need an improvement of the previous lemma. But before, let us knight a computation to become a lemma, in the spirit of the tricks we will use a lot in the next part.
\begin{lemma}\label{l6}For any $a,b\in M$
$$||f_{u,t}^{(2)}(a,1,b)||_{2}^{2}=\int_{u}^{t}2\Re \tau\left( (\phi_{s-u}f_{s,t}^{(2)}(a,1,b))^{*}\phi_{s-u}(\Gamma(\phi_{t-s}(a),\phi_{t-s}(b)))\right) $$
$$\int_{u}^{t}ds||\Delta^{1/2}(f_{s,t}^{(2)}(a,1,b))||_{2}^{2}=-||f_{u,t}((a,1,b)||_{2}^{2}+\int_{0}^{t}2\Re \tau\left( (f_{s,t}^{(2)}(a,1,b))^{*}(\Gamma(\phi_{t-s}(a), \phi_{t-s}(b)))\right).$$
\end{lemma}
\begin{proof}
Without loss of generality, $a,b$ self-adjoints in $M$. 
Let us define the $\alpha$ variant~:
$$f_{u,t}^{(2,\alpha)}(a,1,b)=\int_{u}^{t}ds\phi_{s-u,\alpha}(\Gamma_{\alpha}(\phi_{t-s,\alpha}(a),1,\phi_{t-s,\alpha}(b))=\phi_{t-u,\alpha}(ab)-\phi_{t-u,\alpha}(a)\phi_{t-u,\alpha}(b)$$
 We prove first $\sigma$-weak convergence in $M$ of this quantity to $f_{u,t}^{(2)}(a,1,b)$. Of course by boundedness in $M$, it suffices to prove convergence in $L^{1}(M)$. Now the previous lemma proves $(\Gamma-\Gamma^{\alpha})(\phi_{t-s}(a),1,\phi_{t-s}(b))$ converges in $L^{1}([u,t],L^{1}(M))$ to $0$ since $\phi_{t-s}(a)\in D(A)$ gives pointwise convergence and a dominated convergence theorem (DCT) concludes via domination by $(1+16(K+1)^{2})\tilde{\E}_{1}(\phi_{t-s}(a))^{1/2}\tilde{\E}_{1}(\phi_{t-s}(b))^{1/2}$ (cf. lemma \ref{basic} (i)). Moreover $||\Gamma^{\alpha}(\phi_{t-s}(a)-\phi_{t-s,\alpha}(a),\phi_{t-s,\alpha}(b))||_{1}\leq \tilde{\E}_{1}^{\alpha}(\phi_{t-s}(a)-\phi_{t-s,\alpha}(a))^{1/2}\tilde{\E}_{1}^{\alpha}(\phi_{t-s,\alpha}(b))^{1/2}$
Let us show the integral of this indeed goes to $0$. For, note $$\int_{u}^{t}ds\ \tilde{\E}_{1}^{\alpha}(\phi_{t-s}(a))\rightarrow \int_{u}^{t}ds\ \tilde{\E}_{1}(\phi_{t-s}(a)),$$ by pointwise convergence and DCT.
 Likewise, $$\int_{u}^{t}ds\ \tilde{\E}_{1}^{\alpha}(\phi_{t-s,\alpha}(a))=||\phi_{u,\alpha}(a)||_{2}^{2}-||\phi_{t,\alpha}(a)||_{2}^{2}\rightarrow \int_{0}^{t}ds\ \tilde{\E}_{1}(\phi_{t-s}(a))=||\phi_{u}(a)||_{2}^{2}-||\phi_{t}(a)||_{2}^{2}.$$
Finally for any $\gamma$,
\begin{align*}\int_{u}^{t}ds\ |\E_{1}^{\alpha}(\hat{G_{\gamma}}\phi_{t-s}(a),\phi_{t-s,\alpha}(a)-\phi_{t-s}(a))|&\leq \int_{0}^{t}ds\ ||A^{*}\hat{G_{\gamma}}\phi_{t-s}(a)||_{2} ||\phi_{t-s,\alpha}(a)-\phi_{t-s}(a))||_{2}\\ &\leq 2\gamma ||a||_{2}\int_{0}^{t}||\phi_{t-s,\alpha}(a)-\phi_{t-s}(a))||_{2}ds ,\end{align*}
also goes to $0$, and (lemma \ref{basic} (i) again)
\begin{align*}\int_{u}^{t}ds\ |\E_{1}^{\alpha}&((\hat{G_{\gamma}}-id)\phi_{t-s}(a),\phi_{t-s,\alpha}(a)-\phi_{t-s}(a))|
\\ &\leq 4(K+1)\left(\int_{u}^{t}ds\ \tilde{\E}_{1}((\hat{G_{\gamma}}-id)\phi_{t-s}(a))\right)^{1/2}\left(\int_{u}^{t}ds\ \tilde{\E_{1}}^{\alpha}(\phi_{t-s,\alpha}(a)-\phi_{t-s}(a)))\right)^{1/2},\end{align*}
the second integral is bounded independently of $\alpha$ and the first goes to zero in $\gamma\rightarrow\infty$ by DCT. The symmetric case $\int_{u}^{t}ds |\E_{1}^{\alpha}(\phi_{t-s,\alpha}(a)-\phi_{t-s}(a),\phi_{t-s}(a))|\rightarrow 0$ is easier.

Putting everything together, and using Cauchy-Schwarz, one gets the claimed convergence.

We now come back to a formula for the $L^{2}$-norm~: $$||f_{u,t}^{(2,\alpha)}(a,1,b)||_{2}^{2}=\int_{u}^{t}2\Re \tau\left( (\phi_{s,\alpha}f_{s,t}^{(2,\alpha)}(a,1,b))^{*}\phi_{s,\alpha}(\Gamma_{\alpha}(\phi_{t-s,\alpha}(a),\phi_{t-s,\alpha}(b)))\right) $$

the Fubini Theorem used being justified since $\Gamma_{\alpha}$ valued in $L^{2}$ (even in $M$). Using boundedness in $M$ of $f_{s,t}^{(2,\alpha)}$, we get as above convergence of the second term to the corresponding term without $\alpha$ (using at the end the weak convergence first proven), and thus, this proves the first formula. Or Rather, Using a variant with $\alpha,\beta$ scalar product and previously proven weak convergence in $M$ one gets the limit is actually $||f_{s,t}^{(2)}||_{2}$ and as a consequence norm $\|.\|_2$ convergence.  Similarly, we get formulas (idem without $\tilde{\E}_{1}^{\alpha}$ and $AG_{\alpha}$, or without $G_{\beta}$)~: \begin{align*}\tilde{\E}_{1}^{\alpha}(G_{\beta}(f_{u,t}^{(2,\alpha)}))=\int_{s}^{t}du&\Re\tau((1+AG_{\alpha})G_{\beta}(\phi_{u-s}^{\alpha}(f_{u,t}^{(2,\alpha)*})) G_{\beta}\phi_{u-s}^{\alpha}\Gamma_{\alpha}(\phi_{t-s,\alpha}(a),\phi_{t-s,\alpha}(b)))\\&+\Re\tau(G_{\beta}(\phi_{u-s}^{\alpha}(f_{u,t}^{(2,\alpha)*})) (1+AG_{\alpha})G_{\beta}\phi_{u-s}^{\alpha}\Gamma_{\alpha}(\phi_{t-s,\alpha}(a),\phi_{t-s,\alpha}(b))).\end{align*}

Now note that the derivative in $u$ of $||G_{\beta}f_{u,t}^{(2,\alpha)}||_{2}^{2}$ is~:

$\tilde{\E}_{1}^{\alpha}(G_{\beta}(f_{u,t}^{(2,\alpha)}))-||G_{\beta}f_{u,t}^{(2,\alpha)}||_{2}^{2}-2\Re\tau(G_{\beta}(f_{u,t}^{(2,\alpha)*}) G_{\beta}\Gamma_{\alpha}(\phi_{t-u,\alpha}(a),\phi_{t-u,\alpha}(b)))$

so that we get~:\begin{align*}\int_{s}^{t}du\ \tilde{\E}_{1}^{\alpha}(G_{\beta}(f_{u,t}^{(2,\alpha)}))&=-||G_{\beta}f_{s,t}^{(2,\alpha)}||_{2}^{2}+\int_{s}^{t}du\ ||G_{\beta}f_{u,t}^{(2,\alpha)}||_{2}^{2}\\ &+\int_{s}^{t}du\ 2\Re\tau(G_{\beta}(f_{u,t}^{(2,\alpha)*}) G_{\beta}\Gamma_{\alpha}(\phi_{t-u,\alpha}(a),\phi_{t-u,\alpha}(b))).\end{align*}

Taking $\alpha\rightarrow\infty$ (we keep equality at this stage) and then $\beta\rightarrow\infty$ (using for this second limit first Fatou's lemma to get $\int_{s}^{t}du\liminf_{\beta}\tilde{\E}_{1}(G_{\beta}(f_{u,t}^{(2)})))<\infty$ so that by lemma \ref{basic} (ii)' 
 one gets almost  surely $f_{u,t}^{(2)}\in D(\E)$ and thus by (iii) the liminf is actually a lim equal to $\tilde{\E}_{1}((f_{u,t}^{(2)})))$, one can then get equality applying DCT, with domination (via (i),(i)' of lemma \ref{basic}) by  $16(K+1)^{2}$ times the limit now already known to be integrable), one concludes the proof of the second formula.  If we first take $\beta\rightarrow\infty$, and then $\alpha\rightarrow\infty$, one thus deduces $$\int_{s}^{t}du\ \tilde{\E}_{1}^{\alpha}(f_{u,t}^{(2,\alpha)})\rightarrow \int_{s}^{t}du\ \tilde{\E}_{1}(f_{u,t}^{(2)}).$$
Finally (we give a stronger convergence for further use using an argument used earlier in this proof) for any $\gamma$,
\begin{align*}\int_{u}^{t}ds &\ |\E_{1}^{\alpha}(\hat{G_{\gamma}}f_{u,t}^{(2)}(a,1,b),f_{u,t}^{(2,\alpha)}(a,1,b)-f_{u,t}^{(2)}(a,1,b))|\\ &\leq \int_{0}^{t}ds\  ||A^{*}\hat{G_{\gamma}}f_{u,t}^{(2)}(a,1,b)||_{2} ||f_{u,t}^{(2,\alpha)}(a,1,b)-f_{u,t}^{(2)}(a,1,b))||_{2}\\ &\leq 2\gamma ||f_{u,t}^{(2)}(a,1,b)||_{2}\int_{0}^{t}||f_{u,t}^{(2,\alpha)}(a,1,b)-f_{u,t}^{(2)}(a,1,b))||_{2}ds ,\end{align*}
also goes to $0$, and 
\begin{align*}&\int_{u}^{t}ds\  |\E_{1}^{\alpha}((\hat{G_{\gamma}}-id)f_{u,t}^{(2)}(a,1,b)),f_{u,t}^{(2,\alpha)}(a,1,b)-f_{u,t}^{(2)}(a,1,b)))|
\\ &\leq 4(K+1)\left(\int_{u}^{t}ds\ \tilde{\E}_{1}((\hat{G_{\gamma}}-id)f_{u,t}^{(2)}(a,1,b))\right)^{1/2}\left(\int_{u}^{t}ds\ \tilde{\E}_{1}^{\alpha}(f_{u,t}^{(2,\alpha)}(a,1,b)-f_{u,t}^{(2)}(a,1,b)))\right)^{1/2},\end{align*}
the second integral is bounded independently of $\alpha$ and the first goes to zero in $\gamma\rightarrow\infty$ by DCT. The symmetric case $\int_{u}^{t}ds |\E_{1}^{\alpha}(f_{u,t}^{(2,\alpha)}(a,1,b)-f_{u,t}^{(2)}(a,1,b)),f_{s,t}^{(2)}(a,1,b))))|\rightarrow 0$ is easier.
Summing up we get~:
$$\int_{s}^{t}du\ \tilde{\E}_{1}^{\alpha}(f_{u,t}^{(2,\alpha)}(a,1,b)-f_{u,t}^{(2)}(a,1,b)))\rightarrow 0.$$
\end{proof}

\begin{lemma}\label{GammaLim}
If $\xi_{\alpha}=f\o g^{\alpha}\o h$ with $f,h\in D(\Delta)\cap D(A)\cap M, g^{\alpha}\in M$ uniformly bounded and convergent to $g^{\infty}$ in $L^{2}(M)$, then 
$$||\Gamma_{\alpha}(\xi_{\alpha})-\Gamma(\xi_{\infty})||_{1}\rightarrow 0.$$
Especially, if $D(\Delta)\cap D(A)\cap M$ is a core for $D(\Delta^{1/2})$ (e.g. $A=\Delta$) then this is true for any $f,h\in D(\Delta^{1/2})$.
\end{lemma}
Of course the whole statement is in the (almost) absence of assumption on $g$, at least in terms of domain of $\Delta$, we don't want to assume even $g\in B$, in which case the result would be rather easy. The crucial point will be to approximate $g$ only in $||.||_{2}$ with elements in $D(A)$ but this only on the problematic term. The main part of the proof will deal with estimates for all other terms.

\begin{proof}
Let us recall 
$\Gamma_{\alpha}^{(i)}(f\o g\o h)=\Gamma_{\alpha}^{(i)}(fg\o h)-f\Gamma_{\alpha}^{(i)}(g\o h)$.

Without loss of generality, $\xi= f\o g^{\alpha} \o h$, $f,g_{\alpha},h$ self-adjoints in $M$, $f,h\in D(A)\cap D(\Delta)$. Indeed via $||\Gamma_{\alpha}(f\o g \o h)||_{1}\leq C ||g|| \tilde{\E}^{\alpha}_{1}(f)^{1/2} \tilde{\E}^{\alpha}_{1}(h)^{1/2}\leq 16 C(K+1)^{2}||g||\tilde{\E}_{1}(f)^{1/2}\tilde{\E}_{1}(h)^{1/2}$, the second statement follows from the first. We will write $g$ for a generic element $g^{\alpha}$, $||g||$ their common uniform bound etc.
First, recall $||1-\phi_{t}(a)||_{2}\leq t||A(a)||_{2}$.
Second let us note that $||\phi_{t}(fg)-f\phi_{t}(g)||_{2}\leq ||\phi_{t}(fg)-\phi_{t}(f)\phi_{t}(g)||_{2}+||1-\phi_{t}(f)||_{2}||g||$.

Using the previous lemmas, we have to bound $||f_{0,t}^{(2)}(f,1,g)||_{2}=||\phi_{t}(fg)-\phi_{t}(f)\phi_{t}(g)||_{2}$. We first give a preliminary estimate, to be used later in a better one.
\begin{align*}||&f_{u,t}^{(2)}(f,1,g)||_{2}^{2}+\int_{u}^{t}ds\ ||\Delta^{1/2}(f_{s,t}^{(2)}(a,1,b))||_{2}^{2}\leq 2\int_{u}^{t}ds\ |\tau\left( (f_{s,t}^{(2)}(f,1,g))^{*}\Gamma(\phi_{t-s}(f), \phi_{t-s}(g))\right)|\\ &\leq 2\int_{u}^{t}ds\ ||(f_{s,t}^{(2)}(f,1,g))|| ||\Delta^{1/2}\phi_{t-s}(f)||_{2} ||\Delta^{1/2}\phi_{t-s}(g)||_{2}\\&\leq 4||f||||g||^{2}\left(\int_{u}^{t}ds\ ||\phi_{t-s}(Af)||_{2}||f||_{2}\right)^{1/2} \\&\leq 4||f||^{3/2}||g||^{2} ||Af||_{2}^{1/2}(t-u)^{1/2}\leq c'(f,g)(t-u)^{1/2}
\end{align*}



Let us now improve this first estimate

Let us fix an $\epsilon >0$
Choose $\delta$, $2||(G_{\delta}-id)Af||_{2}^{1/2}\leq \epsilon ||f||_{2}^{1/2}$.
Then choose a $y_{\epsilon}\in M$ such that~:
$||\Gamma(f^{*},f)-y_{\epsilon}||_{1}<\epsilon ||f||_{2}^{2}.$
Next take $\gamma$ (and $\alpha_{0}$) such that (for $\alpha>\alpha_{0}$, and consider from now on only those $g=g^{\alpha}$, recall we call $||g||$ the supremum of norms of $g_\alpha$) $||(1-G_{\gamma})(g)||_{2}^{2}||y_{\epsilon}||\leq \epsilon ||g||^{2}||f||_{2}^{2}$
Finally choose an $x_{\epsilon}\in M$ with $||\Gamma(f, f^{*})-x_{\epsilon}||_{1}^{1/2}||AG_{\gamma}(g)||_{2}^{1/2}<\epsilon||g||_{2}^{1/2}||f||_{2}$.

Now let us compute.
\begin{align*}&(\ast)=||f_{u,t}^{(2)}(f,1,g)||_{2}^{2}+\int_{u}^{t}ds\ ||\Delta^{1/2}(f_{s,t}^{(2)}(a,1,b))||_{2}^{2}\\ &\leq 2\int_{u}^{t}ds\ |\tau\left( (f_{s,t}^{(2)}(f,1,g))^{*}\Gamma(\phi_{t-s}(f), \phi_{t-s}(g))\right)|\\ &\leq 2\int_{u}^{t}ds\ |\tau\left( (f_{s,t}^{(2)}(f,1,g))^{*}\Gamma(f, \phi_{t-s}(G_{\gamma}(g)))\right)|\\ &+2\int_{u}^{t}ds\ |\tau\left( (f_{s,t}^{(2)}(f,1,g))^{*}\Gamma(f, \phi_{t-s}((1-G_{\gamma})(g)))\right)|\\ &+4||f||||g||\int_{u}^{t}ds\ ||(\phi_{t-s}-id)(Af)||_{2}^{1/2}||\phi_{t-s}(f)-f||_{2}^{1/2} ||\Delta^{1/2}\phi_{t-s}(g)||_{2}
\\ &\leq 2\int_{u}^{t}ds\ \tau\left( (f_{s,t}^{(2)}(f,1,g))^{*}\Gamma(f, f^{*})(f_{s,t}^{(2)}(f,1,g)))\right)^{1/2}\times \\ &\times \tau\left(\Gamma(\phi_{t-s}(G_{\gamma}(g)),\phi_{t-s}(G_{\gamma}(g)))\right)^{1/2}\\ &+2\int_{u}^{t}ds\  |-\tau\left( \Gamma(f,\phi_{t-s}((1-G_{\gamma})(g)),f_{s,t}^{(2)}(f,1,g))^{*})\right)\\ &+\tau\left( (f_{s,t}^{(2)}(f,1,g))^{*}\Delta(f)\phi_{t-s}((1-G_{\gamma})(g)))\right)|\\ &+4||f||||g||^2\left(\int_{u}^{t}ds\ \left((t-s)||A^{2}G_{\delta}f||_{2}+2||(G_{\delta}-id)Af||_{2}\right)(t-s)||Af||_{2}\right)^{1/2}
\\ & \leq  4\int_{u}^{t}ds\ \left(||\Gamma(f, f^{*})-x_{\epsilon}||_{1}^{1/2}2||f||||g||+|| (f_{s,t}^{(2)}(f,1,g))^{*}||_{2}||x_{\epsilon}||^{1/2}\right) ||AG_{\gamma}(g)||_{2}^{1/2}||g||_{2}^{1/2}
\\ &+ 2\int_{u}^{t}ds\ \tau\left( \Gamma(f,\phi_{t-s}((1-G_{\gamma})(g)))\phi_{t-s}((1-G_{\gamma})(g^{*}))),f^{*})\right)^{1/2}||\Delta^{1/2}f_{s,t}^{(2)}(f,1,g))||_{2}\\ &+4\int_{u}^{t}ds||f_{s,t}^{(2)}(f,1,g))||_{2}||\Delta(f)||_{2}||g||\\ &+\frac{4}{\sqrt{3}}||f||||g||^{2}(t-u)\left((t-u)^{1/2}||A^{2}G_{\delta}f||_{2}^{1/2}+\epsilon ||f||_{2}^{1/2}\right)||Af||_{2}^{1/2}
\\ &\leq 4\left((t-u)\epsilon2||f||^{2}||g||^{2}+\frac{4}{5}c'(f,g)^{1/2}(t-u)^{5/4}||x_{\epsilon}||^{1/2} ||AG_{\gamma}(g)||_{2}^{1/2}||g||_{2}^{1/2}\right)
\\ &+ 2\left(\int_{u}^{t}ds\ ||\Gamma(f^{*},f)-y_{\epsilon}||_{1}4||g||^{2}+||\phi_{t-s}((1-G_{\gamma})(g)))||_{2}^{2}||y_{\epsilon}||\right)^{1/2}\times \\ &\times\left(\int_{u}^{t}ds||\Delta^{1/2}f_{s,t}^{(2)}(f,1,g))||_{2}^{2}\right)^{1/2}\\ &+\frac{16}{5}c'(f,g)^{1/2}(t-u)^{5/4}||\Delta(f)||_{2}||g||\\ &+\frac{4}{\sqrt{3}}||f||||g||^{2}(t-u)\left((t-u)^{1/2}||A^{2}G_{\delta}f||_{2}^{1/2}+\epsilon ||f||_{2}^{1/2}\right)||Af||_{2}^{1/2} \end{align*}
\begin{align*}
&(\ast)\leq
 6\sqrt{\epsilon}||g||||f||\sqrt{t-u}\left(\int_{u}^{t}ds||\Delta^{1/2}f_{s,t}^{(2)}(f,1,g))||_{2}^{2}\right)^{1/2}\\ &+ 4(t-u)\epsilon||f||^{3/2}||g||^{2}(||Af||_{2}^{1/2}+||f||^{1/2})\\ &+8\delta||f||||g||^{2}(t-u)^{3/2}||Af||_{2} \\ &+\frac{8}{5}c'(f,g)^{1/2}(t-u)^{5/4} \left(||x_{\epsilon}||^{1/2} ||AG_{\gamma}(g)||_{2}^{1/2}||g||^{1/2}+2||\Delta(f)||_{2}||g||\right)
\end{align*}
Thus \begin{align*}&\left(\int_{u}^{t}ds\ ||\Delta^{1/2} f_{s,t}^{(2)}(f,1,g))||_{2}^{2}\right)^{1/2}\leq \sqrt{\epsilon}C(f,g)\sqrt{t-u}+d(f,g,\gamma,\epsilon,\delta)(t-u)^{5/8}(1+(t-u)^{1/8})
\end{align*}

 As a consequence 
\begin{align*}||&f_{u,t}^{(2)}(f,1,g)||_{2}^{2}\leq \epsilon C'(f,g)(t-u)+d'(f,g,\gamma,\epsilon,\delta)(t-u)^{9/8}(1+(t-u)^{3/8})
\end{align*}
We can now come back to the main line of the proof of convergences to zero, with all estimates required at hand.
\begin{align*}||\Gamma_{\alpha}^{(0)}(\xi)||_{1}&\leq \int_{0}^{\infty}dt\ \alpha^{2} e^{-\alpha t}||(1-\phi_{2t})(fg)-f((1-\phi_{2t})(g))||_{2}||1-\phi_{2t}(h)||_{2} \\&\leq \int_{0}^{\infty}dt\ \alpha^{2} e^{-\alpha t}||\phi_{2t}(fg)-f\phi_{2t}(g))||_{2}||1-\phi_{2t}(h)||_{2}\\&\leq \int_{0}^{\infty}dt\ \alpha^{2} e^{-\alpha t}\left(\sqrt{c'(f,g)}t^{1/4}+2t||Af||_{2}||g||\right)2t||Ah||_{2}.\end{align*}



 This converges to 0 since the measures $dt1_{|0,\infty]}\alpha^{2} e^{-\alpha t}t$ converges weakly to a Dirac in $0$.

Now we have to bound~:
\begin{align*}||&(G_{\alpha}\Gamma G_{\alpha}^{\o}-\Gamma)(\xi)||_{1}\\ &\leq
\alpha^{2}||\int_{0}^{\infty}dt\ e^{-\alpha t}\int_{0}^{t}ds\ \phi_{s}(\Gamma(\phi_{t-s}(fg),\phi_{t-s}(h)))-f\phi_{s}(\Gamma(\phi_{t-s}(g),\phi_{t-s}(h)))-\Gamma(f,g,h)||_{1}\\ &\leq \alpha^{2}\int_{0}^{\infty}dt\ e^{-\alpha t}||\int_{0}^{t}ds\ (f-\phi_{t}(f))\phi_{s}(\Gamma(\phi_{t-s}(g),\phi_{t-s}(h)))||_{1}\\ &+\alpha^{2}||\int_{0}^{\infty}dt\ \int_{0}^{t}ds\ \phi_{s}(\Gamma(\phi_{t-s}(fg),\phi_{t-s}(h)))-\phi_{t}(f)\phi_{s}(\Gamma(\phi_{t-s}(g),\phi_{t-s}(h)))-\Gamma(f,g,h)||_{1}\end{align*}
Let's bound each term~:\begin{align*}\alpha^{2}&\int_{0}^{\infty}dt\ e^{-\alpha t}||(f-\phi_{t}(f))\int_{0}^{t}ds\ \phi_{s}(\Gamma(\phi_{t-s}(g),\phi_{t-s}(h)))||_{1} \\ &\leq \alpha^{2}\int_{0}^{\infty}dt\ e^{-\alpha t}||(f-\phi_{t}(f))||_{2}||f_{0,t}^{(2)}(g,1,h)||_{2}\\\ &\leq\alpha^{2}\int_{0}^{\infty}dt\ e^{-\alpha t}2t||Af||_{2}c'(h,g)^{1/2}t^{1/4}\end{align*}

Again this gives convergence to zero.

Let us decompose a bit more the last term in the next~:

\begin{claim}
\begin{align*}&\int_{0}^{t}ds\ \phi_{s}(\Gamma(\phi_{t-s}(fg),\phi_{t-s}(h)))-\phi_{t}(f)\phi_{s}(\Gamma(\phi_{t-s}(g),\phi_{t-s}(h)))-\Gamma(f,g,h)\\ &=\int_{0}^{t}ds\ \phi_{s}(\Gamma(f_{s,t}^{(2)}(f,1,g),\phi_{t-s}(h)))+\phi_{s}(\Gamma(\phi_{t-s}(f),\phi_{t-s}(g),\phi_{t-s}(h)))-\Gamma(f,g,h)\\ &+\int_{0}^{t}du\ \phi_{u}(\Gamma(\phi_{t-u}(f),f_{u,t}^{(2)}(g,1,h)))
\end{align*}
\end{claim}
\begin{proof}
 (Note we use $\Gamma_{\beta}$ to be in the condition of lemma \ref{l3} (since $\Gamma_{\beta}(f,g)\in M$), otherwise, formally this claim is only two successive applications of this lemma (once in the computation of this page, once in the one of the next page) and the use of a derivation property for $\Gamma$ (in the next page)) .
 
 \begin{align*}&\int_{0}^{t}ds\ \phi_{t}(f)\phi_{s}(\Gamma(\phi_{t-s}(g),\phi_{t-s}(h)))-\phi_{s}(\Gamma(\phi_{t-s}(fg),\phi_{t-s}(h)))\\ &=-\int_{0}^{t}ds\ \phi_{s}(\Gamma(\phi_{t-s}(fg),\phi_{t-s}(h)))\\ &+\int_{0}^{t}ds\ \phi_{s}(\phi_{t-s}(f))\phi_{s}(\Gamma(\phi_{t-s}(g),\phi_{t-s}(h)))\\ & \ \ \ \ \ \ \ \ \ \ \ \ \ \ \ \ \ \ \ \ \ \ \ \ \ \ \ \ \  -\int_{0}^{t}ds\ \phi_{s,\beta}(\phi_{t-s,\beta}(f))\phi_{s,\beta}(\Gamma(\phi_{t-s}(g),\phi_{t-s}(h)))\\&+\int_{0}^{t}ds\ \phi_{s,\beta}(\phi_{t-s,\beta}(f))\phi_{s,\beta}(\Gamma_{\beta}(\phi_{t-s}(g),\phi_{t-s}(h)))\\ & \ \ \ \ \ \ \ \ \ \ \ \ \ \ \  -\int_{0}^{t}ds\  \phi_{s,\beta}(\phi_{t-s,\beta}(f))\phi_{s,\beta}(\Gamma_{\beta}(\phi_{t-s,\beta}(g),\phi_{t-s,\beta}(h)))\\ &+\int_{0}^{t}ds\ \phi_{s,\beta}(\phi_{t-s,\beta}(f))\phi_{s,\beta}((\Gamma-\Gamma_{\beta})(\phi_{t-s}(g),\phi_{t-s}(h)))\\  &+\int_{0}^{t}ds\ \phi_{s,\beta}(\phi_{t-s,\beta}(f)\Gamma_{\beta}(\phi_{t-s,\beta}(g),\phi_{t-s,\beta}(h)))\\ & \ \ \ \ \ \ \ \ \ \ \ \ \ \ \ \ \ \ \ \ \  -\int_{0}^{t}ds\ \int_{0}^{s}du\ \phi_{u,\beta}(\Gamma_\beta(\phi_{t-u,\beta}(f),\phi_{s-u,\beta}\Gamma_{\beta}(\phi_{t-s,\beta}(g),\phi_{t-s,\beta}(h)))
\end{align*}
\begin{align*}&\int_{0}^{t}ds\ \phi_{t}(f)\phi_{s}(\Gamma(\phi_{t-s}(g),\phi_{t-s}(h)))-\phi_{s}(\Gamma(\phi_{t-s}(fg),\phi_{t-s}(h)))\\ &=\int_{0}^{t}ds\ \phi_{s}(\phi_{t-s}(f))\phi_{s}(\Gamma(\phi_{t-s}(g),\phi_{t-s}(h)))\\ & \ \ \ \ \ \ \ \ \ \ \ \ \ \ \  -\int_{0}^{t}ds\ \phi_{s,\beta}(\phi_{t-s,\beta}(f))\phi_{s,\beta}(\Gamma(\phi_{t-s}(g),\phi_{t-s}(h)))\\ &+\int_{0}^{t}ds\ \phi_{s,\beta}(\phi_{t-s,\beta}(f))\phi_{s,\beta}(\Gamma_{\beta}(\phi_{t-s}(g),\phi_{t-s}(h)))\\ & \ \ \ \ \ \ \ \ \ \ \ \ \ \ \  -\int_{0}^{t}ds\  \phi_{s,\beta}(\phi_{t-s,\beta}(f))\phi_{s,\beta}(\Gamma_{\beta}(\phi_{t-s,\beta}(g),\phi_{t-s,\beta}(h)))\\&-\int_{0}^{t}ds\ \phi_{s,\beta}(\phi_{t-s,\beta}(f)\Gamma_{\beta}(\phi_{t-s}(g),\phi_{t-s}(h)))\\ & \ \ \ \ \ \ \ \ \ \ \ \ \ \ \  +\int_{0}^{t}ds\ \phi_{s,\beta}(\phi_{t-s,\beta}(f)\Gamma_{\beta}(\phi_{t-s,\beta}(g),\phi_{t-s,\beta}(h)))\\ &+\int_{0}^{t}ds\ \phi_{s,\beta}(\phi_{t-s,\beta}(f)\Gamma(\phi_{t-s}(g),\phi_{t-s}(h)))\\ & \ \ \ \ \ \ \ \ \ \ \ \ \ \ \  -\int_{0}^{t}ds\ \phi_{s}(\phi_{t-s}(f)\Gamma(\phi_{t-s}(g),\phi_{t-s}(h)))\\ &+\int_{0}^{t}ds\ \phi_{s,\beta}(\phi_{t-s,\beta}(f))\phi_{s,\beta}((\Gamma-\Gamma_{\beta})(\phi_{t-s}(g),\phi_{t-s}(h)))\\ & \ \ \ \ \ -\int_{0}^{t}ds\ \phi_{s,\beta}(\phi_{t-s,\beta}(f)(\Gamma-\Gamma_{\beta})(\phi_{t-s}(g),\phi_{t-s}(h)))\\ &+\int_{0}^{t}ds\ \phi_{s}(\phi_{t-s}(f)\Gamma(\phi_{t-s}(g),\phi_{t-s}(h)))\\ & \ \ \ \ \ \ \ \ \ \ \ \ \ \ \  -\int_{0}^{t}ds\ \int_{0}^{s}du\ \phi_{u,\beta}(\Gamma_{\beta}(\phi_{t-u,\beta}(f),\phi_{s-u,\beta}\Gamma_{\beta}(\phi_{t-s,\beta}(g),\phi_{t-s,\beta}(h)))\\ &-\phi_{s}(\Gamma(\phi_{t-s}(fg),\phi_{t-s}(h)))
\end{align*}
The two terms of the fifth pair of lines of the last equation tend to $0$ with $\beta\rightarrow\infty$ in $L^{1}(M)$, by DCT using lemma \ref{l5} for the pointwise limit and domination via $(1+16(K+1)^{2})||f||||\Delta^{1/2}\phi_{t-s}(h)||_{2} ||\Delta^{1/2}\phi_{t-s}(g)||_{2}$.
The four first pair of lines tend to 0 (each  pair of lines) at least weakly in $L^{1}$ using only standard results (including one proven during the proof of lemma \ref{l6} for the second and third line $\int_0^t\tilde{\E}_{1}^{\alpha}(\phi_{t-s}(a)-\phi_{t-s,\alpha}(a))\to 0$ any $a\in M$).
We will write those five first  pair of lines $(1)$ from now on in the computation. 

\begin{align*}&\int_{0}^{t}ds\ \phi_{t}(f)\phi_{s}(\Gamma(\phi_{t-s}(g),\phi_{t-s}(h)))-\phi_{s}(\Gamma(\phi_{t-s}(fg),\phi_{t-s}(h)))\\ &=(1)+\int_{0}^{t}ds\ \phi_{s}(\Gamma(\phi_{t-s}(f)\phi_{t-s}(g),\phi_{t-s}(h)))-\phi_{s}(\Gamma(\phi_{t-s}(fg),\phi_{t-s}(h)))\\ &-\int_{0}^{t}ds\ \phi_{s}(\Gamma(\phi_{t-s}(f),\phi_{t-s}(g),\phi_{t-s}(h)))\\ &-\int_{0}^{t}du\ \phi_{u,\beta}(\Gamma_{\beta}(\phi_{t-u,\beta}(f),f_{u,t}^{(2,\beta)}(g,1,h)))
\\ &=(1)-\int_{0}^{t}ds\ \phi_{s}(\Gamma(f_{s,t}^{(2)}(f,1,g),\phi_{t-s}(h)))+\phi_{s}(\Gamma(\phi_{t-s}(f),\phi_{t-s}(g),\phi_{t-s}(h)))\\ &+\int_{0}^{t}du\ \phi_{u}(\Gamma(\phi_{t-u}(f),f_{u,t}^{(2)}(g,1,h)))-\int_{0}^{t}du\ \phi_{u,\beta}(\Gamma_{\beta}(\phi_{t-u,\beta}(f),f_{u,t}^{(2,\beta)}(g,1,h)))\\ &-\int_{0}^{t}du\ \phi_{u}(\Gamma(\phi_{t-u}(f),f_{u,t}^{(2)}(g,1,h)))
\end{align*}

Now, the second line of the last equation tends to zero, like at the beginning of the proof of lemma \ref{l6}, using the fact proven at the end of  lemma \ref{l6} that $\int_{0}^{t}du\ \tilde{\E}_{1}^{\beta}(f_{u,t}^{(2,\beta)}(g,1,h)-f_{u,t}^{(2)}(g,1,h))\rightarrow 0$.
\end{proof}

We can now use this claim to conclude the proof of the lemma. The previously obtained results readily gives (and this will conclude since $\epsilon$ was arbitrary):

\begin{align*}&\alpha^{2}\int_{0}^{\infty}dt\ e^{-\alpha t}||\int_{0}^{t}ds\phi_{s}(\Gamma(f_{s,t}^{(2)}(f,1,g),\phi_{t-s}(h)))||_{1}\\ &\leq \alpha^{2}\int_{0}^{\infty}dt\ e^{-\alpha t}\left(\sqrt{\epsilon}C(f,g)\sqrt{t}+d(f,g,\gamma,\epsilon,\delta)t^{5/8}(1+t^{1/8})\right)\sqrt{t}||Ah||_{2}^{1/2}||h||_{2}^{1/2},
\\ &\limsup_{\alpha\rightarrow\infty}\ \alpha^{2}\int_{0}^{\infty}dt\ e^{-\alpha t}||\int_{0}^{t}ds\ \phi_{s}(\Gamma(f_{s,t}^{(2)}(f,1,g),\phi_{t-s}(h)))||_{1}\leq \sqrt{\epsilon}C(f,g)||Ah||_{2}^{1/2}||h||_{2}^{1/2},
\\ & \alpha^{2}\int_{0}^{\infty}dt\ e^{-\alpha t}||\int_{0}^{t}ds\ \phi_{s}(\Gamma(\phi_{t-s}(f)-f,\phi_{t-s}(g),\phi_{t-s}(h)))||_{1}\\ &\leq \alpha^{2}\int_{0}^{\infty}dt\ e^{-\alpha t}\left(\int_{0}^{t}ds\ 2\sqrt{t-s}||Af||_{2}^{2}\right)^{1/2}||g||\sqrt{t}||Ah||_{2}^{1/2}||h||_{2}^{1/2}\rightarrow0
\\ & \alpha^{2}||\int_{0}^{\infty}dt\ e^{-\alpha t}\int_{0}^{t}ds\ \phi_{s}(\Gamma(f,\phi_{t-s}(g_{\alpha})-g_{\infty},h))||_{1}=||\alpha\int_{0}^{\infty}dt\ e^{-\alpha t}G_{\alpha}(\Gamma(f,\phi_{t}(g_{\alpha})-g_{\infty},h))||_{1}\\ &\leq \alpha\int_{0}^{\infty}dt\ e^{-\alpha t}\tau(\Gamma(f^{*},f)(\phi_{t}(g_{\alpha})-g_{\infty})(\phi_{t}(g_{\alpha})-g_{\infty})^{*})^{1/2}||\Delta^{1/2}h||_{2}\rightarrow0
\\ &\alpha^{2}\int_{0}^{\infty}dt\ e^{-\alpha t}\int_{0}^{t}ds(\phi_{s}-id)(\Gamma(f,g_{\infty},h))=(G_{\alpha}-id)\Gamma(f,g_{\infty},h)\ ...
\end{align*}

\end{proof}
\subsection{Semigroups on $M\star_{alg} M$}
For Markov processes on commutative algebras, a semigroup as the one of the previous part $\phi_{t}$  is sufficient to define the law of the process. Indeed as is well-known, we can use for ordered times $t_1<t_2<...<t_n$, the formula $\tau(X_1^{[t_1]}...X_n^{[t_n]})=\tau(X_1\phi_{t_2-t_1}(X_2...\phi_{t_n-t_{n-1}}(X_n))$ to define the new trace (also written $\tau$ on the commutative Path space ($X_i$ thought of at time $t_i$ denoted $X_i^{[t_i]}$)). And of course, one may think $\phi_t$ as a semigroup on $M\otimes M$ defined by $\Phi_{t}=\phi_{t}\o Id$ (the second variable being thought of as non evolving and useful to determine the joint law after multiplication). 

We will analogously consider semigroups $\Phi_{s,t}$ on $M\star_{alg} M$ above the previous semigroup i.e. with $\Phi_{s,t}(XYZ)=X\phi_{t-s}(Y)Z$ for $X,Z$ in the second (zero-time) summand and $Y$ in the first (t-time) summand of the free product (we also require this property with $\phi_{t-s}(Y)$ replaced by $\Phi_{s,t}(Y)$ and a general $Y$ in the free product  i.e. $\Phi_{s,t}$ behaves like a conditional expectation since it will the restriction to $M^{[t]}*M^{[0]}$ of the conditional expectation on $M^{[s]}*M^{[0]}$). For the processes we are interested in, taking a conditional expectation of a polynomial in two times variables $(t,0)$ on the $0-s$ time subalgebra will lie in this algebraic free product (or rather this conditional expectation will be $\Phi_{s,t}$ applied on this polynomial). 
Let us introduce notations to compute more easily in this setting. 
 We will put in exponent the time of the algebra considered (the non-zero time being the first summand by convention). We will often consider symmetric semigroups, i.e. those satisfying~:$\tau(\Phi_{0,t}(X_{1}^{[t]}X_{2}^{[0]}...X_{2n-1}^{[t]}))=\tau(\Phi_{0,t}(X_{1}^{[0]}X_{2}^{[t]}...X_{2n-1}^{[0]}))$ where $\tau$ denotes (again) the composition of the trace $\tau$ on $M$ with the natural product on algebraic free product with value $M$. We may later write $S_{t}(X_{1}^{[t]}X_{2}^{[0]}...X_{2n-1}^{[t]})=X_{1}^{[0]}X_{2}^{[t]}...X_{2n-1}^{[0]}$ and write this $\tau\circ\Phi_{0,t}=\tau\circ\Phi_{0,t}\circ S_{t}$.

Moreover, we will need in the sequel (in order to define n times free products and not only two times free products) a slightly more general context given a family $M^{\rho}$ ($\rho$ in a set R) of (positively exhaustively) filtered *-algebraic normed non-commutative probability spaces containing $M$ as a sub-probability space (Let us explain our main example of interest. $R_{n}=\{0\leq t_1\leq t_2 ...\leq t_n\}$, for $\rho\in R_n$, $M^{\rho}=M^{*_{alg}(n+1)}$ is a multi-time free product (the first summand thought of at time $t_n$, the second at time $t_{n-1}$ up to the last one at time 0)
 the filtration $M^{\rho}_{p}$  of algebra, crucial for inductions, is the length in free product (i.e. $P\in M^{\rho}_{p}$ if it can be written as a sum of products of not more than $p$ alternating times). 
We think we have already built a tracial state on $M^{\rho}$ and $M$ lies inside as the highest time sub-algebra here $M_{t_n}$, the state on it coinciding with the previous one. We will make evolve a new highest time put in the other term of the free product. The norm will be the projective norm on the algebraic free product.).
Moreover, let us assume given a set of monomials of degree $p$, $Mon^{\rho}_{p}$ a metric space, with
a continuous (``multiplication") map onto $M^{\rho}_{p}$. This seems a little artificial beyond our main example, but we don't want to be as general as possible, the above notations being mainly a way of emphasizing the main assumptions. Especially we limit here generality to avoid later intricate general inequalities where the argument is quite clear in the main example of interest. (Thus in our main example $Mon^{\rho}_{p}=\sqcup M^{p}$ is a disjoint union of products of $M$ with the product topology, one copy in the disjoint union for each possible alternation of times. The map to the free product being clear). We will use it as a way of having fixed a product decomposition and saying a given product decomposition converges to another one. We put two metrics on it, in our main example, each component of the direct sum being at infinite distance, and $d_{\infty}((a_1,...,a_p),(b_1,...,b_p))=\max||a_i-b_i||$, and $d_{2}((a_1,...,a_p),(b_1,...,b_p))=\max||a_i-b_i||_2$. We  will be interested in sequences of ``monomials" bounded in $d_{\infty}$ converging in $d_{2}$.

 We will consider semigroups  $\Phi_{s,t}^{\rho}$ on $M\star_{alg} M^{\rho}$ (with obvious filtration adding the smallest filtration degrees, we count $1$ as filtration degree for the first summand, we may obviously take the smallest products to get the smallest filtration degrees. In that way, in our main example, for $\rho= (t_1\leq t_2 ...\leq t_n)$, we identify (with the same filtration norm etc) $M\star_{alg} M^{\rho}$ with $M^{(\rho,t)}$, by definition $(\rho,t)=(t_1\leq t_2 ...\leq t_n\leq t_n+t)$). This semigroup will be above the previous semigroup as before, i.e. will satisfy $\Phi_{s,t}^{\rho}(XYZ)=X\phi_{t-s}(Y)Z$ for $X,Z$ in the second (``zero-time") summand (i.e. $M^{\rho}$) and $Y$ in the first (t-time) summand of the free product (we also require this property with $\phi_{t-s}(Y)$ replaced by $\Phi_{s,t}^{\rho}(Y)$ and a general $Y$ in the free product). 

We assume given an evolution  (a continuous linear map) up to ``real time 0" $\Psi^{(\rho)}:M^{\rho}\rightarrow M$ (we want to think of as an already built evolution via semigroup from the highest time in $M^{\rho}$ to $0$. More specifically in our main example, if $\rho=(t_1)$ then $\Psi^{(\rho)}=\Phi_{0,t_1}$ the two times map described before. If $\rho= (t_1\leq t_2 ...\leq t_n)=(\rho',t_n-t_{n-1})$ with  $\rho'= (t_1\leq t_2 ...\leq t_{n-1})$ $\Psi^{(\rho)}=\Psi^{(\rho')}\circ \Phi_{0,t_n-t_{n-1}}^{\rho'}$). 


 We assume given two maps $\sigma~:\R_{+}\times R\rightarrow R$ and $\tau: R\rightarrow\R_{+}$ (we think of them as taking the new time $t$ above the highest time and the previously built times (ordered) $(t_{1}\leq...t_{p})=\rho$ as a point of $R_p$, and giving via the map $(\tau,\sigma)$ the corresponding symmetric family $(t_{1},(t\leq t+t_{p}-t_{p-1}\leq...t+t_{p}-t_{1}))$). We assume this map gives an involution~: $\tau(\sigma(t,\rho))=t,\sigma(\tau(\rho),\sigma(t,\rho))=\rho$. We also need a symmetry map (preserving filtration, continuous for ``projective" 
norms, in fact isometric on monomials for $d_{2}$ and $d_{\infty}$) $S_{t,\rho}:M\star_{alg} M^{\rho}\rightarrow M\star_{alg} M^{\sigma(t,\rho)}$ (again we think it as doing the above symmetry especially exchanging the new highest time and 0, we write analogously maps on underlying tensor products, e.g. in our example $S_{t,(t_1)}(X^{[t+t_1]}Y^{[t_1]}Z^{[0]}T^{[t+t_1]}U^{[t_1]}V^{[0]})=(X^{[0]}Y^{[t]}Z^{[t+t_1]}T^{[0]}U^{[t]}V^{[t+t_1]})$  thus seen on monomials this gives $S_{t,(t_1)}(X^{[t+t_1]},Y^{[t_1]}Z^{[0]},T^{[t+t_1]},U^{[t_1]}V^{[0]})=(X^{[0]}Y^{[t]},Z^{[t+t_1]},T^{[0]}U^{[t]},V^{[t+t_1]})$. ) We will assume $S_{t,\rho}$ has inverse $S_{\tau(\rho),\sigma(t,\rho)}$. We will often consider $(\tau,\sigma,S)$-symmetric families satisfying $\tau(\Phi_{0,t}^{\rho}(X_{1}^{[t]}X_{2}^{[0]}...X_{2n-1}^{[t]}))=\tau(\Phi_{0,t}^{\sigma(t,\rho)}(S_{t,\rho}(X_{1}^{[t]}X_{2}^{[0]}...X_{2n-1}^{[t]})))$, where again $\tau$ denotes the state either in $M^{\rho}$ or $M^{\sigma(t,\rho)}$ composed with the multiplication maps induced on free product by inclusion of $M$.
Finally, to prove that our formulas will keep this symmetry, we will need to prove in a next section alternative formulas replacing evolution of the last time by evolutions in terms of the first time (We will also assume we have first time rewriting of evolutions on $M^{\rho}$, of course, and they will be given by induction in applications). To achieve this goal, 
we don't only need the evolution $\Psi^{(\rho)}$ up to time $0$, but also a decomposition of it via a map 
$\Psi^{(\rho),\tau(\rho)}:M^{\rho}\rightarrow M_{[\tau(\rho)]}*_{alg}M_{[0]}$ (the indices only showing the times we think M to live at, recall, we think about $\tau(\rho)$, not depending on $t$, in general as the first time in $\rho$) such that $\Psi^{(\rho)}=\Phi_{0,\tau(\rho)}\circ\Psi^{(\rho),\tau(\rho)}$ (recall $\Phi$ is the map we will build first on $M*_{alg} M$ as described first in the first paragraph of this section, we should say our semigroup is above this $\Phi$ instead of only $\phi$. In our example of main interest, we of course have $\rho=(t_1=\tau(\rho)\leq...t_p)$. If $p=2$ $\Psi^{(\rho),\tau(\rho)}=\Phi_{0,t_2-t_1}^{(\tau(\rho))}$ so that with have the claimed identity since by definition $\Psi^{(\rho)}=\Psi^{(\tau(\rho))}\circ \Phi_{0,t_2-t_{1}}^{(\tau(\rho))=\rho'}$. In general inductively, we take  $\Psi^{(\rho),\tau(\rho)}=\Psi^{(\rho'),\tau(\rho)}\circ \Phi_{0,t_p-t_{p-1}}^{\rho'}$ where $\rho'=(t_1=\tau(\rho)\leq...t_{p-1})$. Especially, this map is given for free by the construction).

 In that context we may also add filtration degrees in exponents and not only times, e.g. $X^{[0,k]}$.

\begin{definition}
The \emph{level-family} of the semigroup $\Phi_{s,t}$ (always assumed above $\phi$ in the previous sense) on $M\star_{alg} M$ (resp. $\Phi_{s,t}^{\rho}$ on $M\star_{alg} M^{\rho}$)  is a family $f_{s,t}^{(n)}:M^{2n-1}\rightarrow M$ ( resp. $f_{s,t}^{(n,\rho)}:\bigoplus_{n=p+\sum n_{i}}M\otimes M^{\rho}_{n_{1}}\otimes M...M^{\rho}_{n_{p-1}}\otimes M\rightarrow M$) of maps inductively defined (if they exist) such that~: \begin{align*}&\Phi_{s,t}(X_{1}^{[t]}X_{2}^{[0]}...X_{2n-1}^{[t]})=f_{s,t}^{(n)}(X_{1}^{[t]},X_{2}^{[0]},...,X_{2n-1}^{[t]})^{[s]} + \sum_{i_{1}< ...< i_{p}\in[1,n-1]}\\&(f_{s,t}^{(i_{1})}(X_{1}^{[t]},...,X_{2i_{1}-1}^{[t]}))^{[s]}X_{2i_{1}}^{[0]}(f_{s,t}^{(i_{2}-i_{1})}(X_{2i_{1}+1}^{[t]},...,X_{2i_{2}-1}^{[t]}))^{[s]}X_{2i_{2}}^{[0]}...(f_{s,t}^{(n-i_{p})}(X_{2i_{p}+1}^{[t]},...,X_{2n-1}^{[t]}))^{[s]}\\ &
\\&
(resp .\ \Phi_{s,t}^{\rho}(X_{1}^{[t]}X_{2}^{[0,n_{1}]}...X_{2n}^{[0,n_{p}]}X_{2p-1}^{[t]})=f_{s,t}^{(n,\rho)}(X_{1}^{[t]},X_{2}^{[0]},...,X_{2p-1}^{[t]})^{[s]} + \sum_{i_{1}< ...< i_{q}\in[1,p-1]}\\ &(f_{s,t}^{(d_{0,i_{1}},\rho)}(X_{1}^{[t]},...,X_{2i_{1}-1}^{[t]}))^{[s]}X_{2i_{1}}^{[0]}(f_{s,t}^{(d_{i_{1},i_{2}},\rho)}(X_{2i_{1}+1}^{[t]},...,X_{2i_{2}-1}^{[t]}))^{[s]}X_{2i_{2}}^{[0]}...(f_{s,t}^{(d_{i_{q},p},\rho)}(X_{2i_{q}+1}^{[t]},...,X_{2p-1}^{[t]}))^{[s]}),\end{align*}
With $d_{i,j}=j-i+n_{i+1}+...+n_{j-1}$. We may remark $d_{0,p}=n$. We will also sometimes write a degree in index before $\rho$ in $\Phi$ to explicit the filtration degree. More generally we will consider maps (``semigroups below level $N$")  $\Phi^{(N,\rho)}$ only defined on the filtration space $(M* M_{\rho})_{N}$ so that this is of course equivalent to getting an associated level family $(f^{(n,\rho)})_{n\leq N}$.

Said otherwise $f_{0,t}$ is another way of writing a \textit{boolean cumulant} in the (Boolean) non-commutative probability space $M_{[t]}*_{alg}M_{[0]}$, with $\Phi_{0,t}$ as ``conditional expectation" on the component $M_{[0]}$ thought of at time $0$. (see e.g. definition 4.1 in \cite{Po08}, the case $s\neq 0$ is a variant where $\Phi_{s,t}$ is the restriction to $M_{[t]}*_{alg}M_{[0]}$ of the conditional expectation on $M_{[s]}*_{alg}M_{[0]}$ defined in $M_{[t]}*_{alg}M_{[s]}*_{alg}M_{[0]}$, in that context, the claim concerning existence says such a corresponding boolean cumulant has to be valued in $M^{[s]}$; For us this notation has no relation with Boolean probability beyond the fact $f$ satisfy a differential equation easier to read really likely because Boolean probability is natural in non-commutative setting without trace, and we thus divide here the part depending on trace in $f$ and the purely non-commutative one here).

Especially $f_{s,t}^{(1)}=\phi_{t-s}$. We say it is \emph{bounded} (resp. \emph{locally bounded}) (in $M$) 

\noindent if $||f_{s,t}^{(n)}(X_{1},...,X_{2n-1})||\leq C_{n}||X_{1}||...||X_{2n-1}||$ (resp. for $s,t\in[0,T]$ for $C_{n}(T)$).

Given a generator $-1/2A$ of an extension $\phi_{t}:L^{2}(M)\rightarrow L^{2}(M)$ of $f_{0,t}^{(1)}$, generator of a conservative (non-symmetric) completely Dirichlet form as in the previous part, consider the corresponding carr\'{e} du champs $\Gamma(f,g,h)$, $\Delta$ the generator of the symmetric part (recall we assumed $\Gamma$ is also the Carr\'{e} du champs of the symmetric part).
We say $f$ is \emph{locally $\delta$-bounded} if it is Bochner measurable as valued in $D(\Delta^{1/2})=D(\E)$ and if $\int_{u}^{t}ds ||\Delta^{1/2}f_{s,t}^{(n)}(X_{1}^{[t]},X_{2}^{[0]},...,X_{2n-1}^{[t]})||_{2}^{2}<\infty$ for any $u,t,n.$

Given another semigroup bellow level $(N-1)$ $H_{s,t}^{N-1}$ ($H_{s,t}^{N-1,\rho}$ in the case of supplementary index $\rho$), and corresponding family $(h_{s,t}^{(n,\rho)})_{n\leq N-1}$. We assume $H$ is locally bounded in $M$
 We say $(f^{(n)})_{n\leq N}$ is \emph{affiliated to $A$} (or $\phi_{t}$) relative to $H$ if it is locally bounded in $M$ and locally $\delta$-bounded so that the following integral converges absolutely (as Bochner integral)  and~:
\begin{align*}&f_{s,t}^{(n)}(X_{1}^{[t]},...,X_{2n-1}^{[t]})=\sum_{i\leq j=1}^{n-1}\int_{s}^{t}du\\ & \phi_{u-s}(\Gamma(f_{u,t}^{(i)}(X_{1}^{[t]},...,X_{2i-1}^{[t]}),H_{0,u}^{(j-i+1)}(X_{2i}^{[u]},F\Phi_{u,t}(X_{2i+1}^{[t]},...,X_{2j-1}^{[t]}),X_{2j}^{[u]}),f_{u,t}^{(n-j)}(X_{2j+1}^{[t]},...,X_{2n-1}^{[t]})),\\ &
\end{align*}(resp. in forgetting arguments since indices determine them uniquely $$f_{s,t}^{(n,\rho)}=\sum_{i\leq j=1}^{p-1}\int_{s}^{t}du\ \phi_{u-s}\circ\Gamma\circ f_{u,t}^{(d_{0,i},\rho)}\o\left(\Psi_{H}^{(\sigma(u,\rho))}\circ H_{0,\tau(\rho)}^{(N-1,\sigma(u,\rho))}\circ S_{(\tau,\sigma)(u,\rho)}\circ Id\o\Phi_{u,t}^{\rho}\o Id\right)\o f_{u,t}^{(d_{j,p},\rho)}.$$
($F$ is the formal change of indices u-0 corresponding to the symmetry in second case and not really useful in the first with our definition not taking into account indices explicitly except implicitly by position in the maps, we put it for clarity)

When $\Phi$ is symmetric, one may  want $H$ to be an inductively already built version of $\Phi$.
The meaning of the sum is then that we evolve up to time $u$ gathering the prescribed blocks and at that time, we gather them so that outside we are in $M$ seen at time $u$ evolving according to $\phi$ up to $s$, in the middle we have to take a trace of evolution from time $u$ up to $0$ and we use symmetry in order to have (real) 0 times on $\delta$'s included in $\Gamma$. To do this we can evolve up to time $0$ understood as being in $M^{(\sigma(u,\rho))}$ (we should note the degree of the middle term is at most $d_{i-1,j+1}-2$) first and then evolve via $\Psi$ to the real time $0$ in it.

 We have the same definitions for a family as above up to level $N$ and call it an \emph{N-level-semigroup-family} which includes the corresponding restriction of semigroup property for the above defined $\Phi_{s,t}$ and corresponding restriction of property 
 of being above $\phi$ as before.
An \emph{1-level-semigroup-family} is thus merely a semigroup on $M$ equal to $\phi$.\end{definition}

In order to define inductively an \emph{N-level-semigroup-family}, we will need a notion of $\alpha$-approximation. Recall $\phi_{t,\alpha}$ denote the semigroup of generator $-1/2AG_{\alpha}$
Replacing also $\Gamma$ by $\Gamma_{\alpha}$ (of the previous part) we have all the notions for an \emph{$\alpha$-N-level-semigroup-family}. In the case with extra index $\rho$, we have also to give us variants $H^{\alpha}$, $\Psi^{\alpha}$. Let us define approximation properties for our level semigroup families.
\begin{definition}
Let $(f_{s,t}^{(n,\alpha)})_{n\leq N}$ an \emph{$\alpha$-N-level-semigroup-family} and $(f_{s,t}^{(n)})_{n\leq N}$ an \emph{N-level-semigroup-family}. $ f_{s,t}^{(n,\alpha)}$ converges weakly in $L^{1}$ to $f$ if for any $m\in M$, $n\leq p$, $s,t$~: 

\noindent$ \tau(m^{[s]}f_{s,t}^{(n,\alpha)}(X_{1}^{[t]},X_{2}^{[0]},...,X_{2n-1}^{[t]})^{[s]})\rightarrow \tau(m^{[s]}f_{s,t}^{(n)}(X_{1}^{[t]},X_{2}^{[0]},...,X_{2n-1}^{[t]})^{[s]})$ as $\alpha\rightarrow\infty$.

It converges in $L^{2}$ if for any $n\leq N$, $s,t$~:

\noindent $||f_{s,t}^{(n,\alpha)}(X_{1}^{[t]},X_{2}^{[0]},...,X_{2n-1}^{[t]})-f_{s,t}^{(n)}(X_{1}^{[t]},X_{2}^{[0]},...,X_{2n-1}^{[t]})^{[s]})||_{2}\rightarrow 0$.

It is $\delta$-convergent if (they are locally $\delta$-bounded, it converges in $L^{2}$ and) for any $n\leq N$, $u,t$~: $$\int_{u}^{t}ds\ \tilde{\E_{1}^{\alpha}}(f_{s,t}^{(n,\alpha)}(X_{1}^{[t]},X_{2}^{[0]},...,X_{2n-1}^{[t]})-f_{s,t}^{(n)}(X_{1}^{[t]},X_{2}^{[0]},...,X_{2n-1}^{[t]}))\rightarrow 0.$$

It is $\delta^{+}$-convergent (resp $L^{2+}$) if all the above $L^{2}$ and $\delta$-convergences  (resp only the $L^{2}$) can be improved for sequences $X_{i}^{\alpha}$ (of Monomials in the $\rho$ case)  uniformly bounded in $M$ (resp $d_\infty$) and converging in $L^{2}(M)$ to $X_{i}$ (resp in $d_{2}$ for monomials). (The whole point of this last definition is to circumvent lack of traciality of approximations in the non-symmetric case by improving convergences obtained). 
\end{definition}

We have an analogous definition for $\Psi^{\alpha}$ converging to $\Psi$. The next result basically says we have gathered all potentially useful convergence properties to carry on a definition by induction, as soon as we assume as an extra assumption a boundedness we will prove later by a positivity argument. The analogue statement with $\rho$'s is of course also true, if one assumes $\Psi_H^{\alpha}$ $L^{2+}$-convergent to $\Psi_H$ and locally bounded in $M$ (uniformly in $\alpha$) (i.e. basically the assumptions bellow for $H$).
\begin{theorem}Let $(f_{s,t}^{(n,\alpha)})_{n\leq N}$,$(h_{s,t}^{(n,\alpha)})_{n\leq N}$ be \emph{$\alpha$-N-level-semigroup-families} with $f$ affiliated to $AG_{\alpha}$ relative to $(h_{s,t}^{(n,\alpha)})_{n\leq N-1}$ and $(f_{s,t}^{(n)})_{n\leq N},(h_{s,t}^{(n)})_{n\leq N}$ be \emph{N-level-semigroup-families} with $f$ affiliated to $A$ relative to $h$ similarly. Assume moreover that $(f_{s,t}^{(n,\alpha)})_{n\leq N}$ is $\delta^{+}$-convergent to $(f_{s,t}^{(n)})_{n\leq N}$ and  $(h_{s,t}^{(n,\alpha)})_{n\leq N}$ bounded (uniformly in $\alpha$) in M 
and converges in $L^{2+}$ to $(h_{s,t}^{(n)})_{n\leq N}$. Assume moreover either $D(A)\cap D(\Delta)\cap M$ is a core in $D(\Delta^{1/2})$ or $(f_{s,t}^{(n,\alpha)})_{n\leq N}$ valued in $D(A)\cap D(\Delta)\cap M$. Then the above formulas define an \emph{(N+1)-level-semigroup-family} and \emph{$\alpha$-(N+1)-level-semigroup-family}. We assume the $\alpha$ family $(f_{s,t}^{(n,\alpha)})_{n\leq N+1}$ is locally bounded in $M$ uniformly in $\alpha$. Then it is $\delta^{+}$-convergent to $(f_{s,t}^{(n)})_{n\leq N+1}$ (and they  are necessarily affiliated in the above sense since (in part assumed in part proven to be) locally bounded in $M$). We will call $f$ an \emph{$\alpha$-approximated} (N+1)-level-semigroup-family affiliated to $A$ relative to $h$. When $f$ an \emph{$\alpha$-approximated} (N+1)-level-semigroup-family affiliated to $A$ relative to $h$ and $h$ an \emph{$\alpha$-approximated} (N+1)-level-semigroup-family affiliated to $A^{*}$ relative to $f$ we will say $f$ and $h$ are  \emph{$\alpha$-approximated} (N+1)-level-semigroup-families mutually affiliated to $A$ and $A^{*}$.
 \end{theorem}

\begin{proof}

The proof is similar to the one of lemma \ref{l6} using of course lemma \ref{GammaLim} instead of lemma \ref{l5}. Note also the proof of lemma \ref{l6} contains the initialization for the induction this theorem will enable.
The semigroup property is obvious from the definition.

We can consider all results on each term of the definition above of $f_{s,t}^{(N+1)}$ (the only map for which something needs to be proven), all of the form~:
$$j_{s,t}=\int_{s}^{t}du\\ \ \phi_{u-s}(\Gamma(f_{u,t},g_{0,u;t},h_{u,t})),$$
with by assumption $g_{0,u}$,$f_{u,t}$,$h_{u,t}$ locally bounded in $M$, and $f_{u,t}$,$h_{u,t}$ for each $t$ in $L^{2}([0,t],D(\Delta^{1/2}))$. We also consider $\alpha$ variants ( with $X_{i}^{\alpha}$ instead $X_{i}$ as in the definitions of $+$ convergences) uniformly in $\alpha$ with the same properties and the convergences of the previous definition (remark that the conjunction of boundedness and the various $L^{2}+$ convergences, of $\Phi, H ,\Psi$, with the various isometric on monomials of $S$ (for $d_2$ and $d_\infty$),
 enable $g_{0,u;t}^{(\alpha)}$ to converge in $L^{2}(M)$ (pointwise in $u,t$ to $g_{0,u;t}$)). We prove first $\sigma$-weak convergence in $M$ of $f_{s,t}^{(N+1,\alpha)}$. By weak compacity (since $f_{s,t}^{(N+1,\alpha)}$ uniformly bounded in M), it suffices to prove uniqueness of the limit of $j_{s,t}^{\alpha}$ (we will identify it to be $j_{s,t}$ and as a consequence prove the sum of such $j_{s,t}$ to be in $M$  and equal $f_{s,t}^{(N+1)}$ thus in $M$ not only in $L^{1}$ like $j_{s,t}$) and by density, in taking duality with elements of the form $b\in M$  we in fact prove $j_{s,t}^{\alpha}$ which is only a priori in $L^{1}$ converges weakly in that space and this suffices to get convergence in $M$ of the sum of them $f_{s,t}^{(N+1,\alpha)}$ for which we assumed boundedness 
. Now the lemma \ref{GammaLim} above proves $(\Gamma^{\alpha})(f_{u,t},g_{0,u}^{\alpha},h_{u,t})-(\Gamma^{\alpha})(f_{u,t},g_{0,u},h_{u,t})$ converges pointwise in  $L^{1}(M)$, since we can dominate them as in this lemma by $(1+16(K+1)^{2})||\Delta^{1/2}f_{u,t}||_{2}||\Delta^{1/2}h_{u,t}||_{2} \sup_{\alpha,u}||g_{0,u}^{\alpha}||$ we assumed in $L^{1}([0,t])$ we get convergence in $L^{1}([0,t],L^{1}(M))$ by DCT. Moreover $||\Gamma^{\alpha}(f_{u,t}^{\alpha}-f_{u,t},g_{0,u}^{\alpha},h_{u,t}^{\alpha}))||_{1}\leq \tilde{\E_{1}^{\alpha}}(f_{u,t}^{\alpha}-f_{u,t})^{1/2}\tilde{\E_{1}^{\alpha}}(h_{u,t}^{\alpha})^{1/2}||g||$ since we assumed convergence of this in $L^{1}$ modulo Cauchy-Schwarz, this goes to zero. 
Finally, $(\phi_{u}^{\alpha}-\phi_{u})\Gamma(f_{u,t},g_{0,u},h_{u,t})$ is also well known to converge in $L^{1}$. We have thus even proven norm convergence of $j_{s,t}^{\alpha}$ in $L^{1}$ at this stage.

We now need a formula for the $L^{2}$-norm~: $$||f_{s,t}^{(N+1,\alpha)}||_{2}^{2}=\sum\int_{s}^{t}du\ 2\Re\tau(\phi_{u-s}^{\alpha}(f_{s,t}^{(N+1,\alpha)*}) \phi_{u-s}^{\alpha}(\Gamma^{\alpha}(f_{u,t}^{\alpha},g_{0,u;t}^{\alpha},h_{u,t}^{\alpha})),$$

the Fubini Theorem used being justified since $\Gamma^{\alpha}$ valued in $L^{2}$ (the sum corresponds to various $j^{\alpha}$ terms). Using boundedness in $M$ of $f_{s,t}^{(N+1,\alpha)}$ and its weak convergence in $M$ proven earlier, we get as above convergence of the second term to the corresponding term without $\alpha$. Using a variant with $\alpha,\beta$ scalar product and previously proven weak convergence in $M$ one gets this is actually $||f_{s,t}^{(N+1)}||_{2}$. As a consequence we deduce norm convergence in $L^{2}$. 
 Similarly, we get formulas (idem without $\tilde{\E}_{1}^{\alpha}$ and $AG_{\alpha}$, or without $G_{\beta}$)~: \begin{align*}\tilde{\E}_{1}^{\alpha}(G_{\beta}(f_{u,t}^{(N+1,\alpha)}))=\sum\int_{s}^{t}du&\ \Re\tau((1+AG_{\alpha})G_{\beta}(\phi_{u-s}^{\alpha}(f_{u,t}^{(N+1,\alpha)*})) G_{\beta}\phi_{u-s}^{\alpha}\Gamma_{\alpha}(f_{s,t}^{\alpha},g_{0,u;t}^{\alpha},h_{s,t}^{\alpha}))\\&+\Re\tau(G_{\beta}(\phi_{u-s}^{\alpha}(f_{u,t}^{(N+1,\alpha)*})) (1+AG_{\alpha})G_{\beta}\phi_{u-s}^{\alpha}\Gamma_{\alpha}(f_{s,t}^{\alpha},g_{0,s;t}^{\alpha},h_{s,t}^{\alpha})).\end{align*}

Now note that the derivative in $u$ of $||G_{\beta}f_{u,t}^{(N+1,\alpha)}||_{2}^{2}$ is~:

$\tilde{\E}_{1}^{\alpha}(G_{\beta}(f_{u,t}^{(N+1,\alpha)}))-||G_{\beta}f_{u,t}^{(N+1,\alpha)}||_{2}^{2}-\sum2\Re\tau(G_{\beta}(f_{u,t}^{(N+1,\alpha)*}) G_{\beta}\Gamma_{\alpha}(f_{u,t}^{\alpha},g_{0,u;t}^{\alpha},h_{u,t}^{\alpha}))$

so that we get~:\begin{align*}\int_{s}^{t}du\  \tilde{\E}_{1}^{\alpha}(G_{\beta}(f_{u,t}^{(N+1,\alpha)}))=&-||G_{\beta}f_{s,t}^{(N+1,\alpha)}||_{2}^{2}+\int_{s}^{t}du\ ||G_{\beta}f_{u,t}^{(N+1,\alpha)}||_{2}^{2}\\\ &+\sum\int_{s}^{t}du\ 2\Re\tau(G_{\beta}(f_{u,t}^{(N+1,\alpha)*}) G_{\beta}\Gamma_{\alpha}(f_{u,t}^{\alpha},g_{0,u;t}^{\alpha},h_{u,t}^{\alpha})).\end{align*}

Taking $\alpha\rightarrow\infty$ (we keep equality at this stage and only use previous $L^{2}$ convergence or techniques already used) and then $\beta\rightarrow\infty$ (using for this second limit first Fatou's lemma to get $\int_{s}^{t}du\liminf_{\beta}\tilde{\E}_{1}(G_{\beta}(f_{u,t}^{(N+1)})))<\infty$ so that by lemma \ref{basic} (ii)' one gets almost  surely $f_{u,t}^{(2)}\in D(\E)$ and thus by (iii) of the same lemma the liminf is actually a lim equal to $\tilde{\E}_{1}((f_{u,t}^{(N+1)})))$, one can then get equality applying DCT, with domination $16(K+1)^{2}$ times the limit now already known to be integrable).  If we first take $\beta\rightarrow\infty$, and then $\alpha\rightarrow\infty$, one thus deduces $$\int_{s}^{t}du\ \tilde{\E}_{1}^{\alpha}(f_{u,t}^{(N+1,\alpha)})\rightarrow \int_{s}^{t}du\ \tilde{\E}_{1}(f_{u,t}^{(N+1)}).$$
Finally (in order to prove $\delta$-convergence) for any $\gamma$,
\begin{align*}\int_{u}^{t}ds\ |\E_{1}^{\alpha}(\hat{G_{\gamma}}f_{s,t}^{(N+1)},f_{s,t}^{(N+1,\alpha)}-f_{s,t}^{(N+1)})|&\leq \int_{0}^{t}ds\ ||A^{*}\hat{G_{\gamma}}f_{s,t}^{(2)}||_{2} ||f_{s,t}^{(N+1,\alpha)}-f_{s,t}^{(N+1)})||_{2}\\ &\leq 2\gamma \int_{0}^{t}||f_{s,t}^{(2)}||_{2}||f_{s,t}^{(N+1,\alpha)}-f_{s,t}^{(N+1)})||_{2}ds ,\end{align*}
also goes to $0$, and 
\begin{align*}\int_{u}^{t}ds\ |\E_{1}^{\alpha}&((\hat{G_{\gamma}}-id)f_{s,t}^{(N+1)}),f_{s,t}^{(N+1,\alpha)}-f_{s,t}^{(N+1)}))|
\\ &\leq 4(K+1)\left(\int_{u}^{t}ds\ \tilde{\E}_{1}((\hat{G_{\gamma}}-id)f_{s,t}^{(N+1)})\right)^{1/2}\left(\int_{u}^{t}ds\ \tilde{\E}_{1}^{\alpha}(f_{s,t}^{(N+1,\alpha)}-f_{s,t}^{(N+1)}))\right)^{1/2},\end{align*}
the second integral is bounded independently of $\alpha$ and the first goes to zero in $\gamma\rightarrow\infty$ by DCT (using proposition 2.13(ii) in \cite{MaR}, earlier stated as lemma \ref{basic} (iii)). The symmetric case $\int_{u}^{t}ds\ |\E_{1}^{\alpha}(f_{u,t}^{(N+1,\alpha)}-f_{u,t}^{(N+1)}),f_{s,t}^{(N+1)}))|\rightarrow 0$ is as easy.
Summing up we get~:
$$\int_{s}^{t}du\ \tilde{\E}_{1}^{\alpha}(f_{u,t}^{(N+1,\alpha)}-f_{u,t}^{(N+1)}))\rightarrow 0.$$

\end{proof}

\subsection{Positivity}
From now on we always consider that $f$ and $h$ are  \emph{$\alpha$-approximated} (N+1)-level-semigroup-families mutually affiliated to $A$ and $A^{*}$. We focus on the two time case and let the (mainly notational) multitime generalization to the reader. Actually, as we explained in the introduction, we learned just before publication of this paper that \cite{JRS} found an alternative construction of a dilation of $\phi_{t-s,\alpha}$ which turned out to be the same as our $\alpha$-approximation (in the symmetric case they consider). This gives a more natural proof of this part using another SDE. We explain here our original proof, with several notational improvements though, with respect to a previous preprint. This proof is mainly combinatorial in nature, but we write the implicit combinatorics in algebraic way for our convenience, letting the reader understand the quite obvious combinatorics behind.

We want to show $\tau(X_{0}^{[0]}\Phi_{0,t}^{(N,\rho,\alpha)}(X_{1}^{[t]}X_{2}^{[0]}...X_{2n-1}^{[t]}))$ define positive linear functionals. Since we will stick to sub-filtrations of free product, by this we mean~:$$\left(\tau(X_{0,i}^{[0]}\Phi_{0,t}^{(N,\rho,\alpha)}(X_{1,i}^{[t]}X_{2,i}^{[0]}...X_{n,i}^{[0/t]}X_{n,j}^{[0/t]*}...X_{1,j}^{[t]*}X_{0,j}^{[0*]}))\right)_{(i,j)}$$ is a positive matrix so that one gets a scalar products on the linear span of elements of the form $X_{0}^{[0]}X_{1}^{[t]}X_{2}^{[0]}...X_{n}^{[0/t]}$ (less than $n+1$ elements starting at time 0). Note that we can deduce from this the various boundedness assumptions in the main theorem of the last part (assuming H and $\Phi$ are built in the same way as in the case we consider when they are mutually associated to $A, A^{*}$). Indeed $\tau(X_{0}^{[0]}\Phi_{0,t}^{(N,\rho,\alpha)}(X_{1}^{[t]}X_{2}^{[0]}...X_{n}^{[0/t]}...X_{1}^{[t]*}X_{0}^{[0*]}))$ then define a state in M in position $X_{n}^{[0/t]}$ thus a standard $C^{*}$ algebraic result gives after an induction this bounded by $\tau(X_{0}^{[0]}X_{0}^{[0*]})||X_{n}^{[0/t]}||...||X_{1}^{[t]}||^{2}$ thus we get value of $\Phi_{0,t}^{(N,\rho,\alpha)}$ in $M$ with the right boundedness by duality.

Consider $\eta: M_0\rightarrow M_2(M_0)$ with $\eta(a)=Diag(G_\alpha(a),\hat{G}_\alpha(a))$ a normal completely positive unital map. Consider $\F(\eta)$ (see e.g. \cite{SAVal99} section 2.4) the full Fock space associated to the canonical Hilbert $M-M$ bimodule $H(M,\eta)$ (e.g.  \cite{SAVal99} lemma 2.2) so that we get to creation operators $L_1=L(\xi_1)$, $L_2=L(\xi_2)$ (with the notations of \cite{SAVal99}) such that (seeing $m\in M$ acting on $\F(\eta)$  by a preservation operator) we have $L_imL_j^*=\eta_{ij}(m)$ i.e. here $L_1mL_1^*=G_\alpha(m)$, $L_2mL_2^*=\hat{G}_\alpha(m)$. (especially $||L_i||=1$) We work on $\M=C^*(M,L_i)$. We define $G_{\alpha,B}(x)=L_1xL_1^*$ a completely positive map on $\M$ $\A_\alpha(x)=\alpha(x-G_{\alpha,B}(x))$ extending $AG_{\alpha}$. By boundedness we can exponentiate in $\phi_{t,\alpha}^B=exp(-\A_{\alpha}t/2)$.
If we write $\Gamma_{\alpha,B}(x,y)=\A_{\alpha}(x)y+x\A_{\alpha}(y)-\A_{\alpha}(xy)=\alpha [x,L_1][L_1^*,y]$, we have as previously~:
$$\phi_{t,\alpha}^{B}(BC)=\phi_{t,\alpha}^{B}(B)\phi_{t,\alpha}^{B}(C)+\int_{0}^{t}ds\ \phi_{s,\alpha}^{B}(\Gamma_{\alpha,B}(\phi_{t-s,\alpha}^{B}(B),\phi_{t-s,\alpha}^{B}(C))).$$

We can get a converging expansion.

\setlength{\hoffset}{-3cm}
\begin{align*}&\phi_{t,\alpha,(i)}^{B}(BC)=\phi_{t,\alpha,(i)}^{B}(B)\phi_{t,\alpha,(i)}^{B}(C)\\ &+\alpha\int_{0}^{t}ds\phi_{s,\alpha,(i+1)}^{B}\left(\left[\phi_{t-s,\alpha,(i)}^{B}(B),L_1\right]\right)\phi_{s,\alpha,(i+1)}^{B}\left(\left[L_1^{*},\phi_{t-s,\alpha,(i)}^{B}(C))\right]\right)\\ &+\alpha^2\int_{0}^{t}ds\int_0^sdu\phi_{u,\alpha,(i+1)}^{B}\Gamma_{\alpha,B}(\phi_{s-u,\alpha,(i+1)}^{B}\left(\left[\phi_{t-s,\alpha,(i)}^{B}(B),L_1\right]\right),\phi_{s-u,\alpha,(i+1)}^{B}\left(\left[L_1^{*},\phi_{t-s,\alpha,(i)}^{B}(C))\right]\right)).\end{align*}
\setlength{\hoffset}{-2cm}

We need a notation to keep the next iterated equation reasonable, thus let us give a name to iterated commutators of $L_1's$, precisely for $s=\{s_{1}\geq...\geq s_{j+1}\}$ with $s_{1}\leq t, s_{j+1}\geq u$,~:$$\CV_{u,s,t}^{j+1}(B)=\alpha^{(j+1)/2}\phi_{s_{j+1}-u,\alpha}^{B}\left(\left[...\phi_{s_2-s_3,\alpha}^{B}\left(\left[\phi_{s_1-s_2,\alpha,}^{B}\left(\left[\phi_{t-s_1,\alpha}^{B}(B),L_1\right]\right),L_1\right]\right)...,L_1\right]\right),$$

and by convention, $$\CV_{u,s,t}^{0}(B)=\phi_{t-u,\alpha}^{B}(B),$$


We get~:\begin{align*}&\phi_{t,\alpha}^{B}(BC)=\phi_{t,\alpha}^{B}(B)\phi_{t,\alpha}^{B}(C)+\sum_{k=0}^\infty\int_{0}^{t}ds_{1}...\int_{0}^{s_{k}}ds_{k+1}\CV_{0,(s_{1},...s_{k+1}),t}^{k+1}(B)\CV_{0,(s_{1},...s_{k+1}),t}^{k+1}(C^{*})^*.\end{align*}

Note that $||\phi_{t,\alpha,(j+1)}^{B}(B)||\leq e^{\alpha t}||B||$ so that $||\CV_{u,s,t}^{j+1}(B)||\leq (4\alpha)^{(j+1)/2}e^{\alpha(t-u)}||B||$ implying convergence of the series. 
 We thus got again a positive like expansion (implying complete positivity of $\phi_{t,\alpha}^{B}$ on $\M$, the remaining part of the proof of positivity will consist in getting such an expression for all the formulas above $\Phi_{0,t}^{N,\alpha}$.
We will also need another notation for our convenience~: 
 $$\CV_{u,s,t}^{(1),j+1}(B,C)=\alpha^{(j+1)/2}\phi_{s_{j+1}-u,\alpha}^{B}\left(\left[...\phi_{s_2-s_3,\alpha}^{B}\left(\left[\phi_{s_1-s_2,\alpha,}^{B}\left(\left[\phi_{t-s_1,\alpha}^{B}(B),L_1\right]C\right),L_1\right]\right)...,L_1\right]\right).$$
  We have obvious analogs with respect to $\hat{G}_\alpha$. 

The next lemma gives an alternative recursive definition where we add derivation from below instead of above (like in the previous part, where this was useful to get limits in taking care of domain issues), so that we can get more easily the decomposition as product of operators in an inductive way. For this, we need a slightly more general (partially) 3-time case, without the (notational) troubles of the full  3-times case and useful for inductions, we write for 
$v,w\leq t$
\begin{align*}&f_{s,[v,t,w]}^{(n,\alpha)}(X_{1}^{[v]},X_{2}^{[0]},X_{3}^{[t]}...X_{2n-3}^{[t]},X_{2n-2}^{[0]},X_{2n-1}^{[w]})\\ &=\sum_{i< j+1=2}^{n}\int_{s}^{v\wedge w}du\ \phi_{u-s,\alpha}(\Gamma_\alpha(f_{u,[v,t,t]}^{(i,\alpha)}(X_{1}^{[v]},...,X_{2i-1}^{[t]}), \\ &\ \ \ \ \ \ \ \ \ \ \ H_{0,u}^{(j-i+1,\alpha)}(X_{2i}^{[u]},F\Phi_{u,t}^{\alpha}(X_{2i+1}^{[t]},...,X_{2j-1}^{[t]}),X_{2j}^{[u]}),f_{u,[t,t,w]}^{(n-j,\alpha)}(X_{2j+1}^{[t]},...,X_{2n-1}^{[w]}))
\end{align*}
(strictly speaking, we have to define the case $n=1$ separately, we only define $f_{s,[v,t,t]}^{(n,\alpha)}= f_{s,[t,t,v]}^{(n,\alpha)}=f_{s,v}^{(1,\alpha)}$)

Analogously, we have extensions to $\M$ we denote $f_{s,[v,t,w],B}^{(n,\alpha)}$ where $\phi_{t,\alpha}$ is replaced by $\phi_{t,\alpha,}^{B}$, $\Gamma_{\alpha}$ by $\Gamma_{\alpha,B}$ etc. The following lemma being also true with the obvious changes for those maps.

\begin{lemma}\label{l11}The following equalities are true for any $r$, for $n\geq 2$:

1.\begin{align*}&f_{s,[v,t,w]}^{(n,\alpha)}(X_{1}^{[v]},X_{2}^{[0]},X_{3}^{[t]}...X_{2n-3}^{[t]},X_{2n-2}^{[0]},X_{2n-1}^{[w]})=\sum_{1\leq i < j\leq n}\int_{s}^{v\wedge (w/1_{j=n})}du\ \\ &f_{s,[u,t,w\vee (t1_{j=n})]}^{(n-j+1,\alpha)}(\Gamma_\alpha(\phi_{v-u,\alpha}(X_{1}),\\ &\ \ \ \ \ H_{0,u}^{(i,\alpha)}(X_{2}^{[u]},F\Phi_{u,t}^\alpha(X_{3}^{[t]}...X_{2i-1}^{[t]}),X_{2i}^{[u]}),f_{u,[t,t,t\wedge (w/1_{j=n})]}^{(j-i,\alpha)}(X_{2i+1}^{[t]}...X_{2j-1}^{[t]}))^{[u]},X_{2j}^{[0]}...X_{2n-1}^{[w]})
\end{align*}

We write \begin{align*}&f_{s,[v,t,w]}^{(n,\alpha)}=f_{s,[v,t,w]}^{(n,\alpha),(t)}=f_{s,[v,t,w]}^{(n,\alpha),(a,l)}+f_{s,[v,t,w]}^{(n,\alpha),(b,l)}
\end{align*} the $(a)$ term corresponding to $l+1\leq j\leq n$ (sum over $i$) of the previous term, $(b)$ the remaining part.

2.\begin{align*}&f_{s,[v,t,w]}^{(n,\alpha)}(X_{1}^{[v]},X_{2}^{[0]},X_{3}^{[t]}..X_{2n-3}^{[t]},X_{2n-2}^{[0]},X_{2n-1}^{[w]})=\sum_{0\leq i < j\leq n-1}\int_{s}^{(v/1_{i=0})\wedge w}du\ \\ &f_{s,[v\vee (t1_{j=n}),t,u]}^{(n-j+1,\alpha)}(X_{1}^{[u]}..X_{2i}^{[0]},\Gamma_\alpha(f_{u\wedge (v/1_{j=n}),t}^{(j-i,\alpha)}(X_{2i+1}^{[t]}..X_{2j-1}^{[t]}),\\ &\ \ \ \ \ \ \ \ H_{0,u}^{(i,\alpha)}(X_{2j}^{[u]},F\Phi_{u,t}^\alpha(X_{2j+1}^{[t]}..X_{2n-3}^{[t]}),X_{2n-2}^{[u]}),\phi_{w-u,\alpha}(X_{2n-1}))^{[u]})
\end{align*}
We write \begin{align*}&f_{s,[v,t,w]}^{(n,\alpha)}=f_{s,[v,t,w]}^{(n,\alpha),(a,r,m)}+f_{s,[v,t,w]}^{(n,\alpha),(b,r,m)}
\end{align*} the $(a)$ term corresponding to $i+m-n<r$ (sum over $j$) of the previous term, $(b)$ the remaining part.

3. \begin{align*}&f_{s,[v,t,w]}^{(n,\alpha)}(X_{1}^{[v]},X_{2}^{[0]},X_{3}^{[t]}...,X_{2n-3}^{[t]},X_{2n-2}^{[0]},X_{2n-1}^{[w]})\\ &=\sum_{\small \begin{matrix}&(i, j,t_i,t_j)\in\{(1,n-1,(t),(t));\\ &(1,k,(t),(a,r,n)),k\leq r;\\ &(l,n-1,(a,r),(t)),l\geq r+1\}\end{matrix}}\int_{s}^{v\wedge w}du\ \\ & \phi_{u-s,\alpha}(\Gamma_\alpha(f_{u,[v,t,t]}^{(i,\alpha),t_i}(X_{1}^{[v]},...,X_{2i-1}^{[t]}),\\ &\ \ \ \ \ \ \  \ \ H_{0,u}^{(j-i+1,\alpha)}(X_{2i}^{[u]},F\Phi_{u,t}(X_{2i+1}^{[t]},...,X_{2j-1}^{[t]}),X_{2j}^{[u]}),f_{u,[t,t,w]}^{(n-j,\alpha),t_j}(X_{2j+1}^{[t]},...,X_{2n-1}^{[w]}))
\\ &+\sum_{1\leq i < j\leq r}\int_{s}^{v}du\ f_{s,[u,t,w]}^{(n-j+1,\alpha)}(\Gamma_\alpha(\phi_{v-u,\alpha}(X_{1}),\\ &\ \ \ \ \ \ H_{0,u}^{(i,\alpha)}(X_{2}^{[u]},F\Phi_{u,t}^\alpha(X_{3}^{[t]}...X_{2i-1}^{[t]}),X_{2i}^{[u]}),f_{u,t}^{(j-i,\alpha)}(X_{2i+1}^{[t]}...X_{2j-1}^{[t]}))^{[u]},X_{2j}^{[0]}...X_{2n-1}^{[w]})
\\ &+\sum_{r<i < j\leq n-1}\int_{s}^{w}du\ f_{s,[v,t,u]}^{(i+1,\alpha)}(X_{1}^{[v]}..X_{2i}^{[0]},\Gamma_\alpha(f_{u,t}^{(j-i,\alpha)}(X_{2i+1}^{[t]}..X_{2j-1}^{[t]}),\\ &\ \ \ \ \ \ \  \ \ H_{0,u}^{(n-j,\alpha)}(X_{2j}^{[u]},F\Phi_{u,t}^\alpha(X_{2j+1}^{[t]}..X_{2n-3}^{[t]}),X_{2n-2}^{[u]}),\phi_{w-u,\alpha}(X_{2n-1}))^{[u]})\\ &-\sum_{1\leq i' < j'\leq r<i < j\leq n-1}\int_{s}^{v}du_1\int_{s}^{w}du_2\ \\ &
f_{s,[u_1,t,u_2]}^{(n,\alpha)}(\Gamma_\alpha(\phi_{v-u_1,\alpha}(X_{1}),H_{0,u_1}^{(i,\alpha)}(X_{2}^{[u_1]},F\Phi_{u_1,t}^\alpha(X_{3}^{[t]},...,X_{2i'-1}^{[t]}),X_{2i'}^{[u_1]}),\\ &\ \ \ \ \ \ \ \ \ \ \ \ f_{u_1,t}^{(j'-i',\alpha)}(X_{2i'+1}^{[t]},...,X_{2j'-1}^{[t]}))^{[u_1]},X_{2j'}^{[0]}...X_{2i}^{[0]},\Gamma_\alpha(f_{u_2,t}^{(j-i,\alpha)}(X_{2i+1}^{[t]},...,X_{2j-1}^{[t]}),\\ &\ \ \ \ \ \ \ \ \ \ \ \ H_{0,u_2}^{(n-j,\alpha)}(X_{2j}^{[u_2]},F\Phi_{u_2,t}^\alpha(X_{2j+1}^{[t]},...,X_{2n-3}^{[t]}),X_{2n-2}^{[u_2]}),\phi_{w-u_2,\alpha}(X_{2n-1})^{[u_2]})
\end{align*}

We write for each term of the above sum~:\begin{align*}&f_{s,[v,t,w]}^{(n,\alpha)}=f_{s,[v,t,w]}^{(n,\alpha),(12-21,r)}+f_{s,[v,t,w]}^{(n,\alpha),(3-321,r)}+f_{s,[v,t,w]}^{(n,\alpha),(123-3,r)}-f_{s,[v,t,w]}^{(n,\alpha),(3-3,r)}.\end{align*}
Note that with our previous notation $f_{s,[v,t,w]}^{(n,\alpha),(3-321,r)}=f_{s,[v,t,w]}^{(n,\alpha),(b,r)}$

\end{lemma}

\begin{proof}
The proof of the first formula by induction on $n$ is quite clear (the case $i=1$ of the defining formula before the lemma 
 gives $j=n$ of this one (this is the only case at initialization $n=2$), for all other  values we apply the induction hypothesis on $f^{(i,\alpha)}$ of the defining formula, in which case in the second formula produced by induction hypothesis $j$ is bellow this $i$ of the first formula). The second formula of the lemma is similar. 

Explained in words, in the first formula, we have written terms where we see a $\phi(X_1)$ and just after application of $\Gamma$. Of course, we have just said that by the inductive definitions of the $f$'s such a way of emphasizing this $\Gamma$ whatever the terms involved around exists.
The second formula is the symmetric with $\phi(X_{2n-1})$
Having in mind a imaginary line $r$ where we intend to cut our formulas into to symmetric and adjoint expressions to get positivity by an explicit rewriting, the two notations introduced with those two formulas emphasized the end of the other side of $\Gamma$ (the side not containing $\phi(X_1)$ or $\phi(X_{2n-1})$) and call with a $(b)$ terms not crossing the imaginary line, with an (a) the terms crossing the imaginary line. 

We now want to use these formulas to prove the third.
Let us explain in words what terms appear in each line.
The second and third line are quite clear and explained at the end of the previous paragraph via the remark at the end of the lemma, since they correspond to terms already appearing previously where the $\Gamma$ of $\phi(X_1)$ or $\phi(X_{2n-1})$ does not cross the imaginary line.
Now of course, this could be possible that both cases happen in the same time, we subtract this in the fourth line. The first line is based on all other terms, but we don't want here to write the terms in emphasizing necessarily a term with $\phi(X_1)$ or $\phi(X_{2n-1})$.

Thus look from above, from the defining formula where $\Gamma$ may involve $f's$ as first and last arguments. Necessarily one of those $f$'s does not cross the imaginary line, if this $f$ is not a $\phi$ it involves bellow another $\Gamma$ and has been taken into account before. Thus, the only remaining terms have one or two $\phi$ (as the first or third argument) at their top level $\Gamma$, and the one not a $\phi$ has to cross the imaginary line. But moreover, we want that this one cannot have at its bottom level $\Gamma$, a $\Gamma$ not crossing the imaginary line, in which case it would have been included already before. Thus we use our $(a)$ terms defined with formulas (1) and (2) of the lemma in order to avoid the bad $(b)$ term.
Before explaining this formally, instead of in words, let us emphasize that in all three last lines (with a $(3)$ in their notations) one Gamma being on one side, of the imaginary line we will cut this imaginary line (in our goal of getting a square like expression explicitly positive) only by induction hypothesis. This lemma consist in emphasizing exactly the configurations of terms for which we will have to work later (to avoid repeating a boring work easily done by induction). Let us come back to the proof.

 For, let us give names to several specific terms depending on the $r$ (position of imaginary line) fixed in the statement in the defining sum $f_{s,[v,t,w]}^{(n,\alpha)}=f_{s,[v,t,w],(11)}^{(n,\alpha)}+f_{s,[v,t,w],(1L)}^{(n,\alpha)}+f_{s,[v,t,w],(L1)}^{(n,\alpha)}+f_{s,[v,t,w],(\ell L)}^{(n,\alpha)}+f_{s,[v,t,w],(L\ell)}^{(n,\alpha)}+f_{s,[v,t,w],(\ell\ell)}^{(n,\alpha)}.$  All those terms correspond to specific values of $i, j$ in the defining sum~: respectively $(11)$ for $i=1, j=n-1$; $(1L)$ for $i=1, j\leq r$; $(L1)$ for $i\geq r+1, j=n-1$; $(\ell L)$ for $1<i\leq r, j\leq r$; $(L\ell)$ for $i\geq r+1, r+1\leq j<n-1$; $(\ell\ell)$ for $i\leq r, j\geq r+1$ and not simultaneously $i=1$ and $j=n-1$. Of course since $i\leq j$ those are the only decomposition possible.

Now we also want to decompose more terms with an $L$, i.e. one side crossing the $r$ limit, we do this using the decomposition of the two first statements of this lemma, in looking at the large L side term $f^{(n-j)}$ or $f^{(i)}$, we get e.g $f_{s,[v,t,w],(L1)}^{(n,\alpha)}=f_{s,[v,t,w],(La1)}^{(n,\alpha)}+f_{s,[v,t,w],(Lb1)}^{(n,\alpha)}$ cutting here $f^{(i)}$ the one crossing $r$ (thinking as i as the  n of the decomposition in the first part of the lemma, n as m and r as the same). We thus got 10 terms. 

In the formula  we have to prove, the $(11)$,$(1La)$,$(La1)$ terms are those of the first line $f_{s,[v,t,w]}^{(n,\alpha),(12-21,r)}$  

The second line $f_{s,[v,t,w]}^{(n,\alpha),(3-321,r)}$ of the new formula corresponds to the sum of $(\ell\ell)$, $(\ell L)$,  $(\ell 1)$, $(Lb1)$, and $(Lb\ell)$. The third line $f_{s,[v,t,w]}^{(n,\alpha),(123-3,r)}$ to $(\ell\ell)$, $(L\ell )$,  $(1\ell)$, $(1Lb)$, and $(\ell Lb)$,  the fourth line $f_{s,[v,t,w]}^{(n,\alpha),(3-3,r)}$ with a  minus term corresponds to  $(\ell\ell)$ $(\ell Lb)$, $(Lb \ell)$, those appearing  in both above, so that they subtracts the redundancy in to the previous lines.

Identifying each line with the stated sum uses only the definitions.
\end{proof}

Let us define the following operators inductively on $n$, formulas for $n\leq N$:
$\sigma_{v,s,[u,t]}^{k}:=\sigma_{v,s,[u,t]}^{k,(1)}+\sigma_{v,s,[u,t]}^{k,(2)}+\sigma_{v,s,[u,t]}^{k,(3)},$
as maps defined 
on $\M$
, $k\geq 1$.
\begin{align*}\sigma_{v,s=(\tilde{s},s',s_{j+1},...,s_{k}),[u,t]}^{k,(1)}&(X_1^{[u]},X_2^{[0]},...,X_n^{[0/t]})=1_{\{s_{j+1}\leq u\}}\sum_{j=0}^{k-1}\sum_{l=0}^{j}\\ &\CV_{v,(s_{j+1},...,s_{k}),u}^{(1),k-j}(X_1,\Theta_{0,s',s_{j+1}}^{j-l}(X_2^{[s_{j+1}]}F\Sigma_{s_{j+1},\tilde{s},t}^{l}(X_3^{[t]};...,X_n^{[0/t]}))),
\end{align*}
\begin{align*}&\sigma_{v,s=(\tilde{s},s_{j+1},...,s_{k}),[u,t]}^{k,(2)}(X_1^{[u]},X_2^{[0]}...X_n^{[0/t]})=
\sum_{i}\sum_{j=1}^{k}\int_{1_{j\neq k}s_{j+1}+1_{j=k}v}^{ \tilde{s}_j\wedge u}du_1\CV_{v,(s_{j+1}...s_{k}),u_1}^{k-j}(\\&\Gamma_{\alpha,B}(\phi_{u-u_1,B}(X_1),H_{0,u_1}^{(i,\alpha)}(X_{2}^{[u_1]},F\Phi_{u_1,t}^{(\alpha)}(X_{3}^{[t]}...X_{2i-1}^{[t]}),X_{2i}^{[u_1]}),\sigma_{u_1,\tilde{s},[t,t]}^{j}(X_{2i+1}^{[t]}...X_{n}^{[0/t]})))
\end{align*}
\begin{align*}&\sigma_{v,s,[u,t]}^{k,(3)}(X_1^{[u]},X_2^{[0]},...,X_n^{[0/t]}) =\sum_{i< j=2}^{\lfloor (n+1)/2\rfloor}\int_{v1_{k=0}+s_k1_{k\neq 0}}^{u }du_1\ \sigma_{v,s,[u_{1},t]}^{k}\\&(\Gamma_{\alpha,B}(\phi_{v-u_1,\alpha}(X_{1}),H_{0,u_1}^{(i,\alpha)}(X_{2}^{[u_1]},F\Phi_{u_1,t}^{(\alpha)}(X_{3}^{[t]}...),X_{2i}^{[u_1]}),f_{u_1,t}^{(j-i,\alpha)}(..X_{2j-1}^{[t]}))^{[u_1]},X_{2j}^{[0]}..X_n^{[0/t]})\end{align*}

In this definition, we have written ($l\geq 1$) $$\Sigma_{u,s',t}^{l}(X_1^{[t]};...,X_n^{[0/t]}):=\sum_{i=1}^{\lfloor(n-1)/2\rfloor} \Phi_{u,t}^{\alpha}(X_{1}^{[t]},...,X_{2i-1}^{[t]})X_{2i}^{[0]}\sigma_{u,s',t}^{l}(X_{2i+1}^{[t]};...,X_n^{[0/t]}))^{[u]},$$

and for ($l=0$)$$\Sigma_{u,s',t}^{0}(X_1^{[t]};...,X_n^{[0/t]}):= \Phi_{u,t}^{\alpha}(X_{1}^{[t]},...,X_{n}^{[0/t]}).$$

More crucially, we need to assume given a corresponding decomposition (obtained inductively) for $H$ of the type we will prove bellow for $\Phi$, i.e. we assume there are compact spaces with fixed (positive) Radon measures $T_{s}^{k,H}$ included in a locally compact $I_{k,H}$ such that $\mu(T_{s}^{k,H})$ increases with $s$, we have the infinite radius of convergence condition~: for any $K>0$ $\sum_k\mu(T_{s}^{k,H})K^k<\infty$. We also have constants $C, K$ (maybe depending on $N, u$ but not $k,s$) such that $||\Theta_{0,s,u}^{k}(X_1^{[u]},X_2^{[0]},...,X_n^{[0/u]})||_\M \leq C K^{k}||X_1||_\M...||X_n||_\M$ and for all $i=n+p\leq N-1$, $X_{i}\in \M$,  
\begin{align*}H_{0,u}^{(i,\alpha)}&(X_{1}^{[u]},X_{2}^{[0]},X_{3}^{[u]}...,X_n^{[0/u]},X_p'^{[u/0]*}...,X_2'^{[0]*},X_1'^{[u]*})\\ &=\sum_{k=0}^{\infty}\int_{T_{u}^{k,H}}ds \ \Theta_{0,s,u}^{k}(X_1^{[u]},X_2^{[0]},...,X_n^{[0/u]})(\Theta_{0,s,u}^{k}(X_1'^{[u]},X_2'^{[0]},...,X_p'^{[u/0]}))^{*}.\end{align*} 
The previous conditions especially imply this series converges absolutely.
We also assume that the maps (noted identically for different $n$) $\Theta_{0,.,.}^{k}:\{(u,s)\in I_{k,H}\times \R_{+} \ | \ u\in T_{s}^{k,H}\}\rightarrow B(\M^{\hat{\o} n},\M))$
 is continuous and agrees by restriction of $i$ (so that the various extensions of H, also assumed given, do agree).
We of course assume $T_{s}^{0,H}$ is a point and $\Theta_{0,s,u}^{j,0}(X_1^{[u]},X_2^{[0]},...,X_n^{[0/u]})=H_{0,u}^{(\alpha),j}(X_1^{[u]},X_2^{[0]},...,X_n^{[0/u]})$. Note that the previous properties are obvious with our previous expressions (knowing the next lemma to interpret them). 

We can now prove the positivity decomposition lemma~:

\begin{lemma}
We have the following relations (for $i=n+p\leq N$):
\begin{align*}&f_{v,[u,t,w]}^{((n+p+1)/2,\alpha)}(X_{1}^{[v]},X_{2}^{[0]},X_{3}^{[t]}...,X_n^{[0/t]},X_p'^{[t/0]*}...,X_2'^{[0]*},X_1'^{[w]*})\\ &=\sum_{k=1}^{\infty}\int_{T_{v,t}^k}ds\ \sigma_{v,s,[u,t]}^{k}(X_1^{[u]},X_2^{[0]},...,X_n^{[0/t]})(\sigma_{v,s,[w,t]}^{k}(X_1'^{[w]},X_2'^{[0]},...,X_p'^{[t/0]}))^{*}\\ &
f_{v,[u,t,w]}^{((n+p+1)/2,\alpha),(123-3)}(X_{1}^{[v]},X_{2}^{[0]},X_{3}^{[t]}...,X_n^{[0/t]},X_p'^{[t/0]*}...,X_2'^{[0]*},X_1'^{[w]*})\\ &=\sum_{k=1}^{\infty}\int_{T_{v,t}^k}ds\ \sigma_{v,s,[u,t]}^{k}(X_1^{[u]},X_2^{[0]},...,X_n^{[0/t]})(\sigma_{v,s,[w,t]}^{k,(3)}(X_1'^{[w]},X_2'^{[0]},...,X_p'^{[t/0]}))^{*}
\end{align*}
where \begin{eqnarray*}T_{v,u}^k &~:=  \sqcup_{l}\sqcup_{0< i_{1}< j_{1}< i_{2}< ...j_{l-1}< i_{l}= k} \  & \{ (s_{1},...s_{i_1},s'_{1},s_{j_1+1},...s_{i_2},s'_{2},...s_{i_l})\ \\  & &\ \ | \ u\geq s_{1},\ i\mapsto s_{i}\  \text{non-increasing},\ s_{i_l}\geq v,\  s'_m\in T_{0,s_{i_m}}^{j_m-i_m,H}\}\\ \subset I_{k}&~:=  \sqcup_{l}\sqcup_{0< i_{1}< j_{1}< i_{2}< ...j_{l-1}< i_{l}= k}\  & \R_+^{i_1}\times I_{j_1-i_1,H}\times...\R_+^{i_l-j_{l-1}}
\end{eqnarray*}
Moreover, for the canonical product measure (using Lebesgue measure for intervals in $\R$), $\mu(T_{a,b}^k)\leq \mu(T_{a',b'}^k)$ for $a'\leq a\leq b\leq b'$ and we have the infinite radius of convergence condition~: for any $K>0$ $\sum_k\mu(T_{a,b}^k)K^k<\infty$. We also have constants $C, K$ (maybe depending on $N, t$ but not $k$) such that $||\sigma_{v,s,[u,t]}^{k}(X_1^{[u]},X_2^{[0]},...,X_n^{[0/t]})||_\M \leq C K^{k}||X_1||_\M...||X_n||_\M$, so that the previous sums converge absolutely.
\end{lemma}

\begin{proof}

Let us first prove the infinite radius of convergence result~:
\begin{align*}&\mu(T_{a,b}^k):=\\ &\sum_{l,0< i_{1}< j_{1}< i_{2}< ...j_{l-1}< i_{l}= k}\int_a^b ds_1...\int_a^{s_{i_1-1}} ds_{i_1}\int_{ T_{0,s_{i_1}}^{j_1-i_1,H}}ds'_1\int_a^{s_{i_1}} ds_{j_1+1}...\int_a^{s_{i_2-1}} ds_{i_2}...\int_a^{s_{i_l-1}} ds_{i_l},\end{align*}

so that this is obviously increasing with $a$ $b$ as stated in the theorem and~: \begin{align*}&\sum_k K^k\mu(T_{a,b}^k)\sum_{k,l,0< i_{1}< j_{1}< i_{2}< ...j_{l-1}< i_{l}= k}K^k\times\leq\\ &\times\int_a^b ds_1...\int_a^{s_{i_1-1}} ds_{i_1}\int_a^{s_{i_1}} ds_{j_1+1}...\int_a^{s_{i_2-1}} ds_{i_2}...\int_a^{s_{i_l-1}} ds_{i_l}\prod_m\mu(T_{0,b}^{j_m-i_m,H})\\ &\leq\sum_{k,l,0< i_{1}< j_{1}< i_{2}< ...j_{l-1}< i_{l}=k}K^k\frac{(b-a)^{k-\sum_m(j_m-i_m)}}{(k-\sum_m(j_m-i_m))!}\prod_{m=1}^{l-1}\mu(T_{b}^{j_m-i_m,H})\\ &\leq\sum_{k',l\leq k',j_{1}',...j_{l}'}(2K)^{k'+(j_1)'+...+(j_{l-1})'}\frac{(b-a)^{k'}}{k'!}\prod_{m=1}^{l-1}\mu(T_{b}^{j'_m,H})\\ &\leq\sum_{k'}(2K)^{k'}\frac{(b-a)^{k'}}{k'!}\left(1+\sum_j(2K)^j\mu(T_{b}^{j,H})\right)^{k'}<\infty, \end{align*}

In the fourth line we noted $k'$ (the previous power of $(b-a)$) is automatically above $l$ since the sequence is (strictly) increasing, and moreover, the number of terms in the sum over sequences is always bellow the number of parts of $[1,k]$ (a sequence being the positions of changes of colors) thus a bound in $2^k$.The bound on $\sigma$ is obvious as seen formulas and by induction on $n+p$)

The end of the proof is by induction on $n+p$. First note that the second formula is obvious by induction, as is obvious an analogue formula for $f_{v,[u,t,w]}^{(n,\alpha),(3-3)}$. Those terms have been put into emphasis before, especially because the induction is easy on them. 

It thus remains to check the formula for $f_{v,[u,t,w]}^{(n,\alpha),(12-21)}$. In the defining sum, the term $i=1, j=n-1$ ( let's call the corresponding sum $(1)$) is almost obvious either by induction or by our assumption on $H$. However, let us note that we need to use the decomposition for $H$ but with one middle term maybe of the form $(XY)^{[u]}$ where $X$ and $Y$ are two terms coming from a same $f(X_{i},...X_k)$ divided in two $\sigma$'s. But since at the final level we cut in the middle between a $0$ and a $t$, this actually happens only if $k\geq i+2$, so that the number of terms in the alternating pattern on which we apply $H$ being less than two less than $N-1$ the maximal point of our knowledge for $H$, we can add an alternating pattern $X^{[u]}1^{[0]}Y^{[u]}$ to get the wanted decomposition by the assumption. (Of course, $H$ like $\Phi$ is compatible with these insertions of $1$'s). One then decomposes $\Gamma$ and $\phi$ as above (in the decomposition for $\phi$) to get a $\CV$ term. In a shortened formula (without arguments easily deducible by the reader), one gets (sum-sum and sum-integral exchanges being justified by  summability either assumed or inductively proven)~:

\begin{align*}&(1)=\sum_{k=1}^{\infty}\sum_{j=0}^{k-1}\sum_{l=0}^{j}\int_{v}^{u\wedge w}ds_{j+1}...\int_{v}^{s_{k-1}}ds_{k}\int_{T^{j-l,H}_{0,s_{j+1}}}ds'\int_{T^{l}_{s_{j+1},t}}d\tilde{s} \\ &\CV_{v,(s_{j+1},...s_{k}),u}^{(1),k-j}(X_1,\Theta_{0,s',s_{j+1}}^{j-l}(F(\Sigma_{s_{j+1},\tilde{s},t}^{l})))\CV_{v,(s_{j+1},...s_{k}),w}^{(1),k-j}(X_1',(\Theta_{0,s',s_{j+1}}^{j-l}F( \Sigma_{s_{j+1},\tilde{s},t}^{l} )))^*.
\end{align*}

This sums thus takes into account a part of the sum of $\sigma^{(1),k}(\sigma^{(1),k})^*$. Precisely, when $j\neq l$, the disjointness of integration areas removes the double sum corresponding to each  $\sigma^{(1),k}$ (actually the double double sum, one double sum for $j$, one for $l$,the disjointness of integration areas comes from the fact we take a disjoint union over list of integers corresponding to the number of $\Gamma$'s cut in between both sides of the $\sigma\sigma^*$ product, and certain of those integers also roughly corresponds to $j,l$ and $k$'s at least in the case we consider here) and reduces it in (two, one for $j$, one for $l$) simple sums as above. All those terms are thus taken into account here. When $j=l$, the double sum over $l$ cannot be removed (contrary to what happened in the first case, in absence of term coming from $H$, the k-l terms are only a part of the last bunch ot variables $s_{j_{l-1}+1},...,s_{i_l}$ in the space of integration $T$, exactly those corresponding to $\CV$'s above $X_{1}$'s we can always emphasize those variables and we only have here the part with the same number of variable of that kind for $X_1$ and $X_1'$), and we only have here the diagonal part of it, the second part will appear later. In other words, with obvious notations, we have obtained~:

\begin{align*}&(1)=\sum_{k=1}^{\infty}\int_{T_{v,u\wedge w}^k}ds\ \sigma_{v,s,[u,t]}^{J,k,(1,l\neq j)}(\sigma_{w,s,[w,t]}^{J,k,(1,l\neq j)})^{*}\\&+\sum_{k=1}^{\infty}\sum_{l(=j)=0}^{k-1}\int_{v}^{u\wedge w}ds_{j+1}...\int_{v}^{s_{k-1}}ds_{k}\int_{T^{0,H}_{v,s_{j+1}}}ds'\int_{T^{l}_{s_{j+1},t}}d\tilde{s} \\ &\CV_{v,(s_{j+1},...s_{k}),u}^{(1),k-l}(X_1,\Theta_{0,s',s_{j+1}}^{0}(F(\Sigma_{s_{j+1},\tilde{s},t}^{l})))\CV_{v,(s_{j+1},...s_{k}),w}^{(1),k-l}(X_1',(\Theta_{0,s',s_{j+1}}^{0}F( \Sigma_{s_{j+1},\tilde{s},t}^{l} )))^*
,
\end{align*}
where we may have also used the form of $\Theta^{0}$~:
$$\Theta_{0,s',s_{j+1}}^{0}(F(\Sigma_{s_{j+1},\tilde{s},t}^{l}))=H_{0,s_{j+1}}(F(\Phi_{s_{j+1},t}^{J,l}))(\sigma_{s_{j+1},\tilde{s},t}^{l}).$$

Let us now discuss the part, say $(2)$, indexed by $(1,k,(t),(a,r,n)),k\leq r$ of the defining sum of $f_{v,[u,t,w]}^{(n,\alpha),(12-21)}$. Here we obviously want to apply induction hypothesis to $f_{u,[t,t,w]}^{(n-j,\alpha),t_j}$.

As we have noted,  $f_{s,[v,t,w]}^{(n,\alpha),(3-321,r)}=f_{s,[v,t,w]}^{(n,\alpha),(b,r)}$ so that $f_{s,[v,t,w]}^{(n,\alpha),(a,r)}$ decomposes in a $\sigma\sigma^{1}+\sigma\sigma^{2}$ pattern. 

We thus obtain (using of course a derivation property for $\Gamma$):
\begin{align*}&(2)=\sum_{j=1}^{\infty}\sum_{k=j+1}^{\infty}\int_{v}^{u\wedge w}ds_{j+1}...\int_{v}^{s_{k-1}}ds_{k}\int_{T^{j}_{s_{j+1},t}}d\tilde{s} \\ &\CV_{v,(s_{j+1},...s_{k}),u}^{(1),k-j}(X_1,\left(H_{0,s_{j+1}}(F(\Phi_{s_{j+1},t}))\sigma^{j}_{s_{j+1},\tilde{s}[t,t]}\right))(\CV_{v,(s_{j+2},...s_{k}),s_{j+1}}^{k-j-1}([\sigma^{j,(1+2)}_{s_{j+1},\tilde{s},[w,t]},L_{1}]))^*+\\ &\CV_{v,(s_{j+2},...s_{k}),s_{j+1}}^{k-j-1}(\Gamma_{\alpha,B}\left(\phi_{u-s_{j+1},\alpha}(X_1),H_{0,s_{j+1}}(F(\Phi_{s_{j+1},t})),\sigma^{j}_{s_{j+1},\tilde{s},[t,t]}\right))\times\\& \ \ \ \ \ \ \ \ \ \  \ \ \ \ \ \ \ \ \   \ \ \ \ \ \ \  \times (\CV_{v,(s_{j+2},...s_{k}),s_{j+1}}^{k-1-j}(\sigma^{j,(1+2)}_{s_{j+1},\tilde{s},[w,t]}))^*.
\end{align*}

The first part of the first term (with a (1) in the right) exactly gives the lacking component of the (1-1) term, or at least the part where the number of $L_1$'s on $X_1'$ is greater than the one on $X_1$ (of course in the $l=j$ case for the $\sigma^{(1)}$ of the left, nothing being imposed on the right, except a minimal number of $\CV$'s).

The term of the first line with (2) and of the second line with (1) corresponds to parts of (1-2) and (2-1) terms respectively, again the global sum being divided depending on the number of variables of integration in the above bunch of them fixed on $X_1$ or $X_1'$ respectively. 
More precisely, note we get all the terms we need since when $j_r\neq l_r$ in the side of $\sigma^{(1)}$ (assuming it on the right to fix notation) with the terminology of the defining sum, we have to have an higher $j_l> j_r$ in the side of $\sigma^{(2)}$
, this appears thus in the second line of our formula for (2). If $j_r= l_r$, we have as above two cases $j_l> j_r$ appears again likewise in the second line, and $j_r\geq j_l$ in the first line (of the corresponding term  $(3)$ of $f^{(12-21)}$ , here appears the symmetric case with inversion of right left positions of $\sigma^{(2)}$  and $\sigma^{(1)}$). We discuss the appearance of the integral in variable $u_1$ in definition of $\sigma^{(2)}$ bellow.

Finally,  the (2) term in the last line corresponds to  part of the (2-2) term. Of course this contains the case with fewer many terms in the $\CV$ of the left (i.e. $j_l>j_r$ if those are the $j$'s of the definitions for $\sigma^{(2)}$ on left and right), but also, half of the case of the same number of terms (case $j_l=j_r$, i.e. $j=k$ in the defining sum of the right $\sigma^{(2)}$ of the above expression (and not any more of the concluding expression we want it to be equal)). In all cases the variable $s_{j+1}$ (almost) does not appear on the right (after applying the composition of two $\phi$'s one coming from $\CV$, one from $\sigma$, note that this simplification also happened in the $(1-2)$ case). But we said almost since it appears in the integral of the case $j=k$ in the defining sum of the right $\sigma^{(2)}$.  Actually, in the definition of $\sigma^{(2)}$, it is called $u_1$, and we have to apply Fubini Theorem (for continuous functions) in our expression for (2) above to get an integral of the form $\int_{v}^{u\wedge w}ds_{j+2}...\int_{v}^{s_{k-1}}ds_{k}\int_{T^{l}_{v,t}}d\tilde{s} \int_{s_{j+2}1_{j\neq k-1}+v1_{j=k-1}}^{\tilde{s}_j}ds_{j+1}$.
In the case $j_l=j_r$ as we said ($j=k$ in the other notation), we only get the part $\int\ du_l\int_{u_r\geq u_l} du_r$ of the integral (here $u_r$ for instance is the $u_1$ occurring in the definition of the right $\sigma^{(2)}$). Of course the symmetric case (coming from the third term $(3)$ of $f^{(12-21)}$) will give the second half of this integral, so that we get the last lacking term.
\end{proof}

\begin{remark}
We can now prove the required boundedness property of $f_{s,t}^{\alpha}$ (beyond $s=0$). Indeed, we can apply the reasoning of the beginning of the section after having noticed~: \begin{align*}\tau(X_{0}^{[u]}\Phi_{u,t}^{(\alpha)}(X_{1}^{[t]}X_{2}^{[0]}...X_{2n-1}^{[t]})X_{2n}^{[u]})&:=\tau(\Phi_{0,u}^{\alpha}(X_{0}^{[u]}\Phi_{u,t}^{(\alpha)}(X_{1}^{[t]}X_{2}^{[0]}...X_{2n-1}^{[t]})X_{2n}^{[u]}))\\ &=\tau(\Phi_{0,[u,t,u]}^{(\alpha)}(X_{0}^{[u]}X_{1}^{[t]}X_{2}^{[0]}...X_{2n-1}^{[t]}X_{2n}^{[u]})\end{align*} and moreover is of the form $$\tau(X_{0}^{[u]}\Phi_{u,t}^{(\alpha)}(X_{1}^{[t]}X_{2}^{[0]}...X_{2n-1}^{[t]})X_{2n}^{[u]})=\tau(X_{0}f_{u,t}^{(\alpha)}(X_{0}^{[u]}X_{1}^{[t]}X_{2}^{[0]}...X_{2n-1}^{[t]})X_{2n}))+ \text{ terms of smaller order}.$$ Of course this boundedness can also be proven via the complete multitime case, this being just said to have an (almost) completely written proof in the two times case.
\end{remark}

\subsection{Symmetry and traciality}

We keep notations of the previous part 2.3.

Let us start with the symmetry for the two times case~: $\tau\circ\Phi_{0,t}^{\alpha}=\tau\circ\Phi_{0,t}^{\alpha}\circ S_{t}$ in the symmetric semigroup case. As we will see, the general case will follow easily. (of course we can also get the case without $\alpha$ using limits proven in the previous part, but we will stick to that case in computations for boundedness reasons.) In the non-symmetric Dirichlet form context, we will prove $\tau\circ\Phi_{0,t}=\tau\circ\H_{0,t}\circ S_{t}$, but the relation won't be valid at level $\alpha$. 

Since indices are enough to uniquely identify equations, we don't write here arguments ($X_i$'s of the previous formulas). Let us first note the following alternative equation for $f_{s,t}^{(n,\alpha)}$:

\begin{equation}\label{alternate}f_{s,t}^{(n,\alpha)}=\sum_{l,i_1+j_1+...+j_{l-1}+i_l-l=n-1}\int_{s}^{t}du \ \phi_{u-s,\alpha}\Gamma_{\alpha}(f_{u,t}^{(i_{1},\alpha)},h_{t-u,t}^{(j_{1},\alpha)}f_{u,t}^{(i_{2},\alpha)}h_{t-u,t}^{(j_{2},\alpha)}...h_{t-u,t}^{(j_{l-1},\alpha)},f_{u,t}^{(i_{l},\alpha)}).\end{equation}

Indeed, we have a priori, in the middle of the original definition, to apply $H_{0,u}$  to several $f_{u,t}^{(g_{1},\alpha),[0]}X_{2k_{1}}^{[u]}f_{u,t}^{(g_{2},\alpha),[0]}...f_{u,t}^{(g_{p},\alpha),[0]}$, itself determined by a sum of products of the form $h_{0,u}^{(j_{p}',\alpha)}(X_{2k_{1}}^{[u]},f_{u,t}^{(g_{1},\alpha),[0]}X_{2k_{2}}^{[u]}f_{u,t}^{(g_{2},\alpha),[0]}...f_{u,t}^{(g_{q-1},\alpha),[0]},X_{2k_{q}}^{[u]})$, the point is that one can show by a translation of variable by $t-u$ and by induction this equals (after the right summation) a corresponding term of  $h_{t-u,t}^{(j_{p},\alpha)}$. For the reader's convenience, let us write this formally in the following~:
\begin{lemma}\label{symmsemig}
\begin{align*}h_{t-u,t}&^{(j_{p},\alpha)}(X_{2k_{1}},...,X_{2k_{1}'})\\ &=\sum_{\small\begin{matrix}&q,g_i\geq 1,\\ &g_1+...+g_q+1=j_p,\\ &k_{i+1}-k_i=g_i\end{matrix}} h_{0,u}^{(q+1,\alpha)}(X_{2k_{1}}^{[u]},f_{u,t}^{(g_{1},\alpha),[0]}X_{2k_{2}}^{[u]}f_{u,t}^{(g_{2},\alpha),[0]}...f_{u,t}^{(g_{q},\alpha),[0]},X_{2k_{q+1}}^{[u]})\end{align*}
\end{lemma}
\begin{proof}We carry an overall induction over $j_p$ (and prove simultaneously the analogue for $f$). Initialization is tautological. By formula (\ref{alternate}) (the variant for $h$ derived from induction hypothesis thanks to the above reasoning), we have~:
\begin{align*}h_{t-u,t}^{(j_{p},\alpha)}&(X_{2k_{1}},...,X_{2k_{1}'})\\ &=\sum_{l,j_1'+i_1'+...+i_{l-1}'+j_l'-l=j_{p}-1}\int_{t-u}^{t}dv \ \phi_{v-(t-u),\alpha}^{*}\hat{\Gamma}_{\alpha}(h_{v,t}^{(j_{1}',\alpha)},f_{t-v,t}^{(i_{1}',\alpha)}h_{v,t}^{(j_{2}',\alpha)}f_{t-v,t}^{(i_{2}',\alpha)}...f_{t-v,t}^{(i_{l-1}',\alpha)},h_{v,t}^{(j_{l}',\alpha)}).\end{align*}
As we said, we take the new variable $v'=v-(t-u)$ to get~:
\begin{align*}&h_{t-u,t}^{(j_{p},\alpha)}(X_{2k_{1}},...,X_{2k_{1}'})\\ &=\sum_{l,j_1'+i_1'+...+i_{l-1}'+j_l'-l=j_{p}-1}\int_{0}^{u}dv' \ \phi_{v',\alpha}^{*}\hat{\Gamma}_{\alpha}(h_{v'+(t-u),t}^{(j_{1}',\alpha)},f_{u-v',t}^{(i_{1}',\alpha)}h_{v'+(t-u),t}^{(j_{2}',\alpha)}f_{u-v',t}^{(i_{2}',\alpha)}...f_{u-v',t}^{(i_{l-1}',\alpha)},h_{v'+(t-u),t}^{(j_{l}',\alpha)}).\end{align*}
(To keep coherent notations, let us write $l_i$ defined such that $l_1=k_1$ and $l_{m+1}-l_m=j_m'+i_m'-1$, assuming fixed a sum as above).

Applying once again induction hypothesis, one gets 
\begin{align*}&h_{v'+(t-u),t}^{(i_{l}',\alpha)}\\ &=\sum_{\small \begin{matrix}&q'_l,g_{i,l}\geq 1, m_{1,l}=l_{l}\\ &g_{1,l}+...+g_{q'_l,l}+1=i_l'\\ &m_{i+1,l}-m_{i,l}=g_{i,l}\end{matrix}} h_{0,u-v'}^{(q'_l+1,\alpha)}(X_{2m_{1,l}}^{[u-v']},f_{u-v',t}^{(g_{1,l},\alpha),[0]}X_{2m_{2,l}}^{[u-v']}f_{u-v',t}^{(g_{2,l},\alpha),[0]}...f_{u-v',t}^{(g_{q,l},\alpha),[0]},X_{2m_{q'_l+1,l}}^{[u-v']})
\end{align*}

Now, we have to use the semigroup property for $$f_{u-v',t}^{(g_{j,l},\alpha)}=\sum_{\small \begin{matrix}&q''_{j,l}\\ &g_{i,l}\geq 1,g_{1,j,l}''+...+g_{q''_{j,l},j,l}''=g_{j,l}\\ &n_{i+1,j,l}-n_{i,j,l}=g_{i,j,l}, n_{0,j,l}=m_{j,l}\end{matrix}} f_{u-v',u}^{(q_{j,l}''+1,\alpha)}(f_{u,t}^{(g_{1,j,l}'',\alpha)}X_{2n_{1,j,l}}...f_{u,t}^{(g_{q_{j,l}'',j,l}'',\alpha))}).$$ 

We can now reapply induction hypothesis with u instead of t and u-v' instead of u, 0 instead of v' to get
\begin{align*}&h_{v'+(t-u),t}^{(i_{l}',\alpha)}\\ &=\sum_{\small \begin{matrix}&q'_l,g_{i,l}, m_{1,l}=l_{l}\geq 1\\ &g_{1,l}+...+g_{q'_l,l}+1=i_l'\\ &m_{i+1,l}-m_{i,l}=g_{i,l}\end{matrix}} h_{v',u}^{(q'_l+1,\alpha)}(X_{2m_{1,l}}^{[u]},f_{u,t}^{(g_{1,l},\alpha),[0]}X_{2m_{2,l}}^{[u]}f_{u,t}^{(g_{2,l},\alpha),[0]}...f_{u,t}^{(g_{q,l},\alpha),[0]},X_{2m_{q'_l+1,l}}^{[u]}).
\end{align*}
Putting together this expression the previous one and the second relation of the proof for the searched quantity, this readily concludes (also reusing the definition of $h_{0,u}$).
\end{proof}

 We start by proving several relations coming from differentiation of $\tau(f_{s,t}^{(i_{1},\alpha)}h_{t-s,t}^{(j_{1},\alpha)}...f_{s,t}^{(i_{l},\alpha)}h_{t-s,t}^{(j_{l},\alpha)})$ in $s$ (computing these expressions is maybe motivated by the previous remark). Since every maps involved in the $\alpha$-case are bounded, we readily get the following~:
\begin{lemma}
The derivative in $s$ of $\tau(f_{s,t}^{(i_{1},\alpha)}h_{t-s,t}^{(j_{1},\alpha)}...f_{s,t}^{(i_{l},\alpha)}h_{t-s,t}^{(j_{l},\alpha)})$ is~:
\begin{align*}&\frac{1}{4}\sum_{k}\tau(f_{s,t}^{(i_{1},\alpha)}...(AG_{\alpha}-A^{*}G_{\alpha}^{*})(f_{s,t}^{(i_{k},\alpha)})...h_{t-s,t}^{(j_{l},\alpha)})\\ &+\frac{1}{4}\sum_{k}\tau(f_{s,t}^{(i_{1},\alpha)}...(AG_{\alpha}-A^{*}G_{\alpha}^{*})(h_{t-s,t}^{(j_{k},\alpha)})...h_{t-s,t}^{(j_{l},\alpha)})
\\&+\frac{1}{2}\sum_{k<k'}\tau(f_{s,t}^{(i_{1},\alpha)}...(\Gamma_{\alpha}+\hat{\Gamma}_{\alpha})(f_{s,t}^{(i_{k},\alpha)},..,f_{s,t}^{(i_{k'},\alpha)})..h_{t-s,t}^{(j_{l},\alpha)})\\&-\frac{1}{2}\sum_{k<k'}\tau(f_{s,t}^{(i_{1},\alpha)}..(\Gamma_{\alpha}+\hat{\Gamma}_{\alpha})(h_{t-s,t}^{(j_{k},\alpha)},...,h_{t-s,t}^{(j_{k'},\alpha)})...h_{t-s,t}^{(j_{l},\alpha)})\\&+\sum_{k,p}\sum_{j_{k,1}+i_{k,1}+..+j_{k,p}=j_k-1+p>1}\\ &\tau(f_{s,t}^{(i_{1},\alpha)}h_{t-s,t}^{(j_{1},\alpha)}..\hat{\Gamma}_{\alpha}(h_{t-s,t}^{(j_{k,1},\alpha)},f_{s,t}^{(i_{k,1},\alpha)}..f_{s,t}^{(i_{k,p-1},\alpha)},h_{t-s,t}^{(j_{k,p},\alpha)})f_{s,t}^{(i_{k+1},\alpha)}..f_{s,t}^{(i_{l},\alpha)}h_{t-s,t}^{(j_{l},\alpha)})\\ &-\sum_{k,p}\sum_{i_{k,1}+j_{k,1}+..+i_{k,p}=i_k-1+p>1}\\ &\tau(f_{s,t}^{(i_{1},\alpha)}h_{t-s,t}^{(j_{1},\alpha)}..\Gamma_{\alpha}(f_{s,t}^{(i_{k,1},\alpha)},h_{t-s,t}^{(j_{k,1},\alpha)}...h_{t-s,t}^{(j_{k,p-1},\alpha)},f_{s,t}^{(i_{k,p},\alpha)})h_{t-s,t}^{(j_{k},\alpha)}..f_{s,t}^{(i_{l},\alpha)}h_{t-s,t}^{(j_{l},\alpha)})\end{align*}

\end{lemma}

\begin{proof}
The very definition (and formula (\ref{alternate}) above) gives the following derivative in s of $\tau(f_{s,t}^{(i_{1},\alpha)}h_{t-s,t}^{(j_{1},\alpha)}...f_{s,t}^{(i_{l},\alpha)}h_{t-s,t}^{(j_{l},\alpha)})$:\begin{align*}&\frac{1}{2}\sum_{k}\tau(f_{s,t}^{(i_{1},\alpha)}h_{t-s,t}^{(j_{1},\alpha)}..AG_{\alpha}(f_{s,t}^{(i_{k},\alpha)})h_{t-s,t}^{(j_{k},\alpha)}..f_{s,t}^{(i_{l},\alpha)}h_{t-s,t}^{(j_{l},\alpha)})\\ &-\frac{1}{2}\sum_{k}\tau(f_{s,t}^{(i_{1},\alpha)}h_{t-s,t}^{(j_{1},\alpha)}..f_{s,t}^{(i_{k},\alpha)}A^{*}G_{\alpha}^{*}(h_{t-s,t}^{(j_{k},\alpha)})...f_{s,t}^{(i_{l},\alpha)}h_{t-s,t}^{(j_{l},\alpha)})\\&+\sum_{k,p}\sum_{j_{k,1}+i_{k,1}+..+j_{k,p}=j_k-1+p>1}\\ &\tau(f_{s,t}^{(i_{1},\alpha)}h_{t-s,t}^{(j_{1},\alpha)}..\hat{\Gamma}_{\alpha}(h_{t-s,t}^{(j_{k,1},\alpha)},f_{s,t}^{(i_{k,1},\alpha)}..f_{s,t}^{(i_{k,p-1},\alpha)},h_{t-s,t}^{(j_{k,p},\alpha)})f_{s,t}^{(i_{k+1},\alpha)}..f_{s,t}^{(i_{l},\alpha)}h_{t-s,t}^{(j_{l},\alpha)})\\ &-\sum_{k,p}\sum_{i_{k,1}+j_{k,1}+..+i_{k,p}=i_k-1+p>1}\\ &\tau(f_{s,t}^{(i_{1},\alpha)}h_{t-s,t}^{(j_{1},\alpha)}..\Gamma_{\alpha}(f_{s,t}^{(i_{k,1},\alpha)},h_{t-s,t}^{(j_{k,1},\alpha)}..h_{t-s,t}^{(j_{k,p-1},\alpha)},f_{s,t}^{(i_{k,p},\alpha)})h_{t-s,t}^{(j_{k},\alpha)}..f_{s,t}^{(i_{l},\alpha)}h_{t-s,t}^{(j_{l},\alpha)})\end{align*}

Now, we can compute the first line 
\begin{align*}&\frac{1}{2}\sum_{k}\tau(f_{s,t}^{(i_{1},\alpha)}h_{0,s}^{(j_{1},\alpha)}..AG_{\alpha}(f_{s,t}^{(i_{k},\alpha)})h_{0,s}^{(j_{k},\alpha)}..f_{s,t}^{(i_{l},\alpha)}h_{0,s}^{(j_{l},\alpha)})\\ &-\frac{1}{2}\sum_{k}\tau(f_{s,t}^{(i_{1},\alpha)}h_{0,s}^{(j_{1},\alpha)}..f_{s,t}^{(i_{k},\alpha)}A^{*}G_{\alpha}^{*}(h_{0,s}^{(j_{k},\alpha)})...f_{s,t}^{(i_{l},\alpha)}h_{0,s}^{(j_{l},\alpha)})\\ &=\frac{1}{4}\sum_{k}\tau(f_{s,t}^{(i_{1},\alpha)}..(AG_{\alpha}-A^{*}G_{\alpha}^{*})(f_{s,t}^{(i_{k},\alpha)})..h_{0,s}^{(j_{l},\alpha)})\\ & +\frac{1}{4}\sum_{k}\tau(f_{s,t}^{(i_{1},\alpha)}..(AG_{\alpha}-A^{*}G_{\alpha}^{*})(h_{0,s}^{(j_{k},\alpha)})..h_{0,s}^{(j_{l},\alpha)})
\\ &+\frac{1}{4}\sum_{k}\tau(f_{s,t}^{(i_{1},\alpha)}..(AG_{\alpha}+A^{*}G_{\alpha}^{*})(f_{s,t}^{(i_{k},\alpha)})h_{0,s}^{(j_{k},\alpha)}..h_{0,s}^{(j_{l},\alpha)})\\ &-\frac{1}{4}\sum_{k}\tau(f_{s,t}^{(i_{1},\alpha)}..f_{s,t}^{(i_{k},\alpha)}(AG_{\alpha}+A^{*}G_{\alpha}^{*})(h_{0,s}^{(j_{k},\alpha)})..h_{0,s}^{(j_{l},\alpha)})\\ &=\frac{1}{4}\sum_{k}\tau(f_{s,t}^{(i_{1},\alpha)}..(AG_{\alpha}-A^{*}G_{\alpha}^{*})(f_{s,t}^{(i_{k},\alpha)})..h_{0,s}^{(j_{l},\alpha)})\\ &+\frac{1}{4}\sum_{k}\tau(f_{s,t}^{(i_{1},\alpha)}..(AG_{\alpha}-A^{*}G_{\alpha}^{*})(h_{0,s}^{(j_{k},\alpha)})..h_{0,s}^{(j_{l},\alpha)})
\\ &+\frac{1}{2}\sum_{k<k'}\tau(f_{s,t}^{(i_{1},\alpha)}..(\Gamma_{\alpha}+\hat{\Gamma}_{\alpha})(f_{s,t}^{(i_{k},\alpha)},..,f_{s,t}^{(i_{k'},\alpha)})..h_{0,s}^{(j_{l},\alpha)})\\ &-\frac{1}{2}\sum_{k<k'}\tau(f_{s,t}^{(i_{1},\alpha)}..(\Gamma_{\alpha}+\hat{\Gamma}_{\alpha})(h_{t-s,t}^{(j_{k},\alpha)},...,h_{t-s,t}^{(j_{k'},\alpha)})...h_{0,s}^{(j_{l},\alpha)})\end{align*}
At the last line we have used the following identity (summing only the definition of $ \Gamma_{\alpha}$):

\begin{align*}&2\sum_{k<k'}\tau(X_1Y_1...(\Gamma_{\alpha}+\hat{\Gamma}_{\alpha})(X_k,Y_k...,X_{k'})...X_lY_l)\\ &-2\sum_{k<k'}\tau(X_1Y_1...X_k(\Gamma_{\alpha}+\hat{\Gamma}_{\alpha})(Y_k,...X_{k'},Y_{k'})...X_lY_l)\\ =&\sum_{k<k'}\tau(X_1Y_1..(AG_{\alpha}+A^{*}\hat{G}_{\alpha})(X_kY_k..Y_{k'-1})X_{k'}..X_lY_l)\\ &+\tau(X_1Y_1..X_k(AG_{\alpha}+A^{*}\hat{G}_{\alpha})(Y_k..Y_{k'-1}X_{k'})..X_lY_l)\\ &\ \ -\tau(X_1Y_1...X_k(AG_{\alpha}+A^{*}\hat{G}_{\alpha})(Y_k...Y_{k'-1})X_{k'}...X_lY_l)\\ &-\tau(X_1Y_1...(AG_{\alpha}+A^{*}\hat{G}_{\alpha})(X_kY_k...X_{k'})...X_lY_l)\\ &-\tau(X_1Y_1...X_k(AG_{\alpha}+A^{*}\hat{G}_{\alpha})(Y_k...X_{k'})...X_lY_l)\\ &-\tau(X_1Y_1...Y_k(AG_{\alpha}+A^{*}\hat{G}_{\alpha})(X_{k+1}...X_{k'}Y_{k'})...X_lY_l)\\ &\ \ +\tau(X_1Y_1...X_k(AG_{\alpha}+A^{*}\hat{G}_{\alpha})(Y_k...Y_{k'})...X_lY_l)\\ &+\tau(X_1Y_1...(AG_{\alpha}+A^{*}\hat{G}_{\alpha})(X_{k+1}...X_{k'})...X_lY_l)\\ =&\sum_{1<k}\tau((AG_{\alpha}+A^{*}\hat{G}_{\alpha})(X_1..Y_{k})..Y_l)-\tau((AG_{\alpha}+A^{*}\hat{G}_{\alpha})(X_1..X_{k})..Y_l)\\ &+\tau(X_1..(AG_{\alpha}+A^{*}\hat{G}_{\alpha})(X_k)..Y_l)\\+&\sum_{k<l}\tau(X_1...(AG_{\alpha}+A^{*}\hat{G}_{\alpha})(Y_k..Y_{l}))+\tau(X_1..(AG_{\alpha}+A^{*}\hat{G}_{\alpha})(X_k..Y_{l}))\\ &-\tau(X_1..(AG_{\alpha}+A^{*}\hat{G}_{\alpha})(Y_k)..Y_l)\\ =&\sum_{k}\tau(X_1...(AG_{\alpha}+A^{*}\hat{G}_{\alpha})(X_k)...Y_l)-\tau(X_1...(AG_{\alpha}+A^{*}\hat{G}_{\alpha})(Y_k)...Y_l)\end{align*}

In the third line we have used the following reasoning to simplify a telescopic sum, all terms with an even number of terms bellow the $A$'s, thus beginning by an X ending by a Y or vice versa, appear twice with different signs, once coming from a $\Gamma$ with X's at end points, once from with Y's at end points; of course this does not apply to boundary terms thus remaining in the third line; all terms (again except boundary terms) with an odd number of terms also appear twice with a different sign, once as the inner term of a $\Gamma$ (in 
$\Gamma(A,B,C)$ a $\A(B)$), once as an outer term (in 
$\Gamma(A,B,C)$ a $\A(ABC)$). For the last equality, we used symmetry of $(AG_{\alpha}+A^{*}\hat{G}_{\alpha})$ to remove almost all terms.
\end{proof}
We thus want to integrate those relations, and sum them so that almost all terms cancel. The point is that the boundary terms of the integrals in 0 and t vanish except when $i_k$'s are $1$ in case we get a non zero value at t, or $j_k$'s are $1$ in case we get a non zero value at 0. Moreover those boundary terms give exactly the expression we want to relate, we have thus obtained~:
\begin{lemma}
\begin{align*}&\tau\circ H_{0,t}^{(n,\alpha)}\circ S_{t}-\tau\circ\Phi_{0,t}^{(n,\alpha)}=\sum_{i_{1}+j_{1}+...+i_{l}+j_{l}=n-1+l>1}\int_{0}^{t}ds \\ &\frac{1}{4}\sum_{k}\tau(f_{s,t}^{(i_{1},\alpha)}...(AG_{\alpha}-A^{*}G_{\alpha}^{*})(f_{s,t}^{(i_{k},\alpha)})...h_{t-s,t}^{(j_{l},\alpha)})\\ &+\frac{1}{4}\sum_{k}\tau(f_{s,t}^{(i_{1},\alpha)}...(AG_{\alpha}-A^{*}G_{\alpha}^{*})(h_{t-s,t}^{(j_{k},\alpha)})...h_{t-s,t}^{(j_{l},\alpha)})
\\ &+\frac{1}{2}\sum_{k<k'}\tau(f_{s,t}^{(i_{1},\alpha)}...(\hat{\Gamma}_{\alpha}-\Gamma_{\alpha})(f_{s,t}^{(i_{k},\alpha)},..,f_{s,t}^{(i_{k'},\alpha)})...h_{t-s,t}^{(j_{l},\alpha)})\\ &-\frac{1}{2}\sum_{k<k'}\tau(f_{s,t}^{(i_{1},\alpha)}..(\Gamma_{\alpha}-\hat{\Gamma}_{\alpha})(h_{t-s,t}^{(j_{k},\alpha)},..,h_{t-s,t}^{(j_{k'},\alpha)})..h_{0,s}^{(j_{l},\alpha)})\end{align*}
Especially, in the symmetric case, the right hand side vanishes as claimed earlier.
\end{lemma}
 In the non-symmetric case, it remains to prove this converges to 0 when $\alpha$ goes to infinity, this is the result of the next~:

\begin{proposition}
With the previous notations (especially $D(\Delta)\cap D(A) \cap M$ is a core for $D(\Delta^{1/2})$), $\tau\circ H_{0,t}^{(n)}\circ S_{t}=\tau\circ\Phi_{0,t}^{(n)}.$
\end{proposition}
\begin{proof}We have to prove that the right hand side of the formula of the previous lemma goes to zero. First, rewrite the $AG_{\alpha}$ terms in terms of the corresponding $\Gamma_{\alpha}$, for instance~: \begin{align*}&\frac{1}{4}\sum_{k}\tau(f_{s,t}^{(i_{1},\alpha)}...(AG_{\alpha}-A^{*}G_{\alpha}^{*})(f_{s,t}^{(i_{k},\alpha)})...h_{t-s,t}^{(j_{l},\alpha)})\\ &+\frac{1}{4}\sum_{k}\tau(f_{s,t}^{(i_{1},\alpha)}...(AG_{\alpha}-A^{*}G_{\alpha}^{*})(h_{t-s,t}^{(j_{k},\alpha)})...h_{t-s,t}^{(j_{l},\alpha)})\\ &=\frac{1}{2}\sum_{k}\tau((\Gamma_{\alpha}-\hat{\Gamma}_{\alpha})(f_{s,t}^{(i_{1},\alpha)}...(f_{s,t}^{(i_{k},\alpha)}),1,h_{t-s,t}^{(j_{k},\alpha)})...h_{t-s,t}^{(j_{l},\alpha)})\\ &+\frac{1}{2}\sum_{k}\tau((\Gamma_{\alpha}-\hat{\Gamma}_{\alpha})(f_{s,t}^{(i_{1},\alpha)}...(h_{t-s,t}^{(j_{k},\alpha)}),1,f_{s,t}^{(i_{k+1},\alpha)}))...h_{t-s,t}^{(j_{l},\alpha)}).\end{align*}
Using $\delta^+$-convergence 
(and the positivity of $\Gamma_{\alpha}$, and a derivation property for $\delta$, to get $||.||_1$ of terms in $\Gamma_{\alpha}$), we can replace $h_{t-s,t}^{(j_{k},\alpha)}$ by $h_{t-s,t}^{(j_{k})}$, idem for $f$, bellow $\Gamma_{\alpha}$ (first and third arguments). We can use the $L^{1}$-convergence of $\Gamma_{\alpha}$ and $\hat{\Gamma}_{\alpha}$ to $\Gamma$ (lemma \ref{GammaLim}) to see that the integral of the  four lines indeed go to zero (a DCT applies as in previous parts to get convergence of this integrals). 
\end{proof}

Let us now sketch the proof of symmetry in multitime case in our main example of interest~: $\tau\circ\Psi_{\rho}\circ\Phi_{0,t}^{\rho}=\tau\circ\Psi_H^{\sigma(t,\rho)}\circ H_{0,\tau(\rho)}^{\sigma(t,\rho)}\circ S_{t,\rho}.$ The proof will only be a matter of decomposing the highest and smallest time to use in the right way induction hypothesis and the two times case. 

%
Recall that via the identification of $\rho=(\rho',u)$ we have a well defined $\sigma(\rho)=(\sigma(u,\rho'),\tau(\rho'))$.

The following equation is in the spirit of (\ref{alternate}) and is the crucial part in extending the two times case to the general case. 
\begin{lemma}
Recall $\tau(\sigma(t,\rho))=\tau(\sigma((\rho,t)))=t$ so that we have defined $\Psi^{(\sigma((\rho,t)),t)}$.
We have the relation~:
$$\Psi_H^{\sigma(\rho)}\circ S_\rho\circ \Phi_{0,t}^{\rho}=\Phi_{0,t}\circ S_{t}\circ \Psi_H^{(\sigma((\rho,t))),t}\circ S_{t,\rho}$$
\end{lemma}

We can now conclude the proof of the symmetry property as follows.
By an induction hypothesis, we know~: $\tau\circ\Psi_{\rho}=\tau\circ\Psi_H^{\sigma(\rho)}\circ S_{\rho}.$
The previous lemma thus gives us  (using the two times case for the second equality) $$\tau\circ\Psi_{\rho}\circ\Phi_{0,t}^{\rho}=\tau\circ\Phi_{0,t}\circ S_{(t)}\circ \Psi_H^{(\sigma((\rho,t))),t}\circ S_{t,\rho}=\tau\circ H_{0,t}\circ \Psi_H^{(\sigma(\rho,t)),t}\circ S_{t,\rho}.$$

Said in words, we want to compute an evolution with a sequence of times gathered in $\rho$ and above a time $t$, after this highest time evolution we first inductively apply symmetry to the part with $\rho$ to get a term like the one in the left hand side of the lemma. This lemma means that if we evolve the highest time $t$, then invert times and evolves, this is the same as inverting times, evolving the same bunch of times corresponding to $\rho$, re-inverting and finally making evolve the highest time $t$. With this highest time evolution $\Phi_{0,t}$ just under the trace, we can now apply symmetry we proved in the two times case to get only $H$ terms.

Recall $\sigma((\rho,t))=(\sigma(t,\rho),\tau(\rho))$ so that by definition, we have the following concluding equation~: $$H_{0,t}\circ \Psi_H^{(\sigma((\rho,t))),t}=H_{0,t}\circ \Psi_H^{(\sigma(t,\rho)),t}\circ H_{0,\tau(\rho)}^{\sigma(t,\rho)}=\Psi_H^{(\sigma(t,\rho))}\circ H_{0,\tau(\rho)}^{\sigma(t,\rho)}.$$

Explained in words again, the evolution of the part corresponding to $\rho$ symmetrized can be decomposed in the evolution of the highest time range in $\rho$ followed by the inversion of the remaining bunch of times in $\rho$, by the very inductive definition, then we can gather this bunch of times with the highest time range $t$ to get the second part of the evolution in the symmetric of the union of $\rho$ with the extra time $t$...
\begin{proof}[Sketch of Proof of lemma 18]
Let us fix several notations. As in definition 8, let us call $\tilde{h}^{n,\sigma(\rho)}$, $\tilde{h}^{n,(\sigma(\rho,t),t)}$ for $\Psi_H^{\sigma(\rho)}$,
$\Psi_H^{(\sigma((\rho,t))),t}$ with the same definition as $f^{n,\rho}_{0,t}$ for $\Phi^{\rho}$, i.e. as we said variants of Boolean cumulants. Let us remind the reader that, in this paragraph we stick to what is called in subsection 2.2 our main example, i.e. $\rho$ is nothing but an ordered sequence of times, $M^{\rho}$ being an algebraic free product of one copy of $M$ for each time. Those maps are defined on $M^{\sigma(\rho)}$ , with value respectively in $M$ or $M_{[t]}*M_{[0]}$ (the free product of $M$ thought of at time $t$ with $M$ thought of at time $0$). As we noted in this definition 8, as soon as such a formula for $\Phi^{\rho}$ or here $\Psi_H^{\sigma(\rho)}$,
$\Psi_H^{(\sigma((\rho,t))),t}$ exists, it completely determines $f$ or $h$, existence follows inductively from our definitions, let us emphasize we gather here all non-zero times, i.e. $\tilde{h}^{n,\sigma(\rho)}$ is really a boolean cumulant with $\Psi_H^{\sigma(\rho)}$ as conditional expectation on $M_{[0]}$,  $\tilde{h}^{n,(\sigma(\rho,t),t)}$ is a boolean cumulant with $\Psi_H^{(\sigma((\rho,t))),t}$ as conditional expectation on $M_{[t]}*M_{[0]}$.

 For example imagine $\sigma(\rho,t_1)=(0\leq t_1\leq t_2)$ then \begin{align*}&\Psi_H^{(\sigma(\rho,t_1),t_1)}(X_{t_2}X_{t_1}X_{t_2}X_{t_1}X_{0}X_{t_1}X_{0}X_{t_2}X_{t_1}))\\ &=\tilde{h}^{9,(\sigma(\rho,t_1),t_1)}(X_{t_2}X_{t_1}X_{t_2}X_{t_1}X_{0}X_{t_1}X_{0}X_{t_2}X_{t_1})^{[t_1]}\\ &+\tilde{h}^{6,(\sigma(\rho,t_1),t_1)}(X_{t_2}X_{t_1}X_{t_2}X_{t_1}X_{0}X_{t_1})^{[t_1]}X_{0}\tilde{h}^{2,(\sigma(\rho,t_1),t_1)}(X_{t_2}X_{t_1})^{[t_1]}\\ &+\tilde{h}^{6,(\sigma(\rho,t_1),t_1)}(X_{t_2}X_{t_1}X_{t_2}X_{t_1})^{[t_1]}X_{0}\tilde{h}^{4,(\sigma(\rho,t_1),t_1)}(X_{t_1}X_{0}X_{t_2}X_{t_1})^{[t_1]}\\ &+\tilde{h}^{3,(\sigma(\rho,t_1),t_1)}(X_{t_2}X_{t_1}X_{t_2}X_{t_1})^{[t_1]}X_{0}(X_{t_1})^{[t_1]}X_{0}\tilde{h}^{2,(\sigma(\rho,t_1),t_1)}(X_{t_2}X_{t_1})^{[t_1]}.\end{align*}

Thus, expanding in this way, $\Psi_H^{\sigma(\rho)}\circ S_\rho\circ \Phi_{0,t}^{\rho}$  is a sum of terms each one being expressed as a product of terms like~:
$(\tilde{h}^{n,\sigma(\rho)}\circ S_\rho\circ \Phi_{0,t}^{\rho})$ or $f^{n,\rho}_{0,t}$ and non-involving terms (originally at highest time in $\rho$ and at arrival thought of at time $0$ because of time inversion $S_\rho$). Indeed after $\Phi_{0,t}^{\rho}$ we got products of unchanged terms (in $M^{\rho}$) and of various $f^{n,\rho}_{0,t}$ thought of at the highest time in $M^{\rho}$, then after time inversion via $S_\rho$, this highest time becomes $0$, and there may be several $f^{n,\rho}$'s non evolving via $\Psi_H^{\sigma(\rho)}$ or they may appear bellow an $\tilde{h}^{n,\sigma(\rho)}$.

Likewise, $\Phi_{0,t}\circ S_{t}\circ \Psi_H^{(\sigma((\rho,t))),t}\circ S_{t,\rho}$ is a sum of terms each one being expressed as a product of terms like $\tilde{h}^{n,(\sigma(\rho,t),t)}\circ S_{t,\rho}$, $f_{0,t}\circ S_{t}\circ \Psi_H^{(\sigma((\rho,t))),t}\circ S_{t,\rho}$ and the same non-evolving terms.

The only point in the proof of the equality of our lemma is the remark~:
$f_{0,t}\circ S_{t}\circ \Psi_H^{(\sigma((\rho,t))),t}\circ S_{t,\rho}=f^{n,\rho}_{0,t}$
and $\tilde{h}^{n,(\sigma(\rho,t),t)}\circ S_{t,\rho}=(\tilde{h}^{\sigma(\rho)}\circ S_\rho\circ \Phi_{0,t}^{\rho})$. In case this may not seem obvious, we explain the first, the second being a 
question of rewriting inductive definition for $\tilde{h}$.
 
To prove those formulas by induction (again on $n$ the number of alternating times term on which they are applied), we have to prove simultaneously $f_{u,t}\circ S_{t}\circ \Psi_H^{(\sigma((\rho,t))),t}\circ S_{t,\rho}=f^{n,\rho}_{u,t}$
and $\tilde{h}^{n,(\sigma((\rho,t)),t)}\circ S_{t,\rho}=(\tilde{h}^{(\sigma((\rho,u)),u)}\circ S_{u,\rho}\circ \Phi_{u,t}^{\rho})$.
Of course $\tilde{h}^{(\sigma((\rho,0)),0)}=\tilde{h}^{\sigma(\rho)}$ and this generalizes what we need.

The very definition of $f^{n,\rho}_{u,t}$ gives, if we write $\hat{h}^{n,\rho,u,t}=\tilde{h}^{(\alpha,\sigma((\rho,u)))}\circ S_{(\tau,\sigma)(u,\rho)}\circ \Phi_{u,t}^{n,\rho}$~:
\begin{align*}f_{s,t}^{(n,\alpha,\rho)}=\sum_{l,i_1+j_1+...+j_{l-1}+i_l-l=n-1}\int_{s}^{t}du \ \phi_{u-s,\alpha}\Gamma_{\alpha}(f_{u,t}^{(i_{1},\alpha,\rho)},\hat{h}^{j_{1},\rho,u,t}f_{u,t}^{(i_{2},\alpha)}\hat{h}^{j_{2},\rho,u,t}...\hat{h}^{j_{l-1},\rho,u,t},f_{u,t}^{(i_{l},\alpha,\rho)}).\end{align*}

But $\hat{h}^{n,\rho,u,t}=h_{0,u}^{(\alpha)}\Psi_H^{(\alpha,\sigma((\rho,u)),u)}\circ S_{(\tau,\sigma)(u,\rho)}\circ \Phi_{u,t}^{n,\rho}=h_{t-u,t}^{(\alpha)}\Psi_H^{n,(\sigma((\rho,t)),t)}\circ S_{t,\rho}$ by induction hypothesis and lemma \ref{symmsemig} .
Using now induction hypothesis in the above expression on $f_{u,t}$'s of the right hand side, and using (\ref{alternate}), we got $f_{s,t}\circ S_{t}\circ \Psi_H^{(\sigma((\rho,t))),t}\circ S_{t,\rho}=f^{n,\rho}_{s,t}$. 
As far as the other equation is concerned, we first note that by definition $\tilde{h}^{n,(\sigma((\rho,t)),t)}=\tilde{h}^{(\sigma(t,\rho),t)}\circ H^{(n,\alpha,\sigma(t,\rho))}_{0,\tau(\rho)}$. Thus an obvious induction on the number of times in $\rho$ reduces the result to proving $h^{(n,\alpha,\sigma(t,\rho))}_{0,\tau(\rho)}\circ S_{t,\rho}=h^{(n,\alpha,\sigma(u,\rho))}_{0,\tau(\rho)}\circ S_{u,\rho}\circ \Phi_{u,t}^{\rho}.$ But modulo inversion of $f$ and $h$ in what we previously proved, we checked (using $S_{t,\rho}$ is the inverse of $S_{\tau(\rho),\sigma(t,\rho)}$ and $\sigma((\sigma(t,\rho),\tau(\rho)))=(\rho,t)$)~: $h_{s,\tau(\rho)}\circ S_{\tau(\rho)}\circ \Psi^{((\rho,t),\tau(\rho))}=h^{(n,\sigma(t,\rho))}_{s,\tau(\rho)}\circ S_{t,\rho}$. But by the semigroup property $\Psi^{((\rho,t),\tau(\rho))}=\Psi^{((\rho,u),\tau(\rho))}\circ\Phi_{u,t}^{\rho}$, so that applying twice the previous formula, one deduces the formula we want.
%

\end{proof}
\begin{remark}
Let us note that traciality of $\tau\circ\Psi_{\rho}\circ\Phi_{0,t}^{\rho}$ is now obvious by induction. Indeed we can move a ``0-time" element around the state by induction hypothesis, and we can move a $t$ -time element in the same way after applying symmetry.
\end{remark}


\subsection{Summary of construction}

Using the remark of the first section on Path spaces,
we have thus obtained the following~:

\begin{theorem}\label{main}
With all the previous notations. Let us assume as in part 2.1, $A$ is the generator of a non-symmetric Dirichlet form, such that the anti-symmetric part is a derivation in the sense of (\ref{derForm}).
Let us assume $D(A)\cap D(\Delta)\cap M$ is a core in $D(\Delta^{1/2})$.
There exists a pair of  $f$ and $h$   \emph{$\alpha$-approximated} (N+1)-level-semigroup-families mutually affiliated to $A$ and $A^{*}$, for all N. The linear functional $\tau$ they induce on Path space $\P_{alg}(M)$ is a tracial state extending $\tau$ in each time and it extends to $\P_{max}(M)$. If $\phi_t=e^{-tA/2}$ is symmetric, $\tau$ is also symmetric (i.e. as soon as only products of terms with time between 0 and t are involved, invariant by the symmetry of times around $t/2$). It is also translation invariant, so that describing $\P_{\R, alg}(M)$ by an inductive limit of $\P_{[-t,\infty), alg}(M)$, it extends to a tracial symmetric state on $\P_{\R, alg}(M)$ and then $\P_{\R,max}(M)$. To distinguish this one of the previous one we call this extension $\tau_{\R}$. Because of the translation invariance, the translation of times on Path space induces a *-homorphism $\alpha_{t}$  on the GNS construction   $(\widetilde{M},\tau)$ of $(\P_{max}(M),\tau)$ and an automorphism (also denoted) $\alpha_{t}$ of the GNS construction   $(\hat{M},\tau)$ of $(\P_{\R,max}(M),\tau_{\R})$. $\alpha_{t}$ is a dilation of $\phi_t$ in the sense that $E_{M_0}(\alpha_t(x))=\phi_t(x)$.  When $\phi_t$ is symmetric, the reversal of time involution also induces a trace preserving involutive automorphism $\beta$ on $(\hat{M},\tau)$ satisfying of course $\beta\alpha_{t}\beta= \alpha_{-t}$.
\end{theorem}

\begin{remark}Let us note that for those semigroups considered, our construction shows $\phi_t$ is factorizable in the sense of \cite{AD04}. 
Actually, one can prove our processes are freely Markovian in the sense of \cite{Vo6}, so that when restricted to integer multiples of $t$, we recover the non-commutative Markov Chain construction of Theorem 6.6 in \cite{AD04}, based on reduced free product with amalgamation. Especially, reversed martingale properties proven in that theorem are also valid here (as it may have been clear using our symmetry properties).
\end{remark}

\section{Application to Transportation cost inequality}
A free Talagrand transportation cost inequality was first proven in \cite{BV01} in the one variable case and we will extend here there approach to the general multivariable case. The key point is to get an estimate on Wasserstein distance between two infinitesimally close points in an Orstein Uhlenbeck process.
More generally, we will first get such estimates along any solution of a free SDE with polynomial drift. Again the key argument is based on starting simultaneously our stationary process at $X_t$ a point of this SDE and get an inequality between the path of this SDE and the stationary variant in comparing drifts. Thus we will start by proving in a first subsection that our process indeed satisfy another free SDE when the derivation is a free difference quotient. In fact, we will prove a slightly more general result about subsystems looking like a free difference quotient for further use when we will prove that our process do satisfy a SDE in a much more general context. We will then turn back to our infinitesimal estimate in a second subsection, and finally to free transportation cost inequality in a third.

\subsection{Solution of our SDE in the free difference quotient case}
We will use the following variant of Paul L\'{e}vy's theorem giving a characterization of Brownian motion proven in \cite{BGC} for the free Brownian motion, we use here an immediate extension to Speicher's $B$-Gaussian stochastic processes \cite{Sp98}. Thus $B$ is a fixed von Neumann sub-algebra with its canonical $\tau$ preserving conditional expectation $E_B$, $\eta:B\rightarrow M_m(B)$ (if $m$ infinite, this denote $B\o B(H_m)$ $H_m$ of dimension $m$), a completely positive map, assumed to be $\tau$-symmetric ($\tau(\eta_{ij}(x)y)=\tau(x\eta_{ji}(y)$) so that via Proposition 2.20 in \cite{SAVal99} the associated $B$-semicircular system is tracial. Given $B_{s}$ be an increasing filtration of von Neumann algebras $B\subset B_{0}$, we will call adapted B-free Brownian motion of covariance $\eta$ a family $X_s^j$ of adapted processes such that $(X_s^j)_{s\geq t}$ is Speicher's $B_s$-Gaussian stochastic processes $X_s^j$ with covariance given by $E_{B_{t}}((X_s^i-X_t^i)b(X_u^j-X_t^j))=((s\wedge u)-t)\eta_{ij}(E_{B}(b))$, for any $s,u\geq t$, $b\in B_{t}$. Stated otherwise in the notations of \cite{SAVal99} $W^{*}(B_t, X_s^j,s\geq t)=\Phi(B_t,\tilde{\eta}\circ E_B)$ where $\tilde{\eta}:B\rightarrow B(L^2([t,\infty))\o H_m)\o B$ given by $\langle 1_{[ts)}\o i,\tilde{\eta}(b)(1_{[tu)}\o j)\rangle=((s\wedge u)-t)\eta_{ij}(b)$ with obvious notations.

\begin{theorem}
Let $B_{s}$ be an increasing filtration of von Neumann algebras in a non-commutative tracial probability space $(M,\tau)$
$Z_{s}=(Z_{s}^{1},...,Z_{s}^{n}), s\in \R_{+}$ an m-tuple of self-adjoint processes adapted to this filtration $Z_{0}=0$ and~:
\begin{enumerate}
\item $\tau(Z_{t}|B_{s})=Z_{s}$
\item$Z_{t}-Z_{s}=U_{t,s}+V_{t,s}$ with $\tau(|U_{t,s}|^{4})\leq K (t-s)^{3/2}$ and $\tau(|V_{t,s}|^{2})\leq K (t-s)^{2}$ 
\item $\tau(Z_{t}^{k}AZ_{t}^{l}B)=\tau(Z_{s}^{k}AZ_{s}^{l}B)+(t-s)\tau(A\eta_{lk}(E_{B}(B)))+o(t-s)$ for any $A,B\in B_s$.
\end{enumerate}
Then $Z$ is a B-free Brownian motion of covariance $\eta$. 
\end{theorem}

Note that we don't assume $Z_s\in M$ a priori, this is the main improvement with respect to the result of \cite{BGC}. In order to prove that boundedness follows from our assumptions, we will crucially use \cite{HRSp07}. The idea is to deduce a variant of Burger's equation for the resolvent of our process. We will actually apply in the same time the method of characteristics to get uniqueness of the solution, thus equality in 1-variable distribution with Brownian motion. We also slightly weaken assumption (2) even though this will be easy to recover their assumption when boundedness will be proven. This will be crucial for us since  $V_{t,s}$ will correspond to an ordinary integral term easy to bound but only in $||.||_2$-norm if we don't want to add strong conditions in the theorem solving our SDEs.
\begin{proof}[Sketch of proof]
Let us use classical resolvent equations (and compute for $b,c\in B$ such that $(b+sc-Z_{s})$ admits an inverse in $M$ for all $s\in [0,t]$, with $||(b+sc-Z_{s})^{-1}||\leq C$):
\begin{align*}&(b+tc-Z_{t})^{-1}=b^{-1}+\\&\sum_{i=0}^{n-1}(b+c(i+1)t/n-Z_{(i+1)t/n})^{-1}(Z_{(i+1)t/n}-Z_{it/n}-ct/n)(b+cit/n-Z_{it/n})^{-1}
\end{align*}\begin{align*}&(b+tc-Z_{t})^{-1}\\ &
=b^{-1}+\sum_{i=0}^{n-1}(b+cit/n-Z_{it/n})^{-1}(Z_{(i+1)t/n}-Z_{it/n}-ct/n)(b+cit/n-Z_{it/n})^{-1}\\ &+\sum_{i=0}^{n-1}(b+cit/n-Z_{it/n})^{-1}((Z_{(i+1)t/n}-Z_{it/n}-ct/n)(b+cit/n-Z_{it/n})^{-1})^2\\ &+\sum_{i=0}^{n-1}((b+c(i+1)t/n-Z_{(i+1)t/n})^{-1}-(b+cit/n-Z_{it/n})^{-1})\times \\ &\ \ \ \ \ \ \ \ \ \times((U_{(i+1)t/n,it/n}+V_{(i+1)t/n,it/n}-ct/n)(b+cit/n-Z_{it/n})^{-1})^2
\end{align*}

After taking a conditional expectation on $B$ and using our assumptions, one gets~:\begin{align*}&E_B((b+tc-Z_{t}^k)^{-1})=b^{-1}+no(t/n)+\\ &\frac{t}{n}\sum_{i=0}^{n-1}E_B((b+cit/n-Z_{it/n}^k)^{-1}\left(\eta_{kk}(E_{B}((b+cit/n-Z_{it/n}^k)^{-1}))-c\right)(b+cit/n-Z_{it/n}^k)^{-1})+E_{B}(R_n)\end{align*}

with \begin{align*}&R_n=\sum_{i=0}^{n-1}(b+c(i+1)t/n-Z_{(i+1)t/n})^{-1}(U_{(i+1)t/n,it/n}+V_{(i+1)t/n,it/n}-ct/n)\times \\ &\ \ \ \ \ \ \ \ \ \times(b+cit/n-Z_{it/n})^{-1}(U_{(i+1)t/n,it/n}(b+cit/n-Z_{it/n})^{-1})^2\\ &-\sum_{i=0}^{n-1}(b+cit/n-Z_{it/n})^{-1}(ct/n)(b+cit/n-Z_{it/n})^{-1}\times \\ &\ \ \ \ \ \ \ \ \ \times(Z_{(i+1)t/n}-Z_{it/n}-ct/n)(b+cit/n-Z_{it/n})^{-1})\\ &-\sum_{i=0}^{n-1}(b+cit/n-Z_{it/n})^{-1}(Z_{(i+1)t/n}-Z_{it/n})(b+cit/n-Z_{it/n})^{-1}\times \\ &\ \ \ \ \ \ \ \ \ \times(ct/n)(b+cit/n-Z_{it/n})^{-1}\\ &\\ &+\sum_{i=0}^{n-1}((b+c(i+1)t/n-Z_{(i+1)t/n})^{-1}-(b+cit/n-Z_{it/n})^{-1})\times \\ &\ \ \ \ \times(U_{(i+1)t/n,it/n}+V_{(i+1)t/n,it/n}-ct/n)(b+cit/n-Z_{it/n})^{-1}(V_{(i+1)t/n,it/n}-ct/n)\times \\ &\ \ \ \ \ \ \ \ \ \times(b+cit/n-Z_{it/n})^{-1}\\ &\\ &+\sum_{i=0}^{n-1}((b+c(i+1)t/n-Z_{(i+1)t/n})^{-1}-(b+cit/n-Z_{it/n})^{-1})\times \\ &\ \ \ \ \ \ \ \ \ \times(V_{(i+1)t/n,it/n}-ct/n)(b+cit/n-Z_{it/n})^{-1}(U_{(i+1)t/n,it/n})(b+cit/n-Z_{it/n})^{-1}\\ &=:S_n+T_n+T'_n+W_n+W'_n\end{align*}

Using a non-commutative H\"{o}lder's inequality, we deduce from our assumptions~:
\begin{align*}||S_n||_1&\leq \sum_{i=0}^{n-1}||U_{(i+1)t/n,it/n}||_4^2||(b+cit/n-Z_{it/n})^{-1}||^3||(b+c(i+1)t/n-Z_{(i+1)t/n})^{-1}||\times \\ &\times ||U_{(i+1)t/n,it/n}+V_{(i+1)t/n,it/n}-ct/n||_2\\ &\leq C^4K^{1/2}(t/n)^{3/4}n((K)^{1/4}(t/n)^{3/8}+((K)^{1/2}+||c||)t/n)
\\ &||T_n||_2,||T_n'||_2\leq ||c||tC^3((K)^{1/4}(t/n)^{3/8}+((K)^{1/2}+||c||)t/n)
\\ &||W_n||_1,||W_n'||_1\leq t2C^3((K)^{1/2}+||c||)((K)^{1/4}(t/n)^{3/8}+((K)^{1/2}+||c||)t/n)
\end{align*}

Thus, we found (using a Riemann integral since $Z_t$ is continuous in $||.||_2$):\begin{align*}E_B&((b+ct-Z_{t}^k)^{-1})\\&=b^{-1}+\int_{0}^tds\ E_B((b+cs-Z_{s}^k)^{-1}\left(\eta_{kk}(E_{B}((b+cs-Z_{s}^k)^{-1}))-c\right)(b+cs-Z_{s}^k)^{-1}).
\end{align*}
If we can apply this to $c=\eta_{kk}(b^{-1})$, Gronwall's Lemma immediately gives $E_B((b+\eta_{kk}(b^{-1})t-Z_{t}^k)^{-1})=b^{-1}$.

Now, we apply Theorem 2.1 in \cite{HRSp07} so that for any $z$ in the upper half plane, one gets a unique solution $W\in B_{+}$ (the elements of $B$ with strictly positive real parts, i.e. larger than $\epsilon I$) to $-izW+t\eta_{kk}(W)W=1$. Let us call $b=(-iW)^{-1}\in B$ (note that $-iW$ has strictly negative imaginary part and is thus invertible in $B$ via e.g. lemma 3.1 in \cite{HaagTho02} (and a double commutant argument) ). As a consequence, $1+t\eta_{kk}(b^{-1})b^{-1}=zb^{-1}$ or $b+t\eta_{kk}(b^{-1})=z$. Thus $b+\eta_{kk}(b^{-1})s-Z_{s}^k=z-Z_{s}^k-\eta_{kk}(b^{-1})(t-s)=(1+\eta_{kk}(iW)(t-s)(z-Z_{s}^k)^{-1})(z-Z_{s}^k)$. Since  $Z_{s}^k$ is self-adjoint, $||(z-Z_{s}^k)^{-1}||\leq 1/|\Im z|$ and by Theorem 2.1 in \cite{HRSp07}, we also have $||W||\leq 1/|\Im z|$, thus $||\eta_{kk}(iW)(t-s)(z-Z_{s}^k)^{-1}||\leq t/|\Im z|^2||\eta_{kk}||$, so that our invertibility assumption is satisfied using Neumann's lemma for $z$ with sufficiently large imaginary part.

We have thus verified our assumptions to get for such a $z$:
$$E_B((z-Z_{t}^k)^{-1})=b^{-1}.$$

Now everything can be applied for a B-Brownian motion of covariance $\eta_{kk}$, so that the equality (and an analytic continuation to upper half plane) proves $Z_{t}^k$ has the same distribution that such a Brownian motion, thus is bounded. The same can be applied to $Z_{t}^k-Z_{s}^k$ to strengthen our assumption to an analogue with $V=0$, i.e. an appropriate $||.||_4$ boundedness.

We can now apply the argument of \cite{BGC} to conclude.
\end{proof}
\begin{theorem}\label{FDQ}
Let $A$, $\delta$ as in section 2 and $(\widetilde{M},\tau)$ given by Theorem \ref{main} with canonical filtration $\widetilde{M}_s$ (induced by the one on path space). Let us assume given $B\subset N\subset M$ von Neumann sub-algebras

We assume $\delta(B)=0$, $A(B)=0$, $N=W^*(B,X^1_0,...,X^n_0)$, $\delta=\delta_1\oplus\delta_2$ with $\delta_2(N)=0$, and after restriction $\delta_1:B\langle X_1,...,X_n\rangle \rightarrow H(B\langle X^1_0,...,X^n_0\rangle,\eta_X)$ is identical to the (multi-variable obvious generalization of) the free difference quotient with respect to a completely positive map $\eta$ of \cite{ShlyFreeAmalg00}, i.e. $\eta_X=\eta\circ E_B$, $H(B\langle X^1_0,...,X^n_0\rangle,\eta_X)$ the bimodule generated by (i.e. completion of the linear span of) $bS_ib'$ ($S_i$ often written $1\o 1$ in one variable case) with $B\langle X_1,...,X_n\rangle$-valued scalar product $\langle bS_ib',cS_jc'\rangle_{B\langle X\rangle}=b'^*\eta_{X,ij}(b^*c)c'$,$\langle bS_ib',cS_jc'\rangle=\tau(\langle bS_ib',cS_jc'\rangle_{B\langle X\rangle}$) , and $\delta_1=(\partial_1,...,\partial_n)$, $\partial_i$ the unique derivation satisfying $\partial_i(X^j_0)=1_{i=j}S_i$, $\partial_i(b)=0, b\in B$.

Let us also assume $X^i_0\in D(A)\cap D(\Delta)$ so that $\xi^i_0=A(X^i_0)\in L^2(M_0)$ and $\tau((X_0^i)^3A(X_0))=\tau((X_0^i)^3\Delta(X_0))$.

Let us call $S_{t}^{i}=X_{t}^{i}-X_{0}^{i}+\frac{1}{2}\int_{0}^{t}\xi_{s}^{i}ds$ 
then $(S_{t}^{1},...,S_{t}^{n})$ is a B-free Brownian motion of covariance $\eta$ adapted to the filtration $\widetilde{M}_s$.
\end{theorem}

\begin{proof}
For any $Y\in L^2(M_{0})$ the function $Y_{s}$ 
is easily seen to be continuous in $L^{2}(\tilde{M})$ 
 Since $||Y_{s}||_{2}$ is constant, they are 
 Riemann integrable, and thus $||\int_{0}^{t}Y_{s}ds||_{2}\leq t||Y_{0}||_{2}$. 
 We write $S_t-S_s=U_{t,s}+V_{t,s}$ with $U_{t,s}=X_t-X_s$, $V_{t,u}=\frac{1}{2}\int_{u}^{t}\xi_{s}^{i}ds$, we have thus checked the assumption on $V_{t,s}$ in (2).

To check the first assumption of Paul L\'{e}vy's theorem, we have to prove cancellation of $\tau((S_{t}^{i}-S_{s}^{i})A_{s})=\tau((X_{t}^{i}-X_{s}^{i}+\frac{1}{2}\int_{s}^{t}du\xi_{u}^{i})A_{s})$
(for $A_{s}$ a non-commutative polynomial in $X_{u},u\leq s$)
By definition of our state \begin{align*}\frac{1}{2}\tau(\int_{s}^{t}du\xi_{u}^{i}A_{s})&=\frac{1}{2}\int_{s}^{t}du\tau((\phi_{u-s}(\xi_{0}^{i}))_{s}A_{s})\\ &=-\int_{s}^{t}\frac{d}{du}\tau((\phi_{u-s}(X_0^{i}))_{s}A_{s})\\ &=\tau(X^{i}_{s}A_{s})-\tau(\phi_{t-s}(X_0^{i})_{s}A_{s})
\\ &=\tau(X^{i}_{s}A_{s})-\tau(X^{i}_{t}A_{s})
\end{align*}

\bigskip
We know check, in a similar way, the third assumption, with the first proven, it suffices to show 
$\tau((S_{t}^{k}-S_{s}^{k})A(S_{t}^{l}-S_{s}^{l})B)=(t-s)\tau(\eta_{kl}(E_B(A))B),$ with again $A,B$ non-commutative polynomials as $A$ above.

Again let us compute  (The first computation is really the same as the previous one)

\begin{align*}\frac{1}{4}\tau(\int_{s}^{t}du\ \xi_{u}^{k}A\int_{s}^{t}du\ \xi_{u}^{l}B)&=\frac{1}{4}\int_{s}^{t}du\int_{s}^{t}dv\ 1_{u\leq v}\tau(\xi_{u}^{k}A(\phi_{v-u}(\xi_{0}^{l}))_{u}B)+1_{v\leq u}\tau((\phi_{u-v}(\xi_{0}^{k}))_{v}A\xi_{v}^{l})B)\\&=\frac{1}{2}\int_{s}^{t}du\ \left(\tau(\xi_{u}^{k}A(X_{0}^{l})_{u}B)-\tau(\xi_{u}^{k}A(\phi_{t-u}(X_{0}^{l}))_{u}B)\right)\\ &+\frac{1}{2}\int_{s}^{t}dv\ \left(\tau((X_{0}^{k})_{v}A\xi_{v}^{l})B)-\tau((\phi_{t-v}(X_{0}^{k}))_{v}A\xi_{v}^{l})B)\right)\\
&=\frac{1}{2}\int_{s}^{t}du\ \tau(\xi_{u}^{k}AX_{u}^{l}B)+\tau(X_{u}^{k}A\xi_{u}^{l}B)\\ &-\frac{1}{2}\tau(\int_{s}^{t}du\ \xi_{u}^{k}AX_{t}^{l}B)-\frac{1}{2}\tau(X_{t}^{k}A\int_{s}^{t}du\ \xi_{u}^{l}B)
\end{align*}

\begin{align*}\frac{1}{2}\int_{s}^{t}du\ \tau(\xi_{u}^{k}AX_{u}^{l}B)&=\frac{1}{2}\int_{s}^{t}du\ \tau((\phi_{u-s}(\xi_{0}^{k}))_{s}A(\phi_{u-s}(X_{0}^{l}))_{s}B)\\&+\sum_{i}\int_{s}^{u}dv\ \tau(\phi_v(\Gamma(\phi_{u-v}(\xi_{0}^{k}),\Phi^*_v(A),\phi_{u-v}(X_{0}^{l})))_{s}B)
\end{align*}

where $\Phi^*_v(A)$ is only a formal expression to replace the term in the definition.

But 

\begin{align*}\frac{1}{2}\int_{s}^{t}du&\ \tau((\phi_{u-s}(\xi_{0}^{k}))_{s}A(\phi_{u-s}(X_{0}^{l}))_{s}B)+\tau((\phi_{u-s}(X_{0}^{k}))_{s}A(\phi_{u-s}(\xi_{0}^{l}))_{s}B)\\ &=\int_{s}^{t}du\ \frac{d}{du}\tau((\phi_{u-s}(X_{0}^{k}))_{s}A(\phi_{u-s}(X_{0}^{l}))_{s}B)\\ &=\tau(X_{s}^{k}AX_{s}^{l}B)-\tau((\phi_{t-s}(X_{0}^{k}))_{s}A(\phi_{t-s}(X_{0}^{l}))_{s}B).
\end{align*}

and likewise, \begin{align*}\frac{1}{2}&\int_{s}^{t}du\ \int_{s}^{u}dv\ \tau(\phi_v(\Gamma(\phi_{u-v}(\xi_{0}^{k}),\Phi^*_v(A),\phi_{u-v}(X_{0}^{l})))_{s}B)+\tau(\phi_v(\Gamma(\phi_{u-v}(X_{0}^{k}),\Phi^*_v(A),\phi_{u-v}(\xi_{0}^{l})))_{s}B)\\&=-\int_{s}^{t}dv\ \tau(\phi_v(\Gamma(\phi_{t-v}(X_{0}^{k}),\Phi^*_v(A),\phi_{t-v}(X_{0}^{l})))_{s}B)+\int_{s}^{t}dv\ \tau(\phi_v(\Gamma(X_{0}^{k},\Phi^*_v(A),X_{0}^{l}))_{s}B)\\&=\int_{s}^{t}dv\ \tau(\phi_v(\Gamma(\phi_{t-v}(X_{0}^{k}),\Phi^*_v(A),\phi_{t-v}(X_{0}^{l})))_{s}B)+\int_{s}^{t}dv\ \tau((\phi_{v}(\eta_{kl}(E_B(\Phi^*_v(A)))_s B)
\\&=\int_{s}^{t}dv\ \tau(\phi_v(\Gamma(\phi_{t-v}(X_{0}^{k}),\Phi^*_v(A),\phi_{t-v}(X_{0}^{l})))_{s}B)+(t-s)\tau(\eta_{kl}(E_B(A)) B)
\end{align*}

We have used that $\Gamma(X,Y,Z)=\sum_{i,j}\langle\partial_i(X^*),Y\partial_j(Z)\rangle_{B\langle X\rangle})+\Gamma_2(X,Y,Z) (X,Z\in N)$,
and for $X_{0}^{k}$ since $\delta_2$ vanishes on them and with the values of free difference quotient we also have
$\Gamma(X_{0}^{k},Y,X_{0}^{l})=\eta_{kl}(E_B(Y))$. Finally, we also used that by symmetry $\tau((b)_0\Phi^*_v(A))=\tau((\phi_v(b))_sA)$ and for $b\in B$ since $A(b)=0$, $\delta(b)=0$,$\phi_v(b)=b$ this equals $\tau(b_0A)$ so that $E_B(\Phi^*_v(A))=E_B(A)$.


If we sum up, we got~:
\begin{align*}\tau((X_{t}^{k}+\frac{1}{2}\int_{s}^{t}du\ \xi_{u}^{k})A(X_{t}^{l}+\frac{1}{2}\int_{s}^{t}du\ \xi_{u}^{l})B)
&=\tau(X_{s}^{k}AX_{s}^{l}B)+(t-s)\tau(\eta_{kl}(E_B(A)) B)
\end{align*}

The result we wanted to prove ($\tau((S_{t}^{k}-S_{s}^{k})A(S_{t}^{l}-S_{s}^{l})B)=(t-s)\tau(\eta_{kl}(E_B(A))B)$) follows from this and the already checked $\tau((S_{t}^{k}-S_{s}^{k})AX_{s}^{l}B)=0$ (and the symmetric one, of course).
\bigskip

It only remains to prove the second assumption of Paul L\'{e}vy's Theorem (by invariance of time we only consider $s=0$, $X_0=X_0^i$). 
We will prove this without care about the constant $K$, recall we have to bound $\tau((X_{t}-X_{0})^{4})$. First note that using $\phi_t$ preserve $\tau$, we have $\tau((X_{t}-X_{0})^{2})=2\tau(X_{0}(X_{0}-\phi_{t}(X_{0}))\leq ||\xi_{0}||||X_{0}||t\leq Ct$. 
 
But  now, using traciality and $\tau(X_t^4)=\tau(X_0^4)$, we get $\tau((X_{t}-X_{0})^{4})=2 \tau((X_{t}-X_{0})X_{0}(X_{t}-X_{0})X_{0})+4\tau((X_{t}-X_{0})^{2}X_{0}^{2})+ 8\tau(X_{0}^{3}(X_{t}-X_{0}))+4\tau(X_{t}^{3}(X_{t}-X_{0}))$

For the two first terms up to a term of order $t^{3/2}$ (we could get a term in $t^2$ with explicit computations, but this bound is automatic by what we already noticed and sufficient) they are equal as above to 
$2 \tau(S_{t}X_{0}S_{t}X_{0})+4\tau(S_{t}^{2}X_{0}^{2})=2t\tau(X_{0}\eta_{ii}(X_0))+4t\tau(\eta_{ii}(1)X_{0}^{2}).$

Finally compute first
$\tau(X_{0}^{3}(X_{t}-X_{0}))=-\frac{1}{2}\int_{0}^{t}ds\tau(\phi_{s}^*(X_{0}^{3})\xi_{0})=-\frac{t}{2}\tau(X_{0}^{3}A(X_{0}))-\frac{1}{2}\int_{0}^{t}ds\tau((\phi_{s}^*-id)(X_{0}^{3})\xi_{0}).$
Likewise $\tau(X_{t}^{3}(X_{t}-X_{0}))=\tau((id-\phi_t)(X_{0}^{3})X_{0})=\frac{1}{2}\int_{0}^{t}ds\tau(\phi_{s}(A(X_{0}^{3}))X_{0})=\frac{t}{2}\tau(A(X_{0}^{3})X_{0})+\frac{1}{2}\int_{0}^{t}ds\tau(A(\phi_{s}-id)((X_{0}^{3}))X_{0}).$

But as far as the terms linear in $t$ are concerned in those two expressions, since we have among our assumptions $\tau(X_{0}^{3}(A-\Delta)(X_{0}))=0$, they sum up to $-4t\tau(X_{0}^{3}A(X_{0}))+2t\tau(A(X_{0}^{3})X_{0})=-2t\tau(X_{0}^{3}\Delta(X_{0}))=-2t\tau(\partial_i(X_0^3)\partial_i(X_0))=-2t\tau(\eta_{ii}(X_0)X_0)-4t\tau(\eta_{ii}(1)X_0^2)$ so that they cancel exactly the various previously found linear terms.
The second terms are bounded by $||A^*(X_{0}^{3})||_2||\xi_{0}||_2t^{2}$ and $||A(X_{0}^{3})||_2||A^*(X_{0})||_2t^{2}$ respectively, we are done.
\end{proof}

\subsection{Infinitesimal estimates along trajectories of free Brownian diffusions with polynomial drift}
Even though we will be mainly interested in an infinitesimal estimate along the Orstein Uhlenbeck process, which can be obtained using a change of variable expressing explicitly the law at time $t$ of this process, we will give here the proof of a result of independent interest, for more general solutions of free SDEs with polynomial drift. Thus if $V_i(s)$ is a ($C^1$, i.e. with coefficients of monomials $C^{1}$ in $s$, see later the second main theorem for details) family of non-commutative polynomials over a sub-algebra $B$, we will assume given on [0,T] a process satisfying~:
$$X_t^i=X_0^i+\int_0^t V_i(s)(X_s^1,...,X_s^n)ds +S_t^i,$$
where $S_t^i$ is a B-free Brownian motion (say of covariance $H:B\to M_n(B)$ $H(b)=Diag(\eta(b))$). Such solutions can be obtained via Lipschitz arguments as in \cite{BS}, but we only assume here we are given a solution, whatever the way we get it.

We state first a more general result and then check that the bounded conjugate variable assumption used here is satisfied in our polynomial drift case. Let us recall the definition of Wasserstein distance relative to a sub-algebra $B$ (with a given state $\theta$ on it), as defined in \cite{BV01} ($\simeq$ denotes equality in distribution):

\begin{align*}d_W((X_1,...,X_n),(Y_1,...,Y_n):B)&=\inf\{||(X_i'-Y_i')_{1\leq i\leq n}||_2\ |\\ & \ (X_1',...,X_n',Y_1',...,Y_n')\subset (M_3,\tau_3),\ B\subset M_3\  \tau_3|_B=\theta,\\ & (X_1',...,X_n',B)\simeq (X_1,...,X_n,B),(Y_1',...,Y_n',B)\simeq (Y_1,...,Y_n,B) \ \}\end{align*}

\begin{theorem}\label{infEst}
Let assume given a process satisfying $$X_t^i=X_0^i-\frac{1}{2}\int_0^t \Xi_s^i ds +S_t^i,$$ with $S_t^i$ is a B-free Brownian motion of covariance $\eta$,  $\Xi_s^i\in W^{*}(B,X_s^1,...,X_s^n)$ (such that the integral make sense). Assume moreover that $X_s^1,...,X_s^n$ has $L^2$ conjugate variable with respect to $B$ and $\eta$ (as in \cite{SAVal99}, or with the notations recalled in Theorem \ref{FDQ} $\partial^*S_i=\xi_s^i\in L^2(M)$), and $\Xi_t^i$ is continuous to the right in $||.||_2$-norm at $s$.
Then for $t>s$ close to $s$~: $$d_W((X_t^1,...,X_t^n),(X_s^1,...,X_s^n):B)=\frac{t-s}{2}||\Xi_s-\xi_s||_2+o(t-s).$$
\end{theorem}
\begin{proof} We fix $s$.
We apply Theorem \ref{FDQ} to the case $N=M=W^{*}(B,X_s^1,...,X_s^n)$ $\delta$ the associated free difference quotient as required by the Theorem, $A=\Delta$. Since $\xi_s^i=\Delta X_s^i\in L^2(M)$ the assumptions of the theorem are satisfied. We are thus given a B-free Brownian motion $S_t^i$ satisfying, if we denote by $\Phi$ evolution according to this stationary SDE~: $\Phi_u(X_s^i)=X_s^i-\frac{1}{2}\int_0^udv\Phi_v(\xi_s^i)+S_u^i.$

Thus we have in $\widetilde{M}$ a sub-algebra $M_S=W^{*}(B,X_s^1,...,X_s^n,\{S_t^i\})$. We also have $\hat{M}$ containing the solution $X_t,t\geq s$ of our SDE, and a copy of $M_S$ in it. We can thus consider the amalgamated free product $\M=\widetilde{M}\star_{M_S}\hat{M}$. 

Now we can compute in $\M$~: \begin{align*}||\Phi_{t-s}(X_s^i)-X_t^i||_2&\leq \frac{1}{2}\int_0^{t-s}dv||\Phi_v(\xi_s^i)-\Xi_{s+v}^i||_2\\ &\leq \frac{t-s}{2}||\xi_s^i-\Xi_{s}^i||_2+\frac{1}{2}\int_0^{t-s}dv||\Phi_v(\xi_s^i)-\xi_s^i||_2+||\Xi_{s}^i-\Xi_{s+v}^i||_2 \\ &\leq \frac{t-s}{2}||\xi_s^i-\Xi_{s}^i||_2+o(t-s)\end{align*}
Indeed, with our assumptions it suffices to note $||\Phi_v(\xi_s^i)-\xi_s^i||_2^2\leq2||\xi_s^i||_2||(\phi_v-id)(\xi_s^i)||_2\to 0$ with $v$. This concludes.
\end{proof}
The following result is an adaptation in free probability of (a special case of) lemma 4.2 in \cite{RoTh02}, except that we have to use Ito Formula for the proof instead of Girsanov Theorem, not (yet) available in free probability.

\begin{theorem}\label{boundConj}
Let assume given on [0,T] a process, bounded by $R\geq 1$ in M on this interval, satisfying~:
$$X_t^i=X_0^i+\int_0^t V_i(s)(X_s^1,...,X_s^n)ds +S_t^i,$$
where $S_t^i$ is a B-free Brownian motion (say of covariance $H:B\to M_n(B)$ $H(b)=Diag(\eta(b))$, $\eta:B\to B$ $\tau$-symmetric $||\eta(1)||\leq 1$).
Moreover $V_i(s)$ is assumed to be a B-valued non-commutative polynomial (of degree bounded on $[0,T]$ by $p$), so that considering $X_{i_1}...X_{i_n}$ we can speak of the coefficient of this variable in $B^{\o_{alg}n+1}$. We assume that those coefficients are $C^{1}$ with value in the corresponding projective tensor product (We say  $V_i(s)$ is $C^1$). Also assume $\partial_iV_j(s)=\sigma\partial_jV_i(s)$ ($\sigma$ the flip, $\partial_i$ ordinary free difference quotient relative to $B$, this is for instance the case when $V_i(s)=D_iV(s)$ with $D_i=m\sigma\partial_i$ the cyclic derivative).

Then $X_t^1,...,X_t^n$ have bounded conjugate variables with respect to $B$, $\eta$, and $$\xi_s^i=\frac{1}{s}E_{W^*(B,X_s^1,...,X_s^n)}\left(X_s^i-X_0^i+\int_0^sdt \ tF_{V_i}(t,X_t^1,...,X_t^n)\right)-V_i(s)(X_s^1,...,X_s^n),$$

where for a $C^1$ B-polynomial $V$ ($X=(X^1,...,X^n)$): \begin{align*}F_V(t,X)=\frac{\partial}{\partial t}V(t,X)+\sum_j(\partial_jV)(t,X)\#V_j(t,X)+ m(1\o \eta E_B\o 1)(\partial_i\o 1\partial_i)(V)(t,X).\end{align*}
\end{theorem}

\begin{proof}
We have to prove that $\tau(\langle S_i,\partial P(X_t^1,...,X_t^n)\rangle_{B\langle X_t \rangle})=\tau(\xi_t^i P(X_t^1,...,X_t^n))$ for a B-non-commutative polynomial $P$, $\partial$ is the $B-\eta$-free difference quotient, recall $S_i$ is the formal variable in the bimodule where $X_t^i$ is sent. Let us write $\delta_s$ the following (Malliavin) Derivation operator defined on B-non-commutative polynomials in $X_u^i$'s (as usual one can assume them algebraically free without loss of generality).
$$\delta_s(P(X_{s_1}^{i_1},...,X_{s_n}^{i_n}))=\sum_{j}(\partial_{(j)}(P))(X_{s_1}^{i_1},...,X_{s_n}^{i_n}) (s\wedge s_j),$$
where $\partial_{(j)}$ is the $B-\eta$-free difference quotient in the j-th variable for $P$ (sending $X_{s_j}^{i_j}$ to $S_{i_j}$). Obviously, $\delta_t P(X_t^1,...,X_t^n)=t\partial P(X_t^1,...,X_t^n)$ so that it suffices to show~:
$$\tau(\langle S_i,\delta_s P(X_s^1,...,X_s^n)\rangle_{B\langle X_s \rangle})-\tau(\Xi_s^iP_s)=0,$$
for $\Xi_s^i=X_s^i-X_0^i-\int_0^sdt\ t F_{V_i}(t,X_t^1,...,X_t^n)-sV_i(s)(X_s^1,...,X_s^n),$ and any non-commutative polynomial $P_s=P(X_s^1,...,X_s^n)$. 
We will first prove using Ito formula a differential equation for the above differences.

Applying Ito formula, one gets ($\partial_{j}$ the ordinary difference quotient ):
\begin{align*}P_t&=P(X_t^1,...,X_t^n)=P(X_0^1,...,X_0^n)+\\ &\sum_i\int_0^tds \ \partial_i(P)(X_s^1,...,X_s^n)\#V_i(s,X_s)+\sum_j m\circ 1\o \eta E_B\o 1(\partial_i\o1\partial_i)(P)(X_s^1,...,X_s^n)\\ &+\int_0^t \partial(P)(X_s^1,...,X_s^n)\#dS_s.\end{align*}
Let us write $\beta_s=\sum_i\partial_i(P)(X_s^1,...,X_s^n)\#V_i(s,X_s)+\sum_j m\circ 1\o \eta E_B\o 1(\partial_i\o1\partial_i)(P)(X_s^1,...,X_s^n)$.

Thus, let us compute likewise~: $\tau(P_t(X_t^i-X_0^i))=\int_0^tds \tau(P_s V_i(s,X_s)+\beta_s(X_s^i-X_0^i)+\langle S_i, \partial(P)(X_s^1,...,X_s^n)\rangle_{B\langle X\rangle}).$

\begin{align*}\tau(P_t tV_i(t,X_t))&=\int_0^tds\  \tau(P_s V_i(s,X_s)+P_s s F_{V_i}(s,X_s^1,...,X_s^n)+\beta_s sV_i(s,X_s))\\ & +\int_0^tds\tau(\langle \partial(P^*)(X_s^1,...,X_s^n),\partial(sV_i(s,X_s) )\rangle_{B\langle X\rangle}).\end{align*}

Thus $$\tau(P_t\Xi_t^i)=\int_0^tds\ \tau(\beta_s\Xi_s^i)+\tau(\langle S_i,\partial(P)(X_s^1,...,X_s^n)\rangle_{B\langle X\rangle}-\langle \partial(P^*)(X_s^1,...,X_s^n),\partial(sV_i(s,X_s) )\rangle_{B\langle X\rangle}).$$

We have to compare this with \begin{align*}\tau&(\langle S_i,\delta_t P(X_t^1,...,X_t^n)\rangle_{B\langle X_s \rangle})=\int_0^t ds\tau(\langle S_i,\partial P(X_s^1,...,X_s^n)\rangle_{B\langle X_s \rangle})+\int_0^tds\tau(\langle S_i,\delta_s \beta_s\rangle_{B\langle X_s \rangle})\\ &-\sum_j\int_0^tds \tau(\langle S_i,\partial_j(P)(X_s^1,...,X_s^n)\#\delta_s^i V_j(s,X_s)\rangle_{B\langle X_s \rangle}-\langle S_j,\partial_j(P)(X_s^1,...,X_s^n)\#\sigma\delta_s^j V_i(s,X_s)\rangle_{B\langle X_s \rangle})\\ &-\int_0^tds \tau(\langle \partial(P^*)(X_s^1,...,X_s^n),\delta_s V_i(s,X_s)\rangle_{B\langle X_s \rangle})\end{align*}
We can note that the second line vanish by our assumption (and the diagonal form of the covariance), so that summing up we have obtained our ``differential equation"~:$$\tau(P_t\Xi_t^i)-\tau(\langle S_i,\delta_t P(X_t^1,...,X_t^n)\rangle_{B\langle X_s \rangle})=\int_0^tds\ \tau(\beta_s\Xi_s^i)-\tau(\langle S_i,\delta_s \beta_s\rangle_{B\langle X_s \rangle}).$$

Let us call $M_n:=n\sup_{s\in[0,T]}\sum_i ||b_i(s)||=Cn$ where $b_i(s)$ are the coefficients of all $V_j(s)$, as explained earlier the norm is the projective norm making those coefficients $C^1$ in $s$. Let $p$ be the maximum degree of $V_j(s)$.

 Let $\tilde{M}_n:=M_n+2n(\frac{R^p}{R^p-1})^2=Dn$. Finally, let $\theta$ a time such that for all monomial $P$ $\tau(P_t\Xi_t^i)-\tau(\langle S_i,\delta_t P(X_t^1,...,X_t^n)\rangle_{B\langle X_s \rangle})=0.$
 
 Let us show quickly that for $P$ monomial of degree less than $n=kp$ (with coefficient in projective norm less than 1 (for $t\geq\theta$ since by definition the left hand side is 0 before)~: $$\tau(P_t\Xi_t^i)-\tau(\langle S_i,\delta_t P(X_t^1,...,X_t^n)\rangle_{B\langle X_s \rangle})\leq \frac{(t-\theta)^l}{l!}R^{(k+l)p}(C+Tp(k+l))\prod_{i=0}^{l-1}\tilde{M}_{(k+i)p}=:A_l(t,k),$$
 
 where $C=sup_{[0,T]}||\Xi_t^i||<\infty$. We prove this by induction on $l$, initialization is obvious by boundedness of $X_t$ by $R\geq 1$.
 
 To prove induction step, note that $\beta_s(P)$ is a sum of (less than $n$ times the number of coefficients of $V_i$) monomials of degree less than $(k+1)p$, plus at most $np$ monomials of degree less than $kp$, plus again at most $np$ monomials of degree less than $(k-1)p$ etc up to degree $0$. Since this list of coefficients comes from second derivatives of $P$, those of degree between $(k-1)p$ and $kp$ has coefficients bounded by $1, R, ...,R^{p-1}$ depending on the number of elements projected on $B$. At the end, using induction hypothesis one thus gets~:\begin{align*}\tau(P_t\Xi_t^i)-\tau(\langle S_i,\delta_t P(X_t^1,...,X_t^n)\rangle_{B\langle X_s \rangle})&\leq\int_\theta^tds A_l(s,k+1)M_{kp}+\sum_{i=1}^k A_l(s,i)n\frac{R^p-1}{R-1}\leq A_{l+1}(t,k).\end{align*}
 
where we noted (before using our notations) \begin{align*}&\sum_{i=1}^k A_l(t,i)\leq \frac{(t-\theta)^l}{l!}R^{lp}\prod_{i=0}^{l-1}\tilde{M}_{(k+i)p}\sum_{i=1}^k R^{ip}(C+Tp(i+l))\\&\leq \frac{(t-\theta)^l}{l!}R^{lp}\prod_{i=0}^{l-1}\tilde{M}_{(k+i)p}\left((C+Tpl)\frac{R^{(k+1)p}-1}{R^p-1}+ Tp((k+2)\frac{R^{p(k+2)}}{R^{p}-1}-R^{p}\frac{R^{p(k+2)}-1}{(R^{p}-1)^2})\right)\\ &\leq \frac{(t-\theta)^l}{l!}R^{(k+l+1)p}\prod_{i=1}^{l}\tilde{M}_{(k+i)p}2\frac{R^{p}}{R^{p}-1}(C+Tpl+Tp((k+1)).\end{align*}
But now, $A_{l}(t,k)\leq \frac{(DpR^p(t-\theta))^l(k+l-1)!}{l!(k-1)!}R^{kp}(C+Tp(k+l))\leq (C+Tp(k+l))2^{k-1}(2DpR^p(t-\theta))^l\to_{l\to\infty} 0$ when $t-\theta<1/2DpR^p$ value independent of $k$, so that one easily deduces by induction one can take $\theta=T$.
\end{proof}
 
\subsection{Transportation cost inequality}
 Let us recall the definition of free entropy with respect to a completely positive map $\eta:B\to B$ as in \cite{ShlyFreeAmalg00} (with a corrected typo concerning normalization):
\begin{align*} \chi^{*}(X_{1},...,X_{n};B,\eta)=\frac{1}{2}&\int_0^\infty\left(\frac{n\tau(\eta(1))}{1+t}-\Phi^*(X_1+\sqrt{t}S_1,...,X_n+\sqrt{t}S_n;B,\eta)\right)ds\\ &+\frac{n}{2}\log (2\pi e)\tau(\eta(1)),
\end{align*}
where $\{S_1,...,S_n\}$ is a B-free semicircular system of covariance $H=Diag(\eta)$, free with amalgamation over $B$ with $\{X_1,...,X_n\}$. Let us also remind the Fisher information 

\noindent $\Phi^{*}(Y_1,...,Y_n;B,\eta)=\sum_i||\xi_i||_2^2$ where $\xi_i$ are the conjugate variables relative to $B,\eta$ of \cite{ShlyFreeAmalg00}, already used earlier.
With our notations (and again a corrected typo), Proposition 8.3 $(a)$ of \cite{ShlyFreeAmalg00} becomes for $C^2=\sum_i\tau(X_i^2)$~:

$$\chi^{*}(X_{1},...,X_{n};B,\eta)\leq \frac{n\tau(\eta(1))}{2}\log(\frac{2\pi e C^2}{n\tau(\eta(1))})\leq \frac{n\tau(\eta(1))}{2}\log(2\pi e)+\frac{C^2-n\tau(\eta(1))}{2}.$$

The second inequality above, coming from a usual concavity inequality of $\log$ is used to prove positivity of the right hand side of the next free Talagrand inequality:
\begin{theorem}\label{Tal}
If $(S_{1},...,S_{n})$ is a family of B-semicircular elements with covariance $\eta$ free with amalgamation over $B$, then for any non-commutative variables $Y_{1},...,Y_{n},B$ (with a joint law with B as in the definition of the corresponding Wasserstein distance).
\begin{align*}d_{W}(&(Y_{1},...,Y_{n}),(S_{1},...,S_{n}):B) \leq \\ & \sqrt{2}\left( \chi^{*}(S_{1},...,S_{n}:B,\eta)-\chi^{*}(Y_{1},...,Y_{n}:B,\eta)-\frac{n\tau(\eta(1))}{2}+\frac{1}{2}\sum_{i=1}^{N}\tau(Y_{i}^{2}))\right)^{1/2}.\end{align*}
\end{theorem}
\begin{proof}
It is well known that the following SDE as in Theorem \ref{boundConj} has a solution $$X_{i}(t)=Y_i-\frac{1}{2}\int_{0}^{t}X_i(s)ds + S^i_t,$$ 

With $S^i_t$ a B-free semicircular family of covariance $\eta$ free with amalgamation over $B$ and moreover in law $X_{i}(s)\simeq e^{-s/2}Y_{i}+\sqrt{1-e^{-s}}S_{i}'$.
Since $X_i(s)$ is continuous in $||.||_2$, we can apply theorem \ref{infEst} to get ($t\geq s$):
$d_{W}((X_{1}(t),...,X_{n}(t)),(X_{1}(s),...,X_{n}(s))^{2}\leq (t-s)^{2}/4 I(X_{1}(s),...,X_{n}(s))+o((t-s)^{2}),$
the latter quantity $I(X_{1},...,X_{n})=\Phi^{*}(X_{1},...,X_{n};B,\eta)-2n\tau(\eta(1))+\sum_{i=1}^{N}\tau(X_{i}^{2})=\sum_{i=1}^{N}||X_{i}-\xi_{i}||_{2}^{2}$ being a variant of free Fisher information $\Phi^{*}(X_{1},...,X_{n};B,\eta)$.

From now on, the proof follows the adaptation of the argument of Otto and Villani \cite{OV} by Biane and Voiculescu in the non-commutative one variable case, this inequality giving their inequality (4) in \cite{BV01}  (with a crucial improvement that we have no $sup_{u\in(s,t)}$ on the right, otherwise the unknown continuity of $\Phi^{*}$ in multi-variable case would have been a problem). From now on we write $X(t)=(X_{1}(t),...,X_{n}(t))$ and $S=(S_{1},...,S_{n})$

As in their lemma 2.7, one deduces~: $$\limsup_{\epsilon\rightarrow 0+}\frac{1}{\epsilon}|d_{W}(X(t+\epsilon),S)-d_{W}(X(t),S)|\leq \frac{1}{2}I(X(t))^{1/2}.$$

Let us write 

\noindent $\Sigma(Y_{1},...,Y_{n})=\left( \chi^{*}(S_{1},...,S_{n}:B,\eta)-\chi^{*}(Y_{1},...,Y_{n}:B,\eta)-\frac{n\tau(\eta(1))}{2}+\frac{1}{2}\sum_{i=1}^{N}\tau(Y_{i}^{2})\right).$

We already noticed~:$\Sigma(Y_{1},...,Y_{n})\geq 0$.
Using (an obvious variant) of Proposition 7.5 b of \cite{Vo5} for the right derivative of $\chi^{*}$ along a semicircular translation, one deduces again as in \cite{BV01} , that the right derivative of $\Sigma(X_{1}(t),...,X_{n}(t))$ is $-\frac{1}{2}I(X_{1}(t),...,X_{n}(t))
$. As a consequence, 
\begin{align*}\liminf_{\epsilon\rightarrow 0+}\frac{1}{\epsilon}&\left(d_{W}(X(t+\epsilon),S)-(2\Sigma(X(t+\epsilon))^{1/2}-d_{W}(X(t),S)+(2\Sigma(X(t)))^{1/2}\right)\\ &\geq -\frac{1}{2}I(X(t))+\frac{1}{\sqrt{8}}I(X(t))(\Sigma(X(t)))^{-1/2}\geq 0.\end{align*}

where the last inequality comes from the logarithmic Sobolev like inequality of Proposition 7.9 in \cite{Vo5} (which as an extension to amalgamated $\eta$ context since the proof only uses a free Stam inequality proven in \cite{ShlyFreeAmalg00} Proposition 4.5
)~:
\begin{align*}\chi^*(Y_{1},...,Y_{n}:B,\eta)&\geq \frac{n\tau(\eta(1))}{2}\log(\frac{2\pi n e\tau(\eta(1))}{\Phi^*(Y_{1},...,Y_{n}:B,\eta)})\\ &\geq \frac{n\tau(\eta(1))}{2}\log(2\pi n e)-\frac{1}{2}(\Phi^*(Y_{1},...,Y_{n}:B,\eta)-n\tau(\eta(1))),\end{align*}
from what follows(the inequality we used) $\Sigma(X(t))\leq \frac{1}{2}I(X(t)).$

Since moreover $X(t)$ is continuous, so is $d_{W}(X(t),S)$ and likewise since so is $\Sigma(X(t))$ via 7.5 b of \cite{Vo5}, one deduces that $d_{W}(X(t),S)-(2\Sigma(X(t)))^{1/2}$ is non-decreasing.
But the semi-continuity of $\chi^{*}$ in Proposition 7.4 of \cite{Vo5} (in our context one replaces in the proof proposition 6.10 of \cite{Vo5} by proposition 4.7 in  \cite{ShlyFreeAmalg00}, see also \cite{BS},\cite{B} for an interpretation as logarithmic-Sobolev inequality) implies 
$\limsup_{t\rightarrow\infty} d_{W}(X(t),S)-(2\Sigma(X(t)))^{1/2}\leq -(2\Sigma(S))^{1/2}=0,$
so that for any $t>0$~: $d_{W}(X(t),S)-(2\Sigma(X(t)))^{1/2}\leq 0.$

The continuity of $\Sigma(X(t))$ at $t=0$ (via Proposition 7.5 of \cite{Vo5}) and continuity of $X(t)$ normwise imply we can take the limit $t\rightarrow 0$ to get the result.


\end{proof}

\begin{remark}
One may be interested in proving an analogue of Theorem 2.2 in \cite{HiaiUeda} for $\chi^*$. Of course, an approach following \cite{BV01}  would consider a SDE with as polynomial drift cyclic derivatives of the potential. For convex potentials as in   \cite{HiaiUeda}, solutions of those SDE have been proven to converge at infinity in \cite{SG}. Our previous section enables us to get local estimates on Wasserstein distance. At this time, the main lacking piece is an unknown change of variable for $\chi^*$ preventing us from computing the derivative of entropy along paths of the SDE.
\end{remark}

\section{Weak solutions of various stationary SDEs}

Given $A$,$\delta$  as in part 2 $(\tilde{M},\tau)$ given by theorem \ref{main}.
We want to prove it satisfies a stochastic differential equation weakly (recall that, in probabilistic literature, solving strongly a SDE means the filtration in which we build the solution is the filtration of the Brownian motion, equivalently (in our context), the solution is in a free product of a Brownian motion von Neumann algebra and the initial condition von Neumann algebra, solving it weakly means the filtration where we build the process is of course adapted to the Brownian motion so that Ito integral makes sense but this may be a larger filtration than the Brownian one).

We will need for this a general assumption. We consider $B\subset M$ a von Neumann sub-algebra with $A(B)=0$, $\delta(B)=0$. We assume the bimodule of value $\H$ of $\delta$ is given by a completely positive map $\eta:B\rightarrow M_m(B)$ in a standard way $\H=H(M;\eta\circ E_B)$ (with the notation introduced in theorem \ref{FDQ} We write $\Xi_1... \Xi_m$ the generating set. Especially any derivation in the coarse correspondence ($B=\C$) satisfy this assumption)

We also consider the usual extension $\Delta^1$ of $\Delta:M\to L^1(M)$, (see e.g. \cite{Pe04} section 1.4 where it is called $\Psi_\delta$).

\begin{theorem}\label{main2}
We keep the above assumption 
and assume moreover $A=\Delta$ and either (case 1) $\delta^*\Xi_i\in L^2(M)$ or (case 2) $\eta=Diag(\tilde{\eta})$, $\tilde{\eta}:B\to B$, B with separable predual and 
the set of regular elements $R=\{ X_i\in D(\Delta^1)\ | \ \delta(X_i)\in H(B'\cap M;\eta\circ E_B)\}$ contain a countable subset $\cC$ such that
 the algebra generated by B and $\cC$ is a core for $\delta$,  (this especially includes any derivation in the case $B$ central if M has separable predual, e.g. any derivation valued in the coarse).

There exists a $W^{*}$-probability space $(\hat{M},\tau)$ (containing $\widetilde{M}$ with agreeing traces in case (1)) such that there is in $\hat{M}$ a natural compatible filtration and an adapted $B$-free Brownian motion $\hat{S}_{t}$ with covariance $\eta$. 
  For any $X\in M_{0}\cap D(\Delta^1)$:
$$X_{t}=X_{0}-\frac{1}{2} \int_{0}^{t} (\Delta^1(X))_{s}ds + \int_{0}^{t}(\delta(X))_{s}\#d\hat{S}_{s}.$$
Moreover, for any $X\in M_{0}$:$$X_{t}=\phi_t(X_{0})+ \int_{0}^{t}(\delta\phi_{t-s}(X_s))_{s}\#d\hat{S}_{s}$$
\end{theorem}
$\hat{M}$ will be produced as a GNS construction for a larger path space associated to another non-symmetric Dirichlet form (even to solve SDEs with symmetric $A$, thus motivating the earlier extra work in this context ).
One could also derive from the SDEs the equations defining our states to show $\hat{M}$ contain $\widetilde{M}$ in case (2), we leave this to the reader.

\begin{remark}
 For instance, since it is easy to show a $L^1$ conjugate variable imply the free difference quotient closable as $L^2$ derivation
 , our theorem, always applies to free difference quotient with  $L^1$ conjugate variable, and thus extend Theorem \ref{FDQ} to this case (but in case with a completely positive map, this extends only the diagonal map case and generally only the symmetric Dirichlet form case).
\end{remark}

\begin{proof}
First note that if $M$ has separable predual and B central, $R=D(\Delta^1)\cap M$, taking a countable dense set D in $L^2(M)$ inside M, $\cC=\{\eta_{n}(x)\ | x\in D,\ n\in \N\}$ is a core for $\delta$ with the right properties.

\begin{step}Avoiding to assume $\delta^*\Xi_i\in L^2(M)$ in case 2 
.\end{step}When we are assuming $\delta^*\Xi_i\in L^2(M)$, one may replace in what follows $M^{(\epsilon)},\overline{M}$ by $M$ and get a more direct result. This step is inspired from \cite{Shly03} (subsection 2.2, basically the case $B=\C$).

Let us call $\cC=\{X_1,...,X_n\}$ (usually n countable).

One may assume without changing properties of $C$ that $||X_i||\leq 1$, $\sum_i ||\delta(X_i)||^2<\infty$, one may for instance replace $X_i$ by $X_i/(||X_i||+2^i ||\delta(X_i)||)$ since we only claim a property of the algebra generated by $\cC$ and $B$.

But here, let $\overline{\eta}=Diag(\tilde{\eta}):B\to M_n(B)$ (as before if $n$ countable, $M_n(B)=B\o B(H)$ $H$ Hilbert space of dimension $n$) $\overline{M}=\Phi(M,\overline{\eta}\circ E_B)$. Especially  
$\overline{M}=W^*(B,X_1,...,X_n,\overline{S_1},...,\overline{S_n})$ 
We can replace $\H$ by $\overline{\H}=H(\overline{M};\eta\circ E_B)$.
Since in this case, $\overline{\H}$ is a direct sum of $\tilde{\H}=H(\overline{M};\tilde{\eta}\circ E_B)$ (say generated by $\Xi$), the i-th component generated by $\Xi_i$ in our previous notations. We can call $\delta_i$ (valued in $\tilde{\H}$) the component of $\delta$ on $\Xi_i$. Let us call $\tilde{\delta}_i:B\langle X_1,...,X_n,\overline{S_1},...,\overline{S_n}\rangle \to \tilde{\H}$ the derivation defined by $\tilde{\delta}_i(B)=\tilde{\delta}_i(X_j)=0$
 $\tilde{\delta}_i(\overline{S_j})=\delta_i(X_j)$. We can also extend $\delta_i$ via ($\delta_i(\overline{S_j})=0$). Note that for any $X\in M\cap D(\Delta_i^1)$, any $P\in M\cap D(\delta_i)\langle \overline{S_1},...,\overline{S_n}\rangle$ $\tau(\Delta_i(X)P)=\langle \delta_i(X),\delta_i(P)\rangle$ (as in example 2.4 in \cite{Shly03}, using here $\overline{S_j}$ free with amalgamation with $M$ over $B$)
 
 We claim that $\tilde{\delta}_i^*(\Xi)=Y_i:=\sum_j(\sigma(\delta_i(X_j))\# \overline{S_j})^*\in L^2$ by our previous assumption on $X_i$, where $\sigma$ is the isometric flip operator extended from $\sigma(a\Xi b)=(b\Xi a)$ (isometry on $H(B'\cap M;\tilde{\eta}\circ E_B)$ come from $\tilde{\eta}$ $\tau$-symmetric, and $E_B(ab)=E_B(ba)$ for any $a\in B'\cap M$), $a\Xi b\# \overline{S_j}=a\overline{S_j} b$ is extended by linearity and isometry.
 Indeed, for $w_{j}\in B\langle X_{1},\ldots ,X_{n}\rangle$ and $i_{1},\ldots ,i_{n}\in \{1,\ldots ,n\}$, let's write $\delta_j(X_i)=\sum _{k}a_{k}^{ij}\Xi b_{k}^{ij}$ ($L^2$ sum, we use in the third line bellow $a_{k}^{ij},b_{k}^{ij}\in B'$ commute with $B$)\begin{eqnarray*}
\tau (Y_{j}^*w_{0}\overline{S_{i_{1}}}w_{1}\cdots w_{k-1}\overline{S_{i_{k}}}w_{k}) & = & \sum _{i}\sum _{s}\tau (b_{s}^{ij}\overline{S_{i}}a_{s}^{ij}w_{0}\overline{S_{i_{1}}}w_{1}\cdots w_{k-1}\overline{S_{i_{k}}}w_{k})\\
 & = & \sum _{i,\ s}\sum _{p,\ i_{p}=i}\tau (b_{s}^{i_{p}j}\tilde{\eta}(E_{B}(a_{s}^{i_{p}j}w_{0}S_{i_{1}}w_{1}\cdots w_{p-1}))w_{p}S_{i_{p+1}}\cdots S_{i_{k}}w_{k})\\
 & = & \sum _{i,\ s}\sum _{p,\ i_{p}=i}\tau (\tilde{\eta}(E_{B}(w_{0}S_{i_{1}}w_{1}\cdots w_{p-1}a_{s}^{i_{p}j}))b_{s}^{i_{p}j}w_{p}S_{i_{p+1}}\cdots S_{i_{k}}w_{k})\\
 & = & \langle \Xi,\tilde{\delta}_j(w_{0}S_{i_{1}}w_{1}\cdots w_{k-1}S_{i_{k}}w_{k})).
\end{eqnarray*}

But now if we consider $M^{(\epsilon)}=W^*(B,X_1+\sqrt{\epsilon}\ \overline{S_1},...,X_n+\sqrt{\epsilon}\ \overline{S_n})$. For $P\in B\langle X_1+\sqrt{\epsilon}\ \overline{S_1},...,X_n+\sqrt{\epsilon}\ \overline{S_n}\rangle.$ $\delta_i(P)=\frac{1}{\sqrt{\epsilon}}\tilde{\delta}_i(P)$
so that when we restrict $\delta_i$ to $B\langle X_1+\sqrt{\epsilon}\ \overline{S_1},...,X_n+\sqrt{\epsilon}\ \overline{S_n}\rangle,$ (let say we call it $\delta_i^{(\epsilon)}$) we get $\delta_i^{(\epsilon)*}(\Xi)=\frac{1}{\sqrt{\epsilon}}E_{M^{(\epsilon)}}(Y_i)\in L^{2}(M^{(\epsilon)}).$ Note also that $\delta_i^{(\epsilon)*}\delta_i^{(\epsilon)}(X_j+\sqrt{\epsilon}\ \overline{S_j})=E_{M^{(\epsilon)}}(\delta_i^*\delta_i(X_j))\in L^1$ (using the remark above on the adjoint of the extension $\delta_i$ to $\overline{M}$). 

We will still call $\delta=\oplus\delta_i^{(\epsilon)}:L^{2}(M^{(\epsilon)})\to \oH$. 
We can also consider $\tilde{\delta}=\oplus\frac{1}{\sqrt{\epsilon}}\tilde{\delta}_i:L^{2}(\overline{M})\to \oH$ as an extension. We will recall later a formula due to Voiculescu for adjoints of derivations, implying $\tilde{\delta}_i^*$ is densely defined, e.g. on $B\langle X_1,\overline{S_1},...,X_n,\overline{S_n}\rangle^{\o 2}$, of course we have then $\delta^*U=E_{L^{2}(M^{(\epsilon)})}\tilde{\delta}^{*}U$, for $U\in D(\tilde{\delta}^{*})$ so that $\delta^*$ is also densely defined as seen the previous domain. $\delta$ being a real closable densely defined derivation, we are thus (almost) in the previous context, especially $\Delta=\delta^*\delta$, $\Delta^1$ are again defined. 
So that we have turned back to our assumption $\delta^*\Xi_i\in L^{2}(M^{(\epsilon)}).$ Let us emphasize the situation is not exactly the same since as a bimodule $\oH$ is not produced by completion starting from $M^{(\epsilon)}$, what we call in two lines $\He$ (this will thus not be clear what we could mean as evolution with time of the integrand in stochastic integral of the SDE, before projection on this sub-bimodule). We again warn the reader that, when we are assuming $\delta^*\Xi_i\in L^2(M)$ (case 1), we replace in what follows $M^{(\epsilon)},\overline{M}$ by $M$ and get a more direct result  likewise in that case $\oH=\He=\H$. We don't write anything else (especially what follows includes a case of non-diagonal $\eta$ via case 1).

\begin{step}Definition of a non-symmetric Dirichlet form\end{step}
Consider $\Me=\Phi(M^{(\epsilon)},\eta\circ E_B)$,  the $M^{(\epsilon)}$-semicircular valued system of \cite{SAVal99}, likewise $\oM=\Phi(\overline{M},\eta\circ E_B)$. Especially  
$\hat{M}^{(\epsilon)}_0=W^*(M^{(\epsilon)},S_1...,S_m)$ and the $L^2$ completion of $M^{(\epsilon)}\o M^{(\epsilon)}\# S_1+...+M^{(\epsilon)}\o M^{(\epsilon)}\# S_m$ is isomorphic to a $\He=H(M^{(\epsilon)};\eta\circ E_B)$ subspace of $\oH$. Let us call as before  $\Xi_i$ the element of $\oH$ corresponding to $S_i$. We call $\hH=H(\Me;\eta\circ E_B)$ the natural extension of $\He$, subspace of $\ohH=H(\oM;\eta\circ E_B)$ the natural extension of $\oH$.


We can extend $\delta$ to a derivation on $\Me\to \ohH$ with $\delta(S_i)=0$ with the same value of $\delta^{*}\Xi_i$. Indeed, for a monomial $P$ in $S_i$'s and several $X_q\in D(\delta)\cap M^{(\epsilon)}$ (let say algebraically free and generating $M^{(\epsilon)}$ without loss of generality) a wick type  
formula enables to rewrite $\tau(P \delta^{*}\Xi_i)$ as a trace in $M$ of a product $P'$ of $X_q$'s with elements of $B$ obtained with several applications of $\eta\circ E_B$ corresponding to non crossing pairings of $S_i$'s. Since $\delta(b)=0, b\in B$ this is again rewritten as a sum in $\ohH$ of several $\langle\partial_{X_q:B}(P')\#\delta(X_q),\Xi_i\rangle_{\oH}$ ($\partial$ a free difference quotient). But now 
fixing the element $X_q$ of application of $\delta$ and summing all applications of $\eta$ around, 
we have in fact obtained $\langle(E_{\overline{M}}\o E_{\overline{M}})(\partial_{X_q:B}(P))\#\delta(X_q),\Xi_i\rangle_{\oH}$. By definition of $\hat{H}_2$  using $\eta\circ E_B=\eta\circ E_B\circ E_{\overline{M}}$, this is exactly 
$\langle(\partial_{X_q:B}(P))\#\delta(X_q),\Xi_i\rangle_{\ohH}$. Summing over $q$ this concludes. The details are left to the reader (in the case of diagonal $\eta$ this is really similar to what we did in step 1).

We can consider $\partial=(\partial_1,...,\partial_m)$ a free difference quotient of $S_i$ (vanishing on $M^{(\epsilon)}$, $\partial_j(S_i)=1_{i=j} \Xi_i$) valued in $\hH\subset \ohH$ as in theorem \ref{FDQ}. Note it is well known that $\partial^*\Xi_i=S_i$.

 We also want to consider two derivations valued in $L^{2}(\Me)$, the first is defined via $\delta\#S(S_i)=0$, $\delta\#S(x)=\delta(x)\#S$, $x\in M^{(\epsilon)}\cap D(\delta)$ which is well defined since $a\Xi_ib\rightarrow aS_ib=:a\Xi_ib\#S$ extends to an isometry of $\He$ to $L^{2}(\Me)$. The second is defined on $M^{(\epsilon)}\langle S_1,...,S_m\rangle$ via $\tilde{\partial}(m)=0$, $m\in M^{(\epsilon)}$ , $\tilde{\partial}(S_i)=\delta^{*}\Xi_i\in L^2(M^{(\epsilon)})$. It has obviously an extension to $\overline{M}\langle S_1,...,S_m\rangle$.

Consider $A_{\beta}=(\delta+\partial)^{*}(\delta+\partial)+\tilde{\partial}-\delta\#S+\beta\partial^{*}\partial$, actually defined via the associated form $\E_{(\beta)}(x,x)=||\delta+\partial(x)||_{2}^{2}+\beta ||\partial(x)||_{2}^{2}+\langle \tilde{\partial}(x)-\delta\#S(x),x\rangle, $ for any $x\in (M^{(\epsilon)}\cap D(\delta))\langle S_1,...,S_m\rangle.$

Let us recall we consider $\partial,\delta:L^2(\Me)\to \ohH$ as unbounded operators and want to compute their adjoints (on dense domains)
First let us note that using $\partial^*\Xi_i=S_i$
 (and using $\partial$ and $\delta$ are real derivations with obvious $*$ structure on their range (and using $M$-valued scalar products such as $\langle a\Xi_jb,c\Xi_id\rangle_{\oM}=b^*\eta_{ji}(E_B(a^*c))d$ (obviously this can be extended on $L^2$-sums) and e.g. $\langle \partial(.^*),\Xi.\rangle_{\oM}(a\Xi_ib):=\sum_j m\circ 1\o\eta_{ji}\circ E_B(\partial_j(a))b$
) we deduce from an analogue of Proposition 4.6 in \cite{Vo5} that
\begin{align*}&\partial^{*}(a\Xi_ib)=E_{L^2(\Me)}\left(aS_ib- \sum_j m\circ 1\o\eta_{ji}\circ E_B(\partial_j(a))b+a m\circ \eta_{ij}\circ E_B\o 1(\partial_j(b))\right)\\ & =E_{L^2(\Me)}\left(aS_ib- \langle \partial(.^*),\Xi.\rangle_{\oM}(a\Xi_ib) +\langle \Xi(.)^*,\partial(.)\rangle_{\oM}(a\Xi_ib)\right),\ \ \forall a,b\in \overline{M}\langle S_1,...,S_m\rangle\\ 
& 
\delta^{*}(a\Xi_ib)=E_{L^2(\Me)}\left(a(\tilde{\delta}^*\Xi_i)b- \langle \tilde{\delta}(a^*),\Xi_ib\rangle_{\oM} -\langle \Xi_ia^*,\tilde{\delta}(b)\rangle_{\oM}\right),\ \ \forall a,b\in (\overline{M}\cap D(\tilde{\delta})))\langle S_1,...,S_m\rangle\\ &=a(\delta^*\Xi_i)b- E_{L^2(\Me)}\left(\langle \delta(a^*),\Xi_ib\rangle_{\oM} -\langle \Xi_ia^*,\delta(b)\rangle_{\oM}\right),\ \ \forall a,b\in (M^{(\epsilon)}\cap D(\tilde{\delta})))\langle S_1,...,S_m\rangle.
\end{align*}

Thus, for any $x,y\in (M^{(\epsilon)}\cap D(\delta))\langle S_1,...,S_m\rangle$: \begin{align*}\langle \delta\#S(x),y\rangle&=\langle \partial^{*}\delta(x)+\langle \partial(.^*),\Xi.\rangle_{\oM}(\delta(x)) +\langle \Xi(.)^*,\partial(.)\rangle_{\oM}(\delta(x)),y\rangle\\ &=\langle \partial^{*}\delta(x)-\delta^{*}\partial(x)+\tilde{\partial}(x),y\rangle\\ \langle\delta\#S(x)-\tilde{\partial}(x),y\rangle&=\langle \delta(x),\partial(y)\rangle-\langle\partial(x),\delta(y)\rangle,\end{align*} where we used an obvious co-associativity for $\partial$ and $\delta$ in the second line. We have thus found the antisymmetric part of the form, this also immediately gives the domination via the symmetric form associated to $\partial\oplus\delta$ clearly equivalent to the symmetric part (with a constant depending on $\beta$). Thus 
 As any antisymmetric real derivation $\delta\#S-\tilde{\partial}$ is a trace preserving derivation.

As a consequence of Proposition 3.5 in \cite{MaR} $\E_{(\beta)}$ is a coercive closed form. The domain of the closure is clearly $D(\overline{\E_{(\beta)}})=D(\overline{\delta\oplus \partial})$. Let us now check it is actually a Dirichlet form using Theorem 2.8 in \cite{NSNC} (non-commutative adaptation of Proposition 4.7 in \cite{MaR} and we will actually use the variant of Proposition 4.10, where the condition can be checked only on a core of the form). For this, consider $\varphi_{\epsilon}$ a a Lipschitz increasing function in $C^{\infty}(\R,[-\epsilon,1+\epsilon])$ with $\varphi_{\epsilon}(t)=t$ on $[0,1]$, $\varphi_{\epsilon}(t)=- \epsilon$ on $(-\infty,-2\epsilon]$
and $\varphi_{\epsilon}(t)=1+\epsilon$ on $[1+2\epsilon,\infty)$. We state the result in the next~:

\begin{lemma}
For any self-adjoint $u\in (D(\delta)\cap M^{(\epsilon)})\langle S_1,...,S_m\rangle$ , any $\epsilon>0$, then $\varphi_{\epsilon}(u)$ given by functional calculus is in $D(\overline{\E_{(\beta)}})$ and $$\liminf_{\epsilon\rightarrow 0}\overline{\E_{(\beta)}}(\varphi_{\epsilon}(u),u-\varphi_{\epsilon}(u))\geq 0,$$ $$\liminf_{\epsilon\rightarrow 0}\overline{\E_{(\beta)}}(u-\varphi_{\epsilon}(u),\varphi_{\epsilon}(u))\geq 0.$$
As a consequence $\overline{\E^{(\beta)}}$ is a (non-symmetric) Dirichlet form for any $\beta$, we will write $\phi_{t}$, $\phi^{*}_{t}$ the corresponding sub-markovian semigroups.
\end{lemma}
\begin{proof}
Since $D(\overline{\E_{(\beta)}})=D(\overline{\delta\oplus \partial})$, the stability result follows from stability of the domain of a symmetric Dirichlet form by $C^{1}$- functional calculus.
Since $(D(\delta)\cap M^{(\epsilon)})\langle S_1,...,S_m\rangle$ is a core in $D(\overline{\E_{(\beta)}})$ by definition, the second statement about the Dirichlet form follows from the first via Proposition 4.10 in \cite{MaR} as we said.

 By sub-additivity of $\liminf$, and since $\liminf_{\epsilon\rightarrow 0}\langle\overline{(\delta+\partial)\oplus\sqrt{\beta} \partial}\varphi_{\epsilon}(u),\overline{(\delta+\partial)\oplus\sqrt{\beta} \partial}(u-\varphi_{\epsilon}(u))\rangle\geq 0$ by the standard result for a symmetric Dirichlet form, it suffices to show~:
$\lim_{\epsilon\rightarrow 0}\langle\delta\#S-\tilde{\partial}\varphi_{\epsilon}(u),u-\varphi_{\epsilon}(u)\rangle= 0.$
Of course, since $\tilde{\delta}:=\delta\#S-\tilde{\partial}$ is antisymmetric, the second case follows at once.
We first need only $u \in D(\overline{\E_{(\beta)}})$.
If one takes a sequence of polynomials $P_{n}$ converging in $C^{2}$ norm to $\varphi_{\epsilon}$ on the spectrum of $u$ from $\bar{\delta}(P_{n}(u))\rightarrow \bar{\delta}(\varphi_{\epsilon}(u))$ and the corresponding result for $\partial$ one deduces (from the weak sector condition), that $\langle \tilde{\delta}(P_{n}(u)), v\rangle\rightarrow \langle\tilde{\delta}(\varphi_{\epsilon}(u)),v\rangle$ for any $v\in D(\overline{\delta\oplus \partial})$.  
But $\tilde{\delta}(P_{n}(u))=L_{u}\otimes R_{u}(\partial'(P_{n}))(\tilde{\delta}(u))$ since $\tilde{\delta}$ is a derivation on $(D(\bar{\delta})\cap M^{(\epsilon)})\langle S_1,...,S_n\rangle$ ($\partial'$ the one variable difference quotient). But since $P_n$ and $\varphi_\epsilon$ are $C^{2}$ $L_{u}\otimes R_{u}(\partial'(P_{n}))(v)$ is clearly in  $D(\overline{\delta\oplus \partial})$ and as before converges in this space to $L_{u}\otimes R_{u}(\partial'(\varphi_\epsilon))(v)$. This implies as before $\langle \tilde{\delta}(P_{n}(u)), v\rangle
\rightarrow \langle L_{u}\otimes R_{u}(\partial'(\varphi_{\epsilon}))\#\tilde{\delta}(u),v\rangle$
As a consequence  we get (since we have identified the limit before)~:

$$\tilde{\delta}(\varphi_{\epsilon}(u))=L_{u}\otimes R_{u}(\partial(\varphi_{\epsilon}))\#\tilde{\delta}(u).$$

Now the expression we are interested in is equal to~: \begin{align*}\langle\tilde{\delta}\varphi_{\epsilon}(u),u-\varphi_{\epsilon}(u)\rangle&=\langle\tilde{\delta}(u),L_{u}\otimes R_{u}(\partial(\varphi_{\epsilon}))\#(u-\varphi_{\epsilon}(u))\rangle\\ &=\langle \tilde\delta(u),\varphi_{\epsilon}'(u)(u-\varphi_{\epsilon}(u))\rangle\end{align*}

The last line follows again after a check on polynomials and taking a limit. But now $\phi_{\epsilon}'(u)^{2}(u-\varphi_{\epsilon}(u))^{2}\leq 4\epsilon^{2}\ 1_{[-2\epsilon,0]\cup[1,1+2\epsilon]}(u)$ by functional calculus since 
$\phi_{\epsilon}'$ is 0 outside $[-2\epsilon,1+2\epsilon]$ and bounded by $1$, $(t-\varphi_{\epsilon}(t))^{2}$  is 0 on $[0,1]$ and bellow $t^2$ or $(t-1)^{2}$ on negative and positive numbers respectively. From the above inequality, one deduces by Dominated convergence Theorem that $(u-\varphi_{\epsilon}(u))\phi_{\epsilon}'(u)$ converges to $0$ in $L^{2}$. This concludes since for $u\in (D(\delta)\cap M^{(\epsilon)})\langle S_1,...,S_m\rangle$ we have $\tilde\delta(u)\in L^{2}$.\end{proof}

Let us remind the reader that $\{S_1,...,S_m\}$ and $M^{(\epsilon)}$ (or $\overline{M}$) are free with amalgamation over $B$ as in example 3.3 (c) of \cite{ShlyFreeAmalg00} using Theorem 2.3 of \cite{ShlyFreeAmalg98}. 
Thus any element $M^{(\epsilon)}\langle S_1,...,S_m\rangle$ can be written as sum of products of terms alternating in the two previous algebras and with zero conditional expectation onto $B$.
Let us write with a slight abuse of notation $(D(\Delta)\cap M^{(\epsilon)})\langle S_1,...,S_m\rangle$ for the subspace where terms coming from $M^{(\epsilon)}$ in the above decomposition are in $D(\Delta)$ (since $D(\Delta)\cap M^{(\epsilon)}$  is not a subalgebra, only $D(\Delta^1)\cap M^{(\epsilon)}$ is, this is not a subalgebra too).

Let us note incidentally that (almost) by definition, $(D(\Delta)\cap M^{(\epsilon)})\langle S_1,...,S_m\rangle$ 
 is a core for $\overline{\delta\oplus \partial}$.

 Moreover it is contained in the domain of $A_{\beta}$ (and of it self-adjoint part). Here we use to show it is contained in $D(\Delta_{\beta})$ not only $D(\Delta_{\beta}^1)$ that on terms with the above alternating pattern $Z=X_1Y_1X_2..., X_i\in M^\epsilon\cap D(\Delta)$, $Y_i\in B\langle S_1,...,S_m\rangle$, $E_B(X_i)=E_B(Y_i)=0$ we have $A_{\beta}(Z)=\sum X_1...Y_iA_{\beta}(X_{i+1})...+X_1...A_{\beta}(Y_i)X_{i+1}...+ X_1...\Gamma(X_i,...,Y_j)...+ X_1...\Gamma(Y_i,...,Y_j)...++ X_1...\Gamma(Y_i,...,X_j)...$ but no terms $\Gamma(X_i,...,X_j)=0$ since those terms  vanish by freeness with amalgamation, and since they were the only terms maybe in $L^1$ and not clearly in $L^2$, we got our claim. Thus the supplementary assumption on domains is satisfied too $D(A_{\beta})\cap D(\Delta_{\beta})\cap M^\epsilon$ is a core for $\overline{\delta\oplus \partial}$.
 
 We are thus in the situation of the previous part 2.

\begin{step} Getting the B-Brownian motion\end{step}
We thus define $ \tau_{\beta}$ on full path space above $\Me$. 
Intuitively we will see those traces correspond to the equation we want to solve coupled respectively with $S_{t,(\beta)}^i=S^i-\frac{1}{2} \int_{0}^{t} S_{s,(\beta)}^i(1+\beta)+2(\delta^*\Xi_i)_{s}ds  + \hat{S}_{t}^i+\sqrt{\beta}S_{t}^{i}{}'$ with $S_{t}'$ another $B$-free Brownian motion with same covariance $B$-free with the first and also adapted. Of course only $\hat{S}_{t,(\beta)}^i=\hat{S}_{t}^i+\sqrt{\beta}S_{t}^{i}{}'$ will appear. The equation is nothing but an Orstein-Ulhenbeck equation with a changed drift, special case of the previous part.

Of course we want to use Theorem \ref{FDQ}. We have $B\subset N=W^*(S_1,...,S_m)$, our derivation $(\delta+\partial)\oplus \sqrt{\beta}\partial$ acts like a free difference quotient when restricted to $N$. By definition $S_i\in D(A_\beta)$ and $A_\beta(S_i)=2\delta^*\Xi_i+S_i(1+\beta)\in L^2(\Me)$ by assumption.
$\Delta(S_i)=\delta^*\Xi_i+S_i(1+\beta)$ and $\tau((S_i)^3(A_\beta(S_i)-\Delta(S_i)))=\tau((S_i)^3(\delta^*\Xi_i))=0$ since $(\delta^*\Xi_i)\in L^2(M^{(\epsilon)})$ and we have a semicircular element (an even distribution).

We have thus checked all the assumptions and get $\hat{S}_{t,(\beta)}^{i}=S_{t}^{i}-S_{0}^{i}+\frac{1}{2} \int_{0}^{t} S_{s,(\beta)}^i(1+\beta)+2(\delta^*\Xi_i)_{s}ds$ is a $B$-free Brownian motion (of covariance $(1+\beta)\eta$).

We are now ready to send $\beta$ to zero and define the von Neumann algebra where the solution will live (maybe after letting $\epsilon$ goes to zero in case 2).
We consider the $C^*$-algebraic free product $A$ of an algebraic path space over $M^{(\epsilon)}$ with a full $C^*$-algebraic  path space  in formal variables $\hat{S}_t^i$ (of norm less than say $4t||\eta_{ii}(1)||$ at time t, i.e we take a universal free product of continuous functions of norm less than $4t||\eta_{ii}(1)||$ product over all times and indices $i$, free product defined as above via an inductive limit of finitely many times free products). We have thus a positive linear functional $\hat{\tau}_{\beta}$ corresponding to the previous $\tau_\beta$ when $\hat{S}_t^i$ is given by $\hat{S}_{t,(\beta)}^{i}$ (with the bound $4t||\eta_{ii}(1)||$ this is valid for small $\beta$'s). We consider $\hat{\tau}_{0}^{(\epsilon)}$ a corresponding weak limit point in the trace state space of $A$. 
 We define finally $\hMe$ the corresponding GNS construction of $A$ with respect to  $\hat{\tau}_{0}^{(\epsilon)}$. Again it contains $\widetilde{M^{(\epsilon)}}$ as a von Neumann subalgebra, thus in case (1) it contains $\widetilde{M}$.

In order to prove that $\hat{S}_t$ is a $B$-free Brownian motion (of covariance $\eta$) adapted to the induced filtration $\hMe_{t}$ from the canonical one on Path space $A$, one can use again Paul L\'{e}vy's theorem again. (1) is immediately true for the weak limit point, 
one can check that K we obtain in (2) can be taken independent of $\beta$. In our proof there is no $o(t-s)$ in (3) the left hand side converges easily with the weak limit, for the right hand side one checks the expression bellow $E_B$ converges weakly in $L^2(B)$ (this is a statement about a sequence of linear forms, B being invariant). Thus by Banach-Sachs Theorem a Cesaro mean converges strongly and writing two $E_B$, everything under $\tau|B$ a Cesaro mean converges thus we have equality at the limit. 

\begin{step}Checking the SDEs (concluding in case 1)\end{step}
It only remains to check the stochastic equation since of course by construction all $\E_{(\beta)}$ agree with $\langle\delta^{*}\delta .,.\rangle $ on $M_{0}$ (note $\delta\#S$ cancels the term coming from change of conjugate variable, that's why we put this term in $A_\beta$)  Thus the state agrees with the previously built one, giving $\tilde{M^{(\epsilon)}}$ (built from $M^{(\epsilon)}$) as a sub-algebra. For checking the SDE, we prove of course that the $L^{2}$-norm of the $\beta$ variant goes to zero. Take 
 $X\in D(\Delta^1)\cap M^{(\epsilon)},X'\in  D(\Delta)\cap M^{(\epsilon)} \Delta(X')=A_{\beta}(X')=A_{0}(X')\in M^{(\epsilon)}$. 

We have first to check that every terms in the equation makes sense and we can apply the previous idea to prove the equation.
First $\int_{0}^{t} (\Delta(X'))_{s}ds$ easily makes sense as an integral of a continuous function with value $L^2(\widetilde{M^{(\epsilon)}})$. This is especially a limit of Riemann sums.
Taking first $X'=\eta_{\alpha}(X)$ and applying DCT (using $\eta_{\alpha}^1\Delta^1(X)=\Delta\eta_{\alpha}(X)$ as proven by duality $\eta_{\alpha}^1$ the extension of the resolvent $\eta_{\alpha}=\alpha/(\alpha+\Delta)$ to $L^1$) to get pointwise convergence and domination),  $\int_{0}^{t} (\Delta^1(X))_{s}ds$ make sense as a Bochner integral in $L^{1}(\widetilde{M^{(\epsilon)}})$ and is the limit in $L^{1}(\widetilde{M^{(\epsilon)}})$ of $\int_{0}^{t} (\Delta\eta_\alpha(X))_{s}ds$.

Recall that (as in \cite{BS98}) $\int_{0}^{t}(E_{\He}(\delta(X)))_{s}\#d\hat{S}_{s}$ is defined via an isometry between $\B_2^a$ and $L^{2}(\hMe)$, where $\B_2^a$ is the closure in $L^{2}([0,\infty),H(\hMe,\eta\circ E_B))$ of the corresponding adapted step functions (Indeed we deduce easily the isometry property for the stochastic integral from the definition of B-free Brownian motion).
Of course we can take $\Theta_n \in M^{(\epsilon)}\o_{alg} M^{(\epsilon)}\#\Xi_1+...+M^{(\epsilon)}\o_{alg} M^{(\epsilon)}\#\Xi_m$ converging to $E_{\He}(\delta(X))$ in $\He$. It thus suffices to prove $s\rightarrow (\Theta_n)_s$ is in $\B_2^a$. Again this follows from continuity of $s\rightarrow (\Theta_n)_s$ in $H(\hMe,\eta\circ E_B))$ since  for instance for $t>s$, $||(a\Xi_i b)_s-(a\Xi_i b)_t||^2=2||(a\Xi_i b)||^2-2\Re \tau(b^*\eta_{ii}E_B(a^*\phi_{t-s}(a))\phi_{t-s}(b))$.
We also want to check $\int_{0}^{t}(\delta(\phi_{t-s}(X))))_{s}\#d\hat{S}_{s}$ is well defined in case 1 (thus no projection on $\He$ necessary). Obviously we can restrict first to $[0,t-\epsilon)$ $\delta(\phi_{t-s}(X))$ is continuous with value $\H$ and approximate it by a step function $f$, so it suffices to approximate $(f(s))_s$ for the integrand of the stochastic integral, and this again is as in the constant case.

Again everything applies at level $\beta$ and we can check the various bounds are uniform in $\beta$ for small $\beta$. Thus 
\begin{align*}||-X_{t}+&X_{0}-\frac{1}{2} \int_{0}^{t} (A_\beta(X'))_{s}ds + \int_{0}^{t}(E_{\He}(\delta(X)))_{s}\#d\hat{S}_{s,(\beta)}||_2^2\\ &\rightarrow||-X_{t}+X_{0}-\frac{1}{2} \int_{0}^{t} (A(X'))_{s}ds + \int_{0}^{t}(E_{\He}(\delta(X)))_{s}\#d\hat{S}_{s}||_2^2\end{align*}
 since we can express this in terms of scalar products replace uniformly in $\beta$ the integrals via (corresponding) sums and get the remaining part by the convergence of states. We have a corresponding result for the second stochastic integral expression in case 1. We have also a corresponding statement for any scalar product of terms of the kind involved in the above sum.

Since we proved $S_{s,(\beta)}$ is a free Brownian motion, we can use stochastic calculus to compute~:

\begin{align*}||\phi_t(X_{0})+ &\int_{0}^{t}(\delta\phi_{t-s}(X_s))_{s}\#d\hat{S}_{s,(\beta)}
||_{2}^{2}=||\phi_t(X_{0})||_2^2+ \int_{0}^{t}||\delta\phi_{t-s}(X_0))_{s}||_2^2(1+\beta)ds
\\ &=||\phi_t(X_{0})||_2^2+\frac{1}{2} \int_{0}^{t}\langle A\phi_{t-s}(X_0),\phi_{t-s}(X_0)\rangle+ \langle \phi_{t-s}(X_0),A\phi_{t-s}(X_0)\rangle(1+\beta)ds
\\ &=||\phi_t(X_{0})||_2^2+ (||X_0||_2^2-||\phi_t(X_{0})||_2^2)(1+\beta)
\end{align*}
Likewise,

$$||X_{0}-\frac{1}{2} \int_{0}^{t} AX'_{s}ds||_{2}^{2}= ||X_{0}||_2^2+\frac{1}{2}\int_{0}^{t}ds\int_{0}^{s}du \ \Re\langle \phi_{s-u}(AX'),AX'\rangle -2\Re \langle X_0'-\phi_t(X_0'),X_0\rangle$$
$$||\int_{0}^{t}(E_{\He}(\delta X))_{s}\#d\hat{S}_{s,(\beta)}||_{2}^{2}=(1+\beta)t||E_{\He}(\delta(X))||^2$$
Since by stationarity $||X_{t}||_{2}^{2}=||X||_{2}^{2}$, it remains to estimate scalar products~:

$$\langle X_{t},X_{0}\rangle=\langle \phi_t(X_{0}),X_{0}\rangle,$$
$$\langle X_{t},\frac{1}{2} \int_{0}^{t} AX_{s}'ds\rangle=\frac{1}{2}\int_{0}^{t}ds\langle \phi_{t-s}(X_{0}),A(X_{0}')\rangle 
$$
\begin{align*}&\langle \frac{1}{2} \int_{0}^{t} AX_{s}'ds,\int_{0}^{t}(E_{\He}(\delta X))_{s}\#d\hat{S}_{s(1+\beta)}\rangle=\frac{1}{2}\int_{0}^{t}ds\int_{0}^{s}du\ \langle \delta \phi_{s-u}(AX_{0}'),E_{\He}(\delta(X_{0}))\rangle\\ &=\frac{1}{2}\int_{0}^{t}ds\int_{0}^{s}du\ \langle \delta \phi_{s-u}(AX_{0}'),E_{\He}(\delta(X_{0}))-\delta(X_{0})\rangle+\frac{1}{2}\int_{0}^{t}ds\int_{0}^{s}du\ \langle \phi_{s-u}(AX_{0}'),\Delta^1X_{0}\rangle
\\ &=\frac{1}{2}\int_{0}^{t}ds\int_{0}^{s}du\ \langle \delta \phi_{s-u}(AX_{0}'),E_{\He}(\delta(X_{0})-\delta(X_{0}))\rangle
+\int_{0}^{t}ds\langle (1-\phi_{s})(AX_{0}'),X_{0}\rangle
\end{align*}
\begin{align*}\langle X_{t},\int_{0}^{t}(\delta(\phi_{t-s}( X))_{s}\#d\hat{S}_{s(1+\beta)}\rangle&=\int_{0}^{t}ds\ \langle \delta \phi_{t-s}(X_{0}),\delta \phi_{t-s}(X_{0})\rangle\\ &=||X_0||_2^2-||\phi_t(X_0||_2^2,\end{align*}
\begin{align*}\langle &X_{t},\int_{0}^{t}(E_{\He}(\delta X))_{s}\#d\hat{S}_{s(1+\beta)}\rangle=\int_{0}^{t}ds\ \langle \delta \phi_{t-s}(X_{0}),E_{\He}(\delta(X_{0}))\rangle\\ &=\int_{0}^{t}ds\ \langle \delta \phi_{t-s}(X_{0}),E_{\He}(\delta(X_{0}))-\delta(X_{0})\rangle
+\int_{0}^{t}ds\ \langle  \phi_{t-s}(X),\Delta^1(X_{0})\rangle\\ &=\int_{0}^{t}ds\ \langle \delta \phi_{t-s}(X_{0}),E_{\He}(\delta(X_{0}))-\delta(X_{0})\rangle+2\langle X_{0}-\phi_{t}(X_{0}),X_{0}\rangle
,\end{align*}

Let us justify the first line of the three last formulas, in proving the following general result (for any $X\in L^{2}(M)$ (e.g. $AX$ above with our special $X$ of interest), any piecewise continuous function with value $\H$:
\begin{align*}\langle X_{t},\int_{0}^{t}( F(s))_{s}\#d\hat{S}_{s(1+\beta)}\rangle&=\int_{0}^{t}ds\ \langle \delta \phi_{t-s}(X_{0}),F(s)\rangle.
\end{align*}

 From the way we proved existence of the stochastic integrals, it suffices to replace  $F(s)$ by $(Z\Xi_i T)$ (or a constant of that form times a characteristic function
 ).  By definition, the left hand side is a limit of Riemann sums with terms of the form
\begin{align*}\tau((\phi_{t-s-h}X)_{s+h}^{*}(Z\Xi_i T)_{s}&\#(\hat{S}_{s+h,(\beta)}-\hat{S}_{s,(\beta)}))\end{align*}

Since by translation invariance we have to consider 
terms of the form $(Z\Xi_i T)\#(\hat{S}_{h,(\beta)}-\hat{S}_{0,(\beta)}))$, $Z,T\in M_0^{(\epsilon)}$, we compute for $Y\in M_0^{(\epsilon)}$~:
\begin{align*}&\tau((Y)_{h}(Z\Xi_i T)\#(\hat{S}_{h,(\beta)}-\hat{S}_{0,(\beta)}))\\ &=\tau(Y_{h}^{*}Z_{0}{S^i}_{h,(\beta)}T_{0})+\frac{1}{2}\int_0^{h}du
\tau(\phi_{h-u}(Y^*)_u(Z_0((S^i_{0,(\beta)}(1+\beta)+2(\delta^*\Xi_i))_u)T_0)ds\\& =\tau(\phi_{h}(Y_{0}^*)Z_{0}\phi_h({S^i}_{0,(\beta)})T_0)+\int_{0}^{h}du\ \langle\delta(\phi_{h-u}(Y_{0})), \phi_{u}^*(Z_{0})(\delta+\partial)(\phi_{h-u}({S^i}_{0,(\beta)})\phi_{u}^*(T_0)\rangle\\ &+\frac{1}{2}\int_0^{h}du\tau((\phi_{h-u}(Y^*))_u(Z_0(A{S^i}_{0,(\beta)}))_u)T_0)ds\\ &
=\tau(T_0\phi_{h}(Y_{0}^*)Z_{0}{S^i}_{0,(\beta)})-\frac{1}{2}\int_{0}^{h}dh'\tau((T_0\phi_{h}(Y_{0}^*)Z_{0})\phi_{h'}(A{S^i}_{0,(\beta)}))\\ &+\frac{1}{2}\int_0^{h}du
\tau((\phi_{h-u}(Y^*))_u(Z_0(A{S^i}_{0,(\beta)}))_u)T_0)ds\\ &+\int_{0}^{h}du\ \langle\delta(\phi_{h-u}(Y_{0})), \phi_{u}^*(Z_{0})(\delta+\partial)(\phi_{h-u}({S^i}_{0,(\beta)})\phi_{u}^*(T_0)\rangle
\\ &=\frac{1}{2}\int_0^{h}du\int_{0}^{u}dv
\langle\delta(\phi_{h-v}(Y^*)),\phi_{v}^*(Z_{0})(\delta+\partial)(\phi_{u-v}(A{S}_{0,(\beta)})))\phi_{v}^*(T_0)\rangle ds\\ &+\int_{0}^{h}du\ \langle\delta(\phi_{h-u}(Y_{0})), \phi_{u}^*(Z_{0})(\delta+\partial)(\phi_{h-u}({S}_{0,(\beta)})\phi_{u}^*(T_0)\rangle
\\ &=\int_{0}^{h}du\ \langle\delta(\phi_{h-u}(Y_{0})), \phi_{u}^*(Z_{0})\partial({S}_{0,(\beta)}^i)\phi_{u}^*(T_0)\rangle
\end{align*}
In the first line we used our definition of $S^i_{s,(\beta)}$. In the second line we used the definition of the state with $\phi_t(Y)\in M^{(\epsilon)}$ so that this is in the kernel of $\partial$, and thus it remains only the written term form the carr\'{e} du Champ term, we also prefer using $A{S}_{0,(\beta)}$.
In the third line we use the differential equation for $\phi_h$ in the first term.
 In the fourth line, we noticed the first term in the third line was zero again by stability of $M^{(\epsilon)}$ via $\phi_h$ (note that if we managed this there is no stability by $\phi_h^*$). Then applying the formula for the state in the third term of the third line, we get a cancellation with the second term. The previous remark applies to the term with carr\'{e} du Champ. Finally, we use a Fubini Theorem and again the differential equation for $\phi$ to remove the integral in $u$, it is easy to put a $(\delta+\partial)^*$ to justify this using formulas above to see the right term is in the domain (at least of a corresponding $L^1$ valued adjoint, this may thus need proving an integral formula first with $S$ replaced by $G_{\alpha}(S)$, details are left to the reader). We also simultaneously simplify 
 two canceling terms.
 
 Now we can sum up all the terms in our Riemann sum (of constant mesh $h$) to get~:
 \begin{align*}&\sum_i\tau((\phi_{t-ih-h}X)_{ih+h}^{*}(Z\Xi_i T)_{ih}\#(\hat{S}_{ih+h,(\beta)}-\hat{S}_{ih,(\beta)}))=\\& \int_{0}^{t}ds\ \langle\delta(\phi_{t-s}(Y_{0})), (\phi_{u(s)}^*-id)(Z_{0})\partial({S}_{0,(\beta)}^i)\phi_{u(s)}^*(T_0)+(Z_{0})\partial({S}_{0,(\beta)}^i)(\phi_{u(s)}^*-id)(T_0)\rangle\\ &
 +\int_{0}^{t}ds\ \langle\delta(\phi_{t-s}(Y_{0})), Z_{0}\Xi_i T_0\rangle
 \end{align*}
 where $u(s)$ is increasing linearly from zero on any new interval of integration of the step function we took. Taking a fixed mesh $h$ for the partition, the first line easily goes to zero from continuity of $\phi^*$
 . This concludes our claim.

To sum up, we got in taking $X'=\eta_{\alpha}(X)$:
\begin{align*}&||-X_{t}+X_{0}-\frac{1}{2} \int_{0}^{t} (\Delta(\eta_\alpha(X)))_{s}ds + \int_{0}^{t}(\delta(X))_{s}\#d\hat{S}_{s,(\beta)}||_2^2\\ &=(1+\beta) t||E_{\He}(\delta(X))||^2-t||\delta(X)||^2-\Re\int_{0}^{t}ds\int_{0}^{s}du\ \langle \delta \phi_{s-u}(\Delta\eta_\alpha(X)),E_{\He}(\delta(X_{0}))-\delta(X)\rangle\\ &-2\Re\int_{0}^{t}ds\ \langle \delta \phi_{t-s}(X),E_{\He}(\delta(X))-\delta(X)\rangle+\Re\int_{0}^{t}ds\langle (1-\phi_{s})(\Delta^1 X),\eta_\alpha(X)-X\rangle\\ &-t\Re\langle \delta(\eta_\alpha(X)-X),\delta(X) \rangle+2\Re\langle (1-\phi_{s})(X),\eta_\alpha(X)-X\rangle
\\ &||-X_{t}+\phi_{t-s}(X_{0}) + \int_{0}^{t}(\delta(\phi_{t-s}(X)))_{s}\#d\hat{S}_{s,(\beta)}||_2^2=\beta(||X_0||_2^2-||\phi_t(X_0)||_2^2)\rightarrow_{\beta\rightarrow 0}0.\end{align*}
In case 1, this gives~:
\begin{align*}&||-X_{t}+X_{0}-\frac{1}{2} \int_{0}^{t} (\Delta(\eta_\alpha(X)))_{s}ds + \int_{0}^{t}(\delta(X))_{s}\#d\hat{S}_{s}||_2^2\\ &\leq 4||\eta_\alpha(X)-X||_2||X||_2+3t||\eta_\alpha(\Delta^{1/2}(X))-\Delta^{1/2}(X)||_2||\Delta^{1/2}(X)||_2.\end{align*}
This especially shows $\int_{0}^{t} (\Delta(\eta_\alpha(X)))_{s}ds$ is bounded in $\alpha$ in $L^2$ thus up to extraction converges weakly (necessarily to its $L^1$ limit $\int_{0}^{t} (\Delta^1(X))_{s}ds$ which is thus in $L^2$ and we get by a standard result~:\begin{align*}&||-X_{t}+X_{0}-\frac{1}{2} \int_{0}^{t} (\Delta^1(X))_{s}ds + \int_{0}^{t}(\delta(X))_{s}\#d\hat{S}_{s}||_2^2\\ &\leq\liminf_{\alpha\to\infty} 4||\eta_\alpha(X)-X||_2||X||_2+3t||\eta_\alpha(\Delta^{1/2}(X))-\Delta^{1/2}(X)||_2||\Delta^{1/2}(X)||_2=0.\end{align*}
(recall $D(\Delta^1)\subset D(\Delta^{1/2})$ by definition.)
 This concludes case 1.

For the end of the proof in case 2, let us note several results we proved at level $\epsilon$ (after taking $\beta\to 0$ as above), for any $X\in D(\Delta^1)$, $Y\in M$~:
$$\tau(\left(X_t-X_0+\frac{1}{2}\int_0^tds(\Delta\eta_\alpha(X))_s\right)Y_0)=\langle\phi_t(X)-X,Y-\eta_{\alpha}(Y)\rangle$$
Since we checked a norm convergence in $L^1$ of the integral above, we get:
$$\tau(\left(X_t-X_0+\frac{1}{2}\int_0^tds(\Delta^1(X))_s\right)Y_0)=0.$$

Likewise, for $X,Y\in D(\delta)$: \begin{align*}\tau(X_t7\int_0^T&(E_{\He}(\delta(P)))_s\# d\hat{S}_t)\\ &=\int_0^{t\wedge T}ds \langle \delta\phi_{t-s}(X^*),E_{\He}(\delta(P))-\delta(P)\rangle+2(\tau((X)_{t}P_{t\wedge T})-\tau((X)_{t}P))
\end{align*}

Finally, it is not hard to show by the same limits for any $X,Y\in D(\Delta^1)$ $\tau(\Delta^1(X)_sY_0)=\tau(X_s\Delta^1(Y)_0)$.

\begin{step}Conclusion in case 2\end{step}
We defined $\hMe$ the corresponding GNS construction of $A$ with respect to  $\hat{\tau}_{0}^{(\epsilon)}$. It contains $\tilde{M^{(\epsilon)}}$ as a von Neumann subalgebra, thus in case (1) it contains $\tilde{M}$.
Instead of taking a weak limit point for $\epsilon$ going to zero, we need a universal $C^*$-algebra $C$ where $M^{(\epsilon)}$ is replaced by a universal free product over $B$ in variables $X_1,...,X_n$ (free product defined as above via an inductive limit of finitely many times free products). We again call $\hat{\tau}_{0}^{(\epsilon)}$ the corresponding state on $C$.
 Moreover we consider the subset of tracial states $\M^c(C)$ such that for any non-commutative polynomial over $B$, say $P$, the map $t_1,...,t_n,t_1',...,t_q'\mapsto \tau(P(X_{t_1}^{i_1},..., X_{t_p}^{i_p},S_{t'_1}^{j_1},...,S_{t'_q}^{j_q}))$ is continuous on products of each $[0,T]$ (recall $X_t^i, S_t^j$ are the canonical generators in $C$), with 
  the topology of uniform convergence on compacts for those functions (this topology is obviously metrizable and finer than the weak-* topology induced by $C$ but we won't use this second fact). Of course  $\hat{\tau}_{0}^{(\epsilon)}\in\M^c(C)$ since $\tau(|X_{t}^i-X_{s}^i|^2)\leq 2|\Re\tau(X_i^2-X_i\phi_{t-s}(X_i))| \leq 2||X_i||_2(t-s)^{1/2}||\delta(X_i)||_2$ (recall we defined it in such a way $||\delta(X_i)||_2$ is independent of $\epsilon$. Likewise,  $\tau(|S_{t}^i-S_{s}^i|^2)\leq 2(t-s)||\tilde{\eta}(1)||$ since we proved $S_{t}^i$ is a B-free Brownian motion of covariance $\tilde{\eta}$. The continuity for all non-commutative polynomials follows using Cauchy-Schwarz (even an equicontinuity).  Actually $\hat{\tau}_{0}^{(\epsilon)}\in K_{2||X_i||_2(t-s)^{1/2}||\delta(X_i)||_2,2(t-s)||\tilde{\eta}(1)||}$ set satisfying the previous inequalities, which is a compact subset of $\M^c(C)$ by a variant of Arzela-Ascoli theorem (we use $B$ has separable predual to replace continuity on B-non-commutative polynomials by countably many B-non-commutative monomials).
We can now consider a weak limit point $\hat{\tau}_{0}$ (up to extraction we will consider a weak limit from now on). As before $\hat{S}_s$ gives a B-free Brownian motion of covariance $\eta$ at the limit. Of course the new state is again translation invariant for variables $X_1,...,X_n$ and at each time we know the state as being our original state on $M$ since the state we took on $M^{(\epsilon)}$ converges to this one.
We define finally $\hM$ the corresponding GNS construction of $C$ with respect to  $\hat{\tau}_{0}$. 

Again, we have to show several integrals make sense. For a B-non-commutative polynomial $P\in \hM_0=M$ in $X_{0}^1,...,X_{0}^n$, as before $\int_0^t(P)_sds$ makes sense as a Riemann integral of continuous functions with value $L^1$ or  $L^2$ (we use $\hat{\tau}_{0}\in \M^c(C)$ to prove continuity). By a dominated convergence theorem for Bochner integral with value $L^1(\hM)$ or $L^2(\hM)$, for any $X\in D(\Delta^1)\cap M$ or $X\in D(\Delta)\cap M$ respectively, $\int_0^t(\Delta^1(X))_sds$ also make sense and is a limit of previously considered integrals for P converging to $\Delta^1(X)$.

Likewise for stochastic integrals with $P\# \Xi_i$ P a non-commutative polynomial, and thus $\int_0^t\delta(X)_s\#dS_s$ or $\int_0^t\delta(\phi_{t-s}(X))_s\#dS_s$ make sense.

Let us remind the reader we considered in step 1, an extension $\delta_i$ on $\overline{M}$ of our original derivation with the same value on non-commutative polynomials than the one we ended up considering on $M^{(\epsilon)}$ (we called $\delta_i^{(\epsilon)}$ before removing extra indices) and of course those non-commutative polynomials are corresponding cores for those variants, thus when considered from $L^2(\overline{M})$ or $L^2(M^{(\epsilon)})$ to $\oH$ the adjoints (in both case known to be densely defined) satisfy, if $U\in D(\delta_i^*)$, then $U\in D(\delta_i^{(\epsilon)*})$ and $\delta_i^{(\epsilon)*}(U)=E_{M^{(\epsilon)}}\delta_i^*(U)$.

At the $\epsilon$ level (after taking a limit $\beta \to 0$), if we consider $U\in \oH$, $P$ a non-commutative polynomial in 
$X_i+\sqrt{\epsilon}\overline{S_i}$:

\noindent $\sum_i|\tau(P_tE_{M^{(\epsilon)}}(\delta_i^*(U_i)))|=|\sum_i\tau(\phi_t(P)\delta_i^*(U_i))|=|\sum_i\langle\delta_i^{(\epsilon)}\phi_t(P^*),U_i\rangle|\leq \sqrt{t}||P||_2||U||$.

Now, taking approximations by polynomials of $\delta_i^*(U)$, one can check we have necessarily the corresponding result at the limit~:
$\sum_i|\tau(P_t\delta_i^*(U_i)))|\leq \sqrt{t}||P||_2||U||$ and this extends beyond non-commutative polynomials by continuity so that for any $Y\in M$ $E_{\hM_0}\alpha_t(Y)\in D(\delta)$ (we call $\alpha_t$ the corresponding dilation endomorphism as before).

Starting again at level $\epsilon$, as we noticed earlier, for any non-commutative polynomials $P,Q$ we have $\tau((\Delta^1(P))_tQ_0)=\tau(P_t(\Delta^1(Q)_0))$. Again, even though in $L^1$ $\Delta^1(P),\Delta^1(Q)$ are easily computed when monomials, for instance $P=X_{i_1}...X_{i_p}$,  $\Delta^1(P)=\sum_jX_{i_1}...\Delta^1(X_{i_j}) ...X_{i_p}+\sum_{j<l}X_{i_1}..\Gamma(X_{i_j},...,X_{i_l}) ...X_{i_p}$, if one replaces $\Delta^1(X_{i_j})$ and $\delta(X_j)$ by polynomials $S_i$, and $ S_j'$ $(||S_j'||\leq ||\delta(X_j)||$, this gives a polynomial $R(X_i,S_i,S_j')$. Now at epsilon level, doing the same computation letting $X_i^{\epsilon}=X_i+\sqrt{\epsilon}\ \overline{S_i}$, one gets $\Delta^1(P(X_i^{\epsilon}))=\sum_jX_{i_1}^{\epsilon}...E_{M^{(\epsilon)}}(\Delta^1(X_{i_j})) ...X_{i_p}^{\epsilon}+\sum_{j<l}X_{i_1}^{\epsilon}...E_{M^{(\epsilon)}}(\Gamma(X_{i_j},X_{i_{j+1}}^{\epsilon}...,X_{i_l})) ...X_{i_p}^{\epsilon}.$

Reminding $||X_i||\leq 1$, $R=1+2 ||\tilde{\eta}(1)||$, for $\epsilon<1$, one gets \begin{align*}||\Delta^1&(P(X_i^{\epsilon}))-R(X_i^{\epsilon},S_i(X^{\epsilon}),S_j'(X^{\epsilon}))||_1
\leq \\ &pR^{p-1} \max_{i\in \{i_p\}}\left(||\Delta^1(X_i^{\epsilon})-S_i(X)||_1+||E_{M^{(\epsilon)}}(S_i(X))-S_i(X^{\epsilon})||_1\right)\\ &+p(p-1)R^{p-2}\left(2\max_{i\in \{i_p\}}(||\delta(X_i)||)+\max_{i\in \{i_p\}}(||\delta(X_i)-S_i(X)||+||S_i(X)-S_i'(X^{\epsilon})||)\right)\\& \times\left(\max_{i\in \{i_p\}}(||\delta(X_i)-S_i'(X)||+||S_i'(X)-S_i'(X^{\epsilon})||)\right)\end{align*}

and one can bound $||E_{M^{(\epsilon)}}(S_i(X))-S_i(X^{\epsilon})||_2\leq 2||S_i(X)-S_i(X^{\epsilon})||_2$.

Thus approximating by fixed polynomials and letting $\epsilon\to 0$ and then varying approximating polynomials, one gets $\tau((\Delta^1(P))_tQ_0)=\tau(P_t(\Delta^1(Q)_0))$ at the limit.


Also, from the end of step 4, we have at epsilon level:
$$\tau((P)_tQ_0)=\tau((P)_0Q_0)-\frac{1}{2}\int_0^tds\tau((AP)_sQ_0),$$

giving the same result at the limit, and with the previous results for any $Y\in M$
\begin{align*}\tau(Y_tQ_0)&=\tau(Y_0Q_0)-\frac{1}{2}\int_0^tds\tau(Y_sAQ_0)\\ &=\tau(Y_0Q_0)-\frac{1}{2}\int_0^tds\langle\delta E_{\hM_0}\alpha_s(Y^*),\delta(Q)\rangle),\end{align*}

and this extends from Q polynomial to any element of $D(\delta)$ by a core property, thus especially for $Q\in D(A)$. Now Taking as usual a Laplace transform one gets (using we have the same equation for the semigroup $\phi_t$ on M)

$\int_0^\infty dt  e^{-\lambda t}\tau((E_{\hM_0}\alpha_s(Y)-\phi_s(Y))(\lambda+A/2)Q)=0,$

since this applies for any $Q\in D(A)$ especially one can replace Q by $(\lambda+A/2)^{-1}(Q)$ and get $E_{\hM_0}\alpha_s(Y)=\phi_s(Y)$ after taking an inverse Laplace transform.

To prove as we want our second SDE, it remains to compute $\alpha_s(Y)-E_{\hM_0}\alpha_s(Y)$. 

From the result at the end of step 4, it is easy to deduce as before (for $Y\in M, Z\in D(\delta)$ after extension by continuity):
\begin{align*}\tau(\alpha_t(Y)\int_0^T(\delta(Z))_s\# d\hat{S}_s) &=2(\tau(Y_{t}A_{t\wedge T})-\tau(Y_{t}Z))\\ &=2(\tau(\phi_{t-{t\wedge T}}(Y)Z)-\tau(\phi_t(Y)Z))
\end{align*}

With this equation and $\tau(Y_tZ_s)=\tau(\phi_{t-s}(Y)Z)$ we have all the tools to compute as in step 4.



Let us compute \begin{align*}\tau(\int_0^t(\delta(\phi_{t-s}(Y)))_s\# d\hat{S}_t&\int_0^T(\delta(Z))_s\# d\hat{S}_t)=\int_0^{t\wedge T}ds \langle \delta(\phi_{t-s}(Y^*)),\delta(Z)\rangle\\ &=\int_0^{t\wedge T}ds\  \tau(\phi_{t-s}(Y)AZ)=2(\tau(\phi_{t-{t\wedge T}}(Y)Z)-\tau(\phi_t(Y)Z))\\ &=\tau(\alpha_t(Y)\int_0^T(\delta(Z))_s\# d\hat{S}_s)
\end{align*}
(we already checked the stochastic integral was well defined).
Extending this to step functions (of $\delta(\phi_{t-s}(Z))$) and taking limits in the definition of stochastic integrals one deduces~:
$$\tau(\alpha_t(Y)\int_0^T(\delta(\phi_{t-s}(Z)))_s\# d\hat{S}_s)=\tau(\int_0^t(\delta(\phi_{t-s}(Y)))_s\# d\hat{S}_t\int_0^T(\delta(\phi_{t-s}(Z)))_s\# d\hat{S}_s),$$

from which one deduces by immediate computation (taking $Y=Z\in M$):

$$||\alpha_t(Y)-\phi_{t}(Y)-\int_0^t(\delta(\phi_{t-s}(Y)))_s\# d\hat{S}_t||_2^2=0.$$
 We thus got the SDE we wanted to prove. Having obtained a scalar product for $\alpha_t(Y)$ with stochastic integrals for general $Y$ we can also compute as in step 4 to get the SDE for $X\in D(\Delta^1)$ (all terms with 
$\epsilon$ or $\beta$ being taken at value $0$).
 
 \begin{align*}&||-X_{t}+X_{0}-\frac{1}{2} \int_{0}^{t} (\Delta(\eta_\alpha(X)))_{s}ds + \int_{0}^{t}(\delta(X))_{s}\#d\hat{S}_{s,(\beta)}||_2^2\\&=\Re\int_{0}^{t}ds\langle (1-\phi_{s})(\Delta^1 X),\eta_\alpha(X)-X\rangle\\ &-t\Re\langle \delta(\eta_\alpha(X)-X),\delta(X) \rangle+2\Re\langle (1-\phi_{s})(X),\eta_\alpha(X)-X\rangle
\end{align*}
The end is a copy of the end in case 1.

\end{proof}
By the definition of stochastic integral, we have an isometry of 
$L^{2}(\hat{M}_{0})\oplus L^{2}_{ad}([0,\infty),H(\hat{M},\eta\circ E_B))$ into $L^{2}(\hat{M})$, but since every generator of $\hat{M}$ described above can be written by stochastic integrals, using Ito formula, we get that the above space is actually dense in the whole $L^{2}(\hat{M})$, as a consequence we deduce~:
\begin{proposition}
(All identifications as  $\hat{M}_{0}-\hat{M}_{0}$ bimodules, recall $\hat{M}_{0}=M$) For $\hat{M}$ given by the previous theorem  $L^{2}(\hat{M})=L^{2}(\hat{M}_{0})\oplus L^{2}_{ad}([0,\infty),H(\hat{M},\eta\circ E_B)\hat{M}\otimes_{B}\hat{M}))$ thus $L^{2}(\hat{M})\ominus L^{2}(\hat{M}_{0})\subset L^{2}([0,\infty),H(\hat{M},\eta\circ E_B))\simeq (H(\hat{M},\eta\circ E_B))\o\ell^{2}(\N)$ (using separability of $L^{2}(\hat{M})$). For instance when $\eta=E_B$,$B$ amenable, the above orthogonal is weakly embeddable in the coarse, when $B=\C$, included in a countable direct sum of coarse correspondences.
\end{proposition}

By definition, recall $\alpha_{t}$ is our notation for our dilation of $\phi_{t}$ i.e $\phi_{t}(x)=E_{M}(\alpha_{t}(x)).$ 
Let us note a general transversality lemma (this could also been deduce from Popa's transversality lemma in \cite{P08} in using the automorphism variant of Theorem \ref{main} with the symmetry we proved   translated into a symmetry with respect to $t/-t$, we prefer giving an elementary direct proof in our context, since dilating a semigroup makes the deduction ``easier"):

\begin{proposition}\label{trans}
For any dilation $\alpha_{t}$  (by homomorphims) of a (necessarily contractive since completely positive) symmetric semigroup $\phi_{t}$, i.e. $E_{M}(\alpha_{t}(x))=\phi_{t}(x)$, we have a transversality relation 
$||\alpha_{2t}(x)-x||_{2}^{2}=2||\alpha_{t}(x)-E_{M}(\alpha_{t}(x))||_{2}^{2}$. Moreover we have an equiconvergence relation $2||x-\phi_{2t}(x)||_{2}^{2}\leq ||\alpha_{2t}(x)-x||_{2}^{2}\leq 4||x||_{2}||x-\phi_{t}(x)||_{2}.$
\end{proposition}

\begin{proof}
First note that $||\alpha_{t}(x)-E_{M}(\alpha_{t}(x))||_{2}^{2}=||\alpha_{t}(x)||_{2}^{2}-||E_{M}(\alpha_{t}(x))||_{2}^{2}=||x||_{2}^{2}-||\phi_{t}(x)||_{2}^{2},$ by Pythagoras' Theorem.
But also $||x-\phi_{2t}(x)||_{2}^{2}=||x||_{2}^{2}+||\phi_{2t}(x)||_{2}^{2}- 2||\phi_{t}(x)||_{2}^{2}\leq ||x||_{2}^{2}-||\phi_{t}(x)||_{2}^{2}$ by symmetry and contractivity (with semigroup property for the inequality). 

Likewise $||\alpha_{2t}(x)-x||_{2}^{2}=2||x||_{2}^{2}- 2||\phi_{t}(x)||_{2}^{2}\geq 2||x-\phi_{2t}(x)||_{2}^{2},$
and $||x||_{2}^{2}- ||\phi_{t}(x)||_{2}^{2}=\langle x-\phi_{t}(x),x\rangle+\langle \phi_{t}(x),x-\phi_{t}(x)\rangle\leq 2||x||_{2}||x-\phi_{t}(x)||_{2}$
\end{proof}
\section{Deformations for Popa's Deformation/Rigidity techniques}
\subsection{Properties of stochastic deformations}

We emphasize properties useful for deformation rigidity techniques in the next result (see also the general proposition \ref{trans}).

Let us fix notations before. We call $\H$ the space of value of our derivation $\delta$.
Let $(\H_{s},\xi_{\phi_{s}})$ be the pointed correspondence associated to the completely positive map $\phi_{s}$.
One can get in a standard way a measurable field of Hilbert spaces over $\R_{+}$ in that way (assuming M with separable predual). 
Likewise for $\H_{s}\o_{M}\H\o_{M}\H_{s}$ .

 Let $\H_{\phi_{s}}=\xi_{\phi_{s}}\o \H \o \xi_{\phi_{s}}\subset \H_{s}\o_{M}\H\o_{M}\H_{s}$ be the (constant) sub-Hilbert field corresponding to $\H$. On the direct integral Hilbert space $\int^{\oplus}_{\small\R_{+}}\H_{s}\o_{M}\H\o_{M}\H_{s}d\lambda$ (with respect to Lebesgue measure). We have an $M-M$ bimodule structure (acting by diagonal operators). We can define $\H_{\infty}$ the sub-bimodule generated by $L^{2}(\R_{+},\H)\simeq \int^{\oplus}_{\small\R_{+}}\H_{\phi_{s}}d\lambda$.

\begin{theorem}\label{corresp}
Let $\phi_{t}$ a symmetric Markov semigroup associated to a symmetric Dirichlet form $\Delta=\delta^{*}\delta$ with derivation $\delta:L^{2}(M)\rightarrow \H$. Then the tracial state $\tau$ of theorem \ref{main}, giving by GNS-construction a von Neumann algebra $M\subset\widetilde{M}$, gives rise to a symmetric dilation $\alpha_{t}:M\rightarrow \widetilde{M}$ of $\phi_{t}$ (induced from $\alpha_{t}(X_{0})=X_{t}$ on Path space). Moreover $ (\alpha_{t}-\phi_{t})(M)\subset \hat{\H}_{\infty}\subset L^{2}(\widetilde{M})\ominus L^{2}(M)$, where $\hat{\H}_{\infty}$ is isomorphic to sub-bimodule of the Hilbert bimodule previously introduced $\H_\infty$. 
More precisely, $\alpha_{t}-\phi_{t}(x)$ is sent to $1_{[0,t]}\delta(\phi_{t-s}(x))$ in the canonical  $L^{2}(\R_{+},\H)$ generating $\H_{\infty}$, and all those images generate $\hat{\H}_{\infty}$.
\end{theorem}

\begin{proof}
The only new result is about the range bimodule. From the definition of the state we see that (for say $t> s$): \begin{align*}\langle X(\alpha_{t})(Y)Z,U(\alpha_{s})(V)W\rangle& = \langle X\phi_{t}(Y)Z,U\phi_{s}(V)W\rangle\\ &+ \int_{0}^{s}du \ \tau( \delta(\phi_{t-u}(Y^{*}))\phi_{u}(X^{*}U)\delta(\phi_{s-u}(V))\phi_{u}(WZ^{*})),\end{align*}
 thus 
$$\langle X(\alpha_{t}-\phi_{t})(Y)Z,U(\alpha_{s}-\phi_{s})(V)W\rangle= \int_{0}^{s}du\ \langle \phi_{u}(XU^{*})\delta(\phi_{t-u}(Y)),\delta(\phi_{s-u}(V))\phi_{u}(WZ^{*})\rangle.$$

We see that we can identify $X(\alpha_{t}-\phi_{t})(Y)Z$ to $u\mapsto 1_{u\in[0,t]}X\xi_{\phi_{u}}\o\delta\phi_{t-u}(Y)\o\xi_{\phi_{u}}Z$ in $\H_{\infty}$, which proves the result.
\end{proof}

\subsection{Deformation/rigidity reminder and applications}

In \cite{{popamal1}, {popa2001}}, Popa introduced a powerful tool to prove the unitary conjugacy of two von Neumann subalgebras of a tracial von Neumann algebra $(M, \tau)$. 
If $A, B \subset (M, \tau)$ are (possibly non-unital) von Neumann subalgebras, denote by $1_A$ (resp. $1_B$) the unit of $A$ (resp. $B$).

\begin{theorem}[Popa, \cite{{popamal1}, {popa2001}}]\label{inter1}
Let $(M, \tau)$ be a finite von Neumann algebra. Let $A, B \subset M$ be possibly non-unital von Neumann subalgebras. The following are equivalent:
\begin{enumerate}
\item There exist $n \geq 1$, a possibly non-unital $\ast$-homomorphism $\psi~: A \to \mathbf{M}_n(\C) \otimes B$ and a non-zero partial isometry $v \in \mathbf{M}_{1, n}(\C) \otimes 1_AM1_B$ such that $x v = v \psi(x)$, for any $x \in A$.

\item The bimodule $\vphantom{}_AL^2(1_AM1_B)_B$ contains a non-zero sub-bimodule $\vphantom{}_AH_B$ which is finitely generated as a right $B$-module. 

\item There is no sequence of unitaries $(u_k)$ in $A$ such that 
\begin{equation*}
\lim_{k \to \infty} \|E_B(a^* u_k b)\|_2 = 0, \forall a, b \in 1_A M 1_B.
\end{equation*}
\end{enumerate}
\end{theorem}
If one of the previous equivalent conditions is satisfied, we shall say that $A$ {\it embeds into} $B$ {\it inside} $M$ and denote $A \preceq_M B$. When M is a $II_{1}$ factor and $A,B \subset M$ are Cartan subalgebras, then $A \preceq_M B$ if and only if there exists a unitary $u \in U(M)$ such that $A=uBu^{*}$, see \cite{popa2001} Theorem A.1 (see also \cite{Va06}, Theorem C.3).

We will use this notion in conjunction with a variant of a lemma of Jesse Peterson (theorem 2.5 in \cite{P10}).
Our proof gives a special case of this result when we can prove existence of dilations of the above type.
Recall the following~:

\begin{definition}
Let $N$ be a finite von Neumann $B,A \subset N$  von Neumann subalgebras and $\H$ an $N$-$N$ Hilbert bimodule.  $\H$ is said to be 
{\bf compact relative to $B\subset N$ as an $A-A$ bimodule} if given any sequence $x_n \in (A)_1$ such that $\| E_B(y x_n z) \|_2 \rightarrow 0$, for all $y,z \in N$ then $\langle x_n \xi 
y_n, \xi \rangle \rightarrow 0$, for any sequence $y_n \in (A)_1$ and $\xi \in \H$. 
\end{definition}

The standard example, as explained in \cite{P10} example 2.3, is a multiple of $L^{2}(N)\o_{B}L^{2}(N)$ (here we can take $A=N$).
\begin{theorem}\label{L2rigNorm}
Let $N$ be a finite von Neumann algebra, $B,A \subset N\subset M$ von Neumann subalgebras, assume $\H\subset L^{2}(M)\ominus L^{2}(N)$ is an Hilbert $N$-$N$ bimodule which is compact relative to $B\subset N$ as an $A-A$ bimodule and 
 assume given a family of $*$-homomorphism $\alpha_{t}:N\rightarrow M$ dilating a symmetric semigroup $\phi_{t}$ on $N$ and with $Range (\alpha_{t}-\phi_{t})\subset \H$.  If $A$ does not embed into $B$ inside $N$, and the associated deformation 
$\alpha_{t}$ converges uniformly in $\| \cdot \|_2$ to the identity on $(A)_1$ then $\alpha_{t}$ converges uniformly in $\| \cdot \|_2$ to 
the identity on the unit ball of the von Neumann algebra generated by its normalizer $(\mathcal{N}_N(A)'')_1$.
\end{theorem}

\begin{proof}
It is well known that uniform convergence on the unit ball is equivalent to uniform convergence on unitaries in our context.
  Let $1 \geq \varepsilon > 0$,  get by the assumption and lemma \ref{trans} 
a $t_0 > 0$ such that $\forall t< t_0$, $x \in B_1$  
$y \in N_1$ we have 
$\| \alpha_t(x) - x \|_2 < \varepsilon / 8$, and $\| \phi_t(x) - x \|_2^{1/2} < \varepsilon / 8$.
Since $A\npreceq_N B $, there exists a $u_{n}$ sequence of unitaries in $A$ such that $||E_{B}(xu_{n}y)||_2\rightarrow 0$ for any $x,y\in N$ and thus since  we assumed $\H$ compact relative to $B\subset N$ as an $A-A$ bimodule, for any $\xi\in\H$ like $\xi=\phi_{t}(v)-\alpha_{t}(v)$, if $v \in \mathcal{N}_N(A)$ so that $vu_{n}v^{*}\in A$, $\langle u_n \xi 
vu_nv^{*}, \xi \rangle \rightarrow 0$.
Hence we have, using the relations $$u_n(\alpha_t( v )) v^*u_n^*v -\alpha_t( v ) =(u_n-\alpha_t(u_n))\alpha_t( v )v^*u_n^*v+ \alpha_t(u_n v )(v^*u_n^*v-\alpha_t(v^*u_n^*v))$$
$$u_n(\phi_t( v )) v^*u_n^*v -\phi_t( v ) =(u_n\phi_t( v )-\phi_t(u_n v ))v^*u_n^*v+ (\phi_t(u_n v )v^*u_n^*v-\phi_t(u_n vv^*u_n^*v))$$

and using also a standard bound on completely positive maps (e.g. Corollary 1.1.2 in \cite{popa2001} $||\phi_{t}(uv)-\phi_{t}(u)v||_{2}\leq 3 ||\phi_{t}(v)-v||_{2}^{1/2}$), we get~: \begin{align*}
&\sqrt{2}\| \phi_t(v) - \alpha_t(v) \|_2 
= \lim_{n \rightarrow \infty} \| u_n(\phi_t(v) - \alpha_t( v )) v^*u_n^*v
			-  (\phi_t(v) - \alpha_t(v)) \|_2
\\ &\leq \sup_{n} ( 3\| \phi_t(u_n) -  u_n \|_2^{1/2}+3\| \phi_t(v^*u_n^*v) -  v^*u_n^*v  \|_2^{1/2} +\| \alpha_t(u_n) -  u_n \|_2+\| \alpha_t(v^*u_n^*v) -  v^*u_n^*v  \|_2)\\ &
	< \varepsilon,
\end{align*}
Using now the transversality part of lemma \ref{trans} this is nothing but $\| v - \alpha_{2t}(v) \|_2 < \varepsilon$. A standard argument concludes.\end{proof}

Let $\Gamma$ be a discrete group which acts on a finite von Neumann algebra (with separable predual) $B$ by preserving a distinguished trace $\tau$ , let $M = B \rtimes \Gamma$ be the crossed product construction and 
suppose that $\pi: \Gamma \rightarrow \U(\K)$ is a $C_0$ representation.  Then the associated $M$-$M$ bimodule given by $\H_\pi = \K \overline 
\otimes L^2M$ with actions satisfying 
$$
b_1u_{\gamma_1} (\xi \otimes \eta) b_2u_{\gamma_2} = (\pi(\gamma_1) \xi ) \otimes (b_1 u_{\gamma_1} \eta
b_2 u_{\gamma_2} ),
$$
for each $\xi \in \K, \eta \in L^2M, \gamma_1, \gamma_2 \in \Gamma, b_1, b_2 \in B$.
 Note (as in example 2.4 of \cite{P10}) $\H_\pi$ is compact relative to $B\subset M$ as an M-M bimodule.
Let $b:\Gamma \rightarrow \K$ be a cocycle ($b(\gamma_1 \gamma_2) = b(\gamma_1) + \pi(\gamma_1)b(\gamma_2)$ for 
all $\gamma_1, \gamma_2 \in \Gamma$) then the derivation $\delta_b$ from $B \Gamma \subset B \rtimes \Gamma$ into the $(B \rtimes \Gamma)$-$(B \rtimes \Gamma)$ 
bimodule $\H_\pi$ which satisfies $\delta_b(xu_\gamma) = b(\gamma) \otimes xu_\gamma$ for all $\gamma \in \Gamma$, $x \in B$ is a closable real 
derivation.

The next theorem can be viewed as a variant of 
Theorems $4.9$ in \cite{ozawapopa},
 A in \cite{ozawapopaII} and (especially close to) $3.5$ in \cite{houdayer9}.

As usual, as always in this part, $\Lim$ denote a state on $\ell^{\infty}(I)$ for $I$ a directed set, which  extends ordinary limit.

\begin{theorem}\label{step}
Let $M=B \rtimes \Gamma$ as above, $M\subset \widetilde{M}$ von Neumann subalgebras, assume $\K$  is a non-amenable representation, $\H_{\infty}\subset L^{2}(\widetilde{M})\ominus L^{2}(M)$ produced as in Theorem \ref{corresp} from $\H$ coming from $\K$ as above. Take $P \subset M$ is a regular weakly compact subalgebra. 
Assume given a family of $*$-homomorphism $\alpha_{t}:M\rightarrow \widetilde{M}$ dilating a symmetric $\| \cdot \|_2$-strongly continuous semigroup $\phi_{t}$ on $M$ and with $Range (\alpha_{t}-\phi_{t})\subset \H_{\infty}$ in the way of Theorem \ref{corresp}. Also assume $\alpha_{t}$ is symmetric (at least so that $\tau(\alpha_{t}(X)Y\alpha_{t}(Z)T)=\tau(X\alpha_{t}(Y)Z\alpha_{t}(T))$.
Then $\alpha_{t}$ (or $\phi_{t}$) converges uniformly in $\| \cdot \|_2$ to the identity on $(P)_1$.
\end{theorem}

\begin{proof}

Assume for contradiction $\alpha_{t}$ does not converge uniformly in $\| \cdot \|_2$ to the identity on $(P)_1$.

Since by assumption $P$ is weakly compact inside $M$, there exists a net $(\eta_n)$ of vectors in $L^2(P \bar{\otimes} \bar{P})_+$ such that
\begin{enumerate}
\item $\lim_n \|\eta_n - (v \otimes \bar{v})\eta_n\|_2 = 0$, $\forall v \in \mathcal{U}(P)$; 
\item $\lim_n \|\eta_n - \Ad(u \otimes \bar{u})\eta_n\|_2 = 0$, $\forall u \in \mathcal{N}_M(P)$;
\item $\langle (a \otimes 1)\eta_n, \eta_n\rangle = \tau(a) = \langle \eta_n, (1 \otimes \bar{a}) \eta_n \rangle$, $\forall a \in M, \forall n$.
\end{enumerate}
We consider $\eta_n \in L^2(M \bar{\otimes} \bar{M})_+$, and note that $(J \otimes \bar{J}) \eta_n = \eta_n$, where $J$ denotes the canonical anti-unitary on $L^2(M)$. We shall simply denote $\mathcal{N}_M(P)$ by $\mathcal{G}$.

Since any self-adjoint element $x \in (P)_1$ can be written
\begin{equation*}
x = \frac12 \|x\|_\infty (u + u^*)
\end{equation*}
where $u \in \mathcal{U}(P)$, it follows that $(\alpha_t)$ does not converge uniformly on $\mathcal{U}(P)$ either. Combining this with lemma \ref{trans}, we get that there exist $0 < c < 1$, a sequence of positive reals $(t_k)$ and a sequence of unitaries $(u_k)$ in $\mathcal{U}(P)$ such that $\lim_{k} t_k = 0$ and $\| \alpha_{t_k}(u_k ) - (E_M \circ \alpha_{t_k})(u_k ) \|_2 \geq c $, $\forall k \in \N$. Since $\|\alpha_{t_k}(u_k )\|_2 = 1$, by Pythagoras's theorem, we obtain
\begin{equation}\label{key}
\|(E_M \circ \alpha_{t_k})(u_k )\|_2 \leq \sqrt{1 - c^2} , \forall k \in \N.
\end{equation}
Set $\delta = \frac{1 - \sqrt{1 - c^2}}{3} $. 

Define for any $n$ and any $k \geq k_0$,
\begin{eqnarray*}
\eta_n^k & = & (\alpha_{t_k} \otimes 1)(\eta_n) \in L^2(\widetilde{M}) \bar{\otimes} L^2(\bar{M}) \\
\xi_n^k & = & (e_M\alpha_{t_k} \otimes 1)(\eta_n) \in L^2(M) \bar{\otimes} L^2(\bar{M}) \\
\zeta_n^k & = & (e_M^\perp\alpha_{t_k} \otimes 1)(\eta_n) \in (L^2(\widetilde{M}) \ominus L^2(M)) \bar{\otimes} L^2(\bar{M}).
\end{eqnarray*}
We observe that by symmetry for all $x \in M$
\begin{equation}\label{norm22}
\|((\alpha_{t_k}(x)-x) \otimes 1) \eta_n^k\|_2^2 =\|((\alpha_{t_k}(x)-x) \otimes 1) \eta_n\|_2^2= \tau(E_{M}((\alpha_{t_k}(x)-x)^*(\alpha_{t_k}(x)-x))) = \|\alpha_{t_k}(x)-x\|_2^2.
\end{equation}

As in the proof of Theorem $4.9$ in \cite{ozawapopa}, noticing that $L^2(M) \bar{\otimes} L^2(\bar{M})$ is an $M \bar{\otimes} \bar{M}$-module and since $\eta_n^k = \xi_n^k + \zeta_n^k$, equation \eqref{norm22} gives that for any $u \in \mathcal{G}$, and for any $k \geq k_0$,
\begin{align}\label{crucial}
\mathop{\Lim}_n \|[u \otimes \bar{u}, \zeta^k_n]\|_2 \nonumber
& \leq  \mathop{\Lim}_n \|[u \otimes \bar{u}, \eta_n^k]\|_2 \\ 
& \leq  \mathop{\Lim}_n \|(\alpha_{t_k} \otimes 1)([u \otimes \bar{u}, \eta_n])\|_2 + 2 \|u - \alpha_{t_k}(u) \|_2 \\
& =  2 \|u - \alpha_{t_k}(u)\|_2. \nonumber
\end{align}
Moreover, for any $x \in M$,
\begin{align}\label{normal}
\nonumber\| (x \otimes 1) \zeta^k_n \|_2 & =  \|(x \otimes 1) (e_M^\perp \otimes 1) \eta_n^k\|_2 \\
& =  \|(e_M^\perp \otimes 1) (x \otimes 1) \eta_n^k\|_2 \\
\nonumber& \leq  \| (x \otimes 1) \eta_n^k\|_2  = \|x\|_2.
\end{align}
\begin{claim}\label{claim1}
For any $k \geq k_0$, 
\begin{equation}\label{crucial2}
\mathop{\Lim}_n \|\zeta_n^k\|_2 \geq \delta.
\end{equation}
\end{claim}

\begin{proof}[Proof of Claim $\ref{claim1}$]
We prove the claim by contradiction. Exactly as in the proof of Theorem 4.9 in \cite{ozawapopa}, 
 we  have 
\begin{align*}
\mathop{\Lim}_n \|\eta_n^k - (e_M \alpha_{t_k}(u_k)  \otimes \bar{u}_k)\xi_n^k\|_2  
\leq & \mathop{\Lim}_n \|(e_M \o 1)\eta_n^k - (e_M \alpha_{t_k}(u_k)  \otimes \bar{u}_k)\eta_n^k\|_2 + 2\mathop{\Lim}_n \| \zeta_n^k\|_2 \\
 \leq & \mathop{\Lim}_n \|(\alpha_{t_k} \otimes 1)(\eta_n - (u_k \otimes \bar{u}_k)\eta_n)\|_2 + 2\delta = 2 \delta.
\end{align*}
Thus, we would get
\begin{align*}
\|(E_M \circ \alpha_{t_k})(u_k )\|_2  
&=  \mathop{\Lim}_n \|((E_M \circ \alpha_{t_k})(u_k) \otimes \bar{u}_k)\eta_n^k\|_2  \\
&\geq  \mathop{\Lim}_n \|(e_M \otimes 1) ((E_M \circ \alpha_{t_k})(u_k)  \otimes \bar{u}_k)\eta_n^k\|_2 
\\ &=  \mathop{\Lim}_n \|(e_M \alpha_{t_k}(u_k) \otimes \bar{u}_k) \xi_n^k\|_2 
\\ &\geq  \mathop{\Lim}_n \|\eta_n^k\|_2 -  2\delta \\
& =  1 -  2\delta > \sqrt{1 - c^2} ,
\end{align*}
which is a contradiction according to $(\ref{key})$. 
\end{proof}

We now use the techniques of the proof of Theorem A in \cite{ozawapopaII}. Define a state $\varphi^{ k}$ on $L^{\infty}([0,\infty),\mathbf{B}(\H) \cap \rho(M^{\op})')$, where $\rho(M^{\op})$ is the right $M$-action on $\H$, by
\begin{equation*}
\varphi^{ k}(x) = \mathop{\Lim}_n \frac{1}{\|\zeta_n^{ k}\|_2^2}\int_{0}^{t_k} \langle (x_{s} \otimes 1) \delta\phi_{t_k-s}\o1(\eta_n),\delta\phi_{t_k-s}\o1(\eta_n)\rangle,
\end{equation*}

Since by assumption and theorem \ref{corresp}, $||\alpha_{t_k}-E_{M}\alpha_{t_k}(x)||_2^2=\int_{0}^{t_k} \langle  \delta\phi_{t_k-s}(x),\delta\phi_{t_k-s}(x)\rangle$, we indeed get one on $x=id$.

\begin{claim}\label{claim2}
Let $a \in \Gamma$. Then one has 
\begin{equation*}
\mathop{\Lim}_k | \varphi^{k} (a x - x a) | = 0,
\end{equation*}
uniformly for $x \in \mathbf{B}(\mathcal{H}) \cap \rho(M^{\op})'$ (seen as constant functions in $L^{\infty}([0,\infty),\mathbf{B}(\H) \cap \rho(M^{\op})')$) with $\|x\|_\infty \leq 1$.
\end{claim}

\begin{proof}[Proof of Claim $\ref{claim2}$]
For every $x \in (\mathbf{B}(\mathcal{H}) \cap \rho(M^{\op})')_{+}$ and denote $\phi(u)$ the bounded function $\phi_{s}(u)$, one has
\begin{align*}
&\varphi^{ k}(\phi(u)^* x \phi(u))\geq \mathop{\Lim}_n \frac{1}{\|\zeta_n^{ k}\|_2^2} \\ &\int_{0}^{t_k} ds   \langle(x \otimes 1)(\phi_{s}(u) \otimes \bar{u}) \delta\phi_{t-s}\o1(\eta_n) ( \phi_{s}(u)\otimes \bar{u})^*, (\phi_{s}(u) \otimes \bar{u}) \delta\phi_{t-s}\o1(\eta_n) ( \phi_{s}(u)\otimes \bar{u})^*\rangle
\end{align*}

so that,
\begin{align*}
\varphi^{ k}(&\phi(u)^* x \phi(u))\geq \varphi^{ k}(x)-  \\ &-\mathop{\Lim}_n\frac{2}{\|\zeta_n^{ k}\|_2} \|x\|_\infty  \left(\int_{0}^{t_k} ds ||(\phi_s(u) \otimes \bar{u}) \delta\phi_{t-s}\o1(\eta_n) (\phi_s(u) \otimes \bar{u})^*- \delta\phi_{t-s}\o1(\eta_n) ||_2^2\right)^{1/2}
\end{align*}
Now, note \begin{align*}||(\phi_s(u) \otimes \bar{u})& \delta\phi_{t-s}\o1(\eta_n) (\phi_s(u) \otimes \bar{u})^*- \delta\phi_{t-s}\o1(\eta_n) ||_2^2\leq 
2||\delta\phi_{t-s}\o1(\eta_n) ||_2^2\\& -2\Re \langle(\phi_s(u) \otimes \bar{u}) \delta\phi_{t-s}\o1(\eta_n) (\phi_s(u)\otimes \bar{u})^*,\delta\phi_{t-s}\o1(\eta_n)\rangle 
\end{align*}

and thus,  using again the explicit structure of theorem \ref{corresp} (and (\ref{crucial}) for the last inequality)
\begin{align*}
\int_{0}^{t_k} ds &||(\phi_s(u) \otimes \bar{u}) \delta\phi_{t-s}\o1(\eta_n) (\phi_s(u) \otimes \bar{u})^*- \delta\phi_{t-s}\o1(\eta_n) ||_2^2 \\ &\leq 2\|\zeta_n^{ k}\|_2^2
-2\Re \langle(u \otimes \bar{u})\zeta_n^{ k}  (u \otimes \bar{u})^*,\zeta_n^{ k}\rangle \\& =
||(u \otimes \bar{u})\zeta_n^{ k}  (u \otimes \bar{u})^*-\zeta_n^{ k}||_2^2 
\\&\leq 
4||u-\alpha_{t_{k}}(u)||_{2}^{2}
\end{align*}
so that with $(\ref{crucial2})$ and using lemma 3.6 in \cite{ozawapopaII}, we finally get (for any $x\in \mathbf{B}(\mathcal{H}) \cap \rho(M^{\op})'$)\begin{equation*}
|\varphi^{ k}(\phi(u)^* x \phi(u)) - \varphi^{ k}(x)| \leq 
\frac{8}{\delta} \|x\|_\infty  ||u-\alpha_{t_{k}}(u)||_{2}
\end{equation*}
Since taking $x=id$ gives $\mathop{\Lim}_k |\varphi^{ k}(\phi(u)^* \phi(u) - 1 )| = 0$, this implies that 
\begin{equation}\label{intphi}
\mathop{\Lim}_k |\varphi^{ k}(\phi(a) x - x \phi(a))| = 0,
\end{equation}
for each $a \in \mbox{span }\mathcal{G}$ and uniformly for $x \in \mathbf{B}(\mathcal{H}) \cap \rho(M^{\op})'$ with $\|x\|_\infty \leq 1$. However, for any $a \in M$ (using again the explicit structure of theorem \ref{corresp} and complete positivity of $\phi_s$, and then again (\ref{normal}) and (\ref{crucial2})),
\begin{eqnarray*}
|\varphi^{ k}(x \phi(a))| & \leq & \mathop{\Lim}_n \frac{1}{\|\zeta_n^{ k}\|_2}  \|x\|_\infty| ||(a \otimes 1) \zeta_n^{ k}|| \\
& \leq & \frac{1}{\delta} \|x\|_\infty \| a\|_2 \end{eqnarray*}
and likewise for $|\varphi^{k}(\phi(a) x)|$. An application of Kaplansky's density theorem proves \ref{intphi} for $a\in M$.
Since for any $u\in Q\Gamma$ (set of finite linear combination with coefficients in Q), $\sup_{s\in[0,t_k]}||\phi_{s}(u)-u||\rightarrow 0$, we 
get a bit more than the result.\end{proof}

Thus if we define a state $\varphi$ by $\varphi(x)=\Lim \varphi^{ k}(x)$, $\varphi$ is a $\Gamma$-central state on $\mathbf{B}(\K)\subset\mathbf{B}(\H) \cap \rho(M^{\op})'$ 
. This gives a contradiction with being a non-amenable representation 

\end{proof}

We can deduce from this 
 several results:

\begin{corollary}\label{GeneDefRig}
Let $M=B \rtimes \Gamma$ as above, $M\subset \widetilde{M}$ von Neumann subalgebras, assume $\K$  is an non-amenable representation, $\H_{\infty}\subset L^{2}(\widetilde{M})\ominus L^{2}(M)$ produced as in Theorem \ref{corresp} from $\H$ coming from $\K$ as above. Take $P \subset M$ is a regular weakly compact subalgebra.  Assume $\H$ is a Hilbert $M$-$M$ bimodule which is compact relative to $B\subset M$ as an $P-P$ bimodule. 
Assume given a family of $*$-homomorphism $\alpha_{t}:M\rightarrow \widetilde{M}$ dilating a symmetric $\| \cdot \|_2$-strongly continuous semigroup $\phi_{t}$ on $M$ not converging uniformly on $(M)_1$ and with $ (\alpha_{t}-\phi_{t})(M)\subset \H_{\infty}$ in the way of Theorem \ref{corresp}. Also assume $\alpha_{t}$ is symmetric (at least so that $\tau(\alpha_{t}(X)Y\alpha_{t}(Z)T)=\tau(X\alpha_{t}(Y)Z\alpha_{t}(T))$.
 Then $P\preceq_M B$
\end{corollary}

\begin{proof}
From the previous results, the only remark necessary to get the corollary is to see that $\H_{\infty}$ satisfy the same compactness property as $\H$, and this is the remark after example 2.4 in \cite{P10} since $\phi_{s}$ is B-bimodular. 

\end{proof}

\begin{corollary}\label{SpecDefRig}
Let $M=B \rtimes \Gamma$  as above with $\Gamma$ a countable discrete group with $\beta_{1}^{(2)}(\Gamma)>0$ and such that M has c.m.a.p..
 Then any amenable regular subalgebra $P\preceq_M B$.
Especially $L(\Gamma)$ has no Cartan subalgebra if $\Gamma$ has c.m.a.p. and $\beta_{1}^{(2)}(\Gamma)>0$. Likewise, for the same $\Gamma$, and a profinite free p.m.p. ergodic action of $\Gamma$ on an standard probability space $X$, $M=L^{\infty}(X) \rtimes \Gamma$ has a unique Cartan subalgebra (up to unitary conjugacy).
\end{corollary}

\begin{remark}
In \cite{ozawapopaII}, analogous results were obtained assuming moreover given a proper cocycle in a non-amenable representation. If we can replace, as seen in the previous corollary, the unbounded cocycle valued in the regular representation by one valued in a non-amenable mixing representation, we have to assume sligthly more than them on the representation. Basically the mixingness (compactness) of the representation replaces the properness of the cocycle.
\end{remark}

\end{document}